\documentclass[11pt,reqno]{amsart}
\usepackage[margin=1in,letterpaper]{geometry}
\usepackage{graphicx}
\usepackage{amssymb}
\usepackage{amsthm}
\usepackage{mathrsfs}
\usepackage{stmaryrd}
\usepackage{accents} 
\usepackage{enumitem} 
\usepackage{subcaption}
\usepackage{microtype}
\usepackage{colonequals}

\usepackage[bookmarksopen,bookmarksdepth=2]{hyperref} 
\allowdisplaybreaks

\makeatletter\let\over\@@over\makeatother

\numberwithin{equation}{section}
\theoremstyle{plain} 
\newtheorem{theorem}{Theorem}[section] 
\newtheorem{proposition}[theorem]{Proposition} 
\newtheorem{corollary}[theorem]{Corollary}
\newtheorem{lemma}[theorem]{Lemma} 
\theoremstyle{remark}
\newtheorem{remark}[theorem]{Remark}
\theoremstyle{definition}
\newtheorem{definition}[theorem]{Definition}
\newtheorem{example}[theorem]{Example}

\newcommand{\be}{\begin{equation}}
\newcommand{\ee}{\end{equation}}%
\newcommand{\bse}{\begin{subequations}}
\newcommand{\ese}{\end{subequations}}

\newcommand{\dist}{\operatorname{dist}}
\newcommand{\realpart}{\operatorname{Re}}
\newcommand{\imagpart}{\operatorname{Im}}
\newcommand{\ar}{\operatorname{arg}}

\newcommand{\range}{\operatorname{rng}}
\newcommand{\kernel}{\operatorname{ker}}

\newcommand{\cod}{\operatorname{codim}}
\newcommand{\dimension}{\operatorname{dim}}
\newcommand{\id}{\operatorname{id}}

\newcommand{\LV}{\left|}
\newcommand{\RV}{\right|}
\newcommand{\LB}{\left[}
\newcommand{\RB}{\right]}
\newcommand{\LC}{\left(}
\newcommand{\RC}{\right)}

\newcommand{\p}{\partial}
\newcommand{\R}{\mathbb{R}} 
\newcommand{\placeholder}{\,\cdot\,}
\newcommand{\maps}{\colon}

\newcommand{\n}[2][]{#1\lVert #2 #1\rVert}
\newcommand{\abs}[2][]{#1\lvert #2 #1\rvert}
\newcommand{\dell}{\partial}

\newcommand{\loc}{{\mathrm{loc}} }

\newcommand\A{\mathscr A}    % kinematic condition operator
\newcommand\B{\mathscr B}    % Bernoulli condition operator
\newcommand\F{\mathscr F}    % the nonlinear problem
\newcommand\Vrho{\mathcal{V}}

\newcommand\Nconf{N_{\mathrm{c}}}
\newcommand\Nvel{N_{\mathrm{v}}}

\newcommand\Wspace{\mathscr W}
\newcommand\Xspace{\mathscr X}
\newcommand\Yspace{\mathscr Y}
\newcommand\fullPspace{\mathbb{P}}
\newcommand\Pspace{\mathcal P}

\newcommand{\cm}{{\mathscr C}}  
       
\newcommand{\km}{{\mathscr K}}            

\newcommand\fluidD{\mathscr{D}}

\newcommand{\Func}{\mathcal{F}}

\newcommand{\confD}{\mathcal{D}}

\newcommand{\fullparam}{\Lambda}
\newcommand{\param}{\lambda}
\newcommand{\proj}{P}

\newcommand{\imagimag}{{\mathrm{ii}} }        
\newcommand{\realimag}{{\mathrm{ri}} }    
\newcommand{\realreal}{{\mathrm{rr}} }        
\newcommand{\imagreal}{{\mathrm{ir}} }

\begin{document}

\title[Global vortex desingularization]{Desingularization and global continuation for hollow vortices}

\date{\today}

\author[R. M. Chen]{Robin Ming Chen}
\address{Department of Mathematics, University of Pittsburgh, Pittsburgh, PA 15260} 
\email{mingchen@pitt.edu}  

\author[S. Walsh]{Samuel Walsh}
\address{Department of Mathematics, University of Missouri, Columbia, MO 65211} 
\email{walshsa@missouri.edu} 

\author[M. H. Wheeler]{Miles H. Wheeler}
\address{Department of Mathematical Sciences, University of Bath, Bath BA2 7AY, United Kingdom}
\email{mw2319@bath.ac.uk}

\begin{abstract}
A hollow vortex is a region of constant pressure bounded by a vortex sheet and suspended inside a perfect fluid; it can therefore be interpreted as a spinning bubble of air in water.  This paper gives a general method for desingularizing non-degenerate steady point vortex configurations into collections of steady hollow vortices.  Our machinery simultaneously treats the translating, rotating, and stationary regimes.  Through global bifurcation theory, we further obtain maximal curves of solutions that continue until the onset of a singularity.  As specific examples, we give the first existence theory for co-rotating hollow vortex pairs and stationary hollow vortex tripoles, as well as a new construction of Pocklington's classical co-translating hollow vortex pairs.  All of these families extend into the non-perturbative regime, and we obtain a rather complete characterization of the limiting behavior along the global bifurcation curve.
\end{abstract}

% Please keep this here! Moving it above the abstract throws an error. --Miles
\maketitle

\setcounter{tocdepth}{1}
\tableofcontents

\section{Introduction}
\label{introduction section}

Consider an ideal fluid lying in an exterior planar domain $\fluidD$ whose boundary is the disjoint union of Jordan curves $\Gamma_1,\ldots, \Gamma_M$.  Assume the motion of the fluid obeys the incompressible Euler equations with the flow in the interior of $\fluidD$ irrotational.   A collection of \emph{hollow vortices} is a solution to this system such that each $\Gamma_k$ is a free boundary along which the pressure is constant and around which there is a nonzero circulation $\gamma_k$.  The complement of $\fluidD$ can then be understood physically as $M$ bubbles of vacuum suspended in the fluid and circumscribed by vortex sheets.

Hollow vortices are one of the classical models of localized vorticity and of basic importance to free streamline theory. While still far less understood than vortex patches, they have enjoyed considerable renewed interest among applied mathematicians in recent years.  A number of authors have obtained rigorous existence results for configurations such as a single rotating hollow vortex (H-state) \cite{crowdy2021hstates}, co-translating pairs \cite{pocklington1896,crowdy2013translating}, von K\'arm\'an vortex streets and vortex arrays \cite{baker1976structure,crowdy2011analytical,traizet2015hollow}, and hollow vortices in the presence of an ambient straining flow \cite{llewellyn2012structure} or shear \cite{zannetti2016hollow}. Interestingly, most of the solutions are given in terms of explicit conformal mappings, while \cite{traizet2015hollow} instead exploits connections to problems in minimal surfaces. Compressible analogues for many of these configurations have also been investigated \cite{moore1987compressible,ardalan1995steady,leppington2006field}.  %See also the numerical studies in \cite{nelson2021corotating}.

Our purpose in this paper is to develop a complete theory of desingularizing steady point vortex configurations into collections of hollow vortices.  Vortex desingularization has been widely applied to construct vortex patches and other related solutions~\cite{turkington1983steady,wan1988desingularizations,smets2010desingularization,cao2014regularization,hmidi2017pairs,cao2021pairs,hassainia2022multipole,garcia2022global,davila2022helices,cao2023sheets}. Many of these results are tailored to specific configurations, and almost all are local in that the solutions are only shown to exist in a neighborhood of the point vortex configuration. Here, we treat translating, rotating, and stationary vortices in a single framework.  Moreover, through analytic global bifurcation theory, we are able to continue the hollow vortex family to a maximal $C^0$ curve of solutions and characterize the singularities that can develop at its extreme. This is both the first local and the first global existence result for general configurations of hollow vortices.

\subsection{Governing equations}

Let $(x,y) \in \mathbb{R}^2$ denote points in the physical domain, which we may identify in the usual way with a complex number $z \colonequals  x + iy \in \mathbb{C}$.  A solution is called \emph{steady} provided it is independent of time when viewed in a reference frame that is either translating at a fixed velocity $c \in \mathbb{C}$ or rotating with a fixed angular velocity $\Omega \in \mathbb{R}$. Without loss of generality, we always assume either $c = 0$ or $\Omega = 0$.  Note that this also includes the stationary case $c = \Omega = 0$.  Thanks to the rotation invariance of the system, moreover, it is always possible to take $c \in \mathbb{R}$.

\begin{figure}
  \centering
  \includegraphics[scale=1.1]{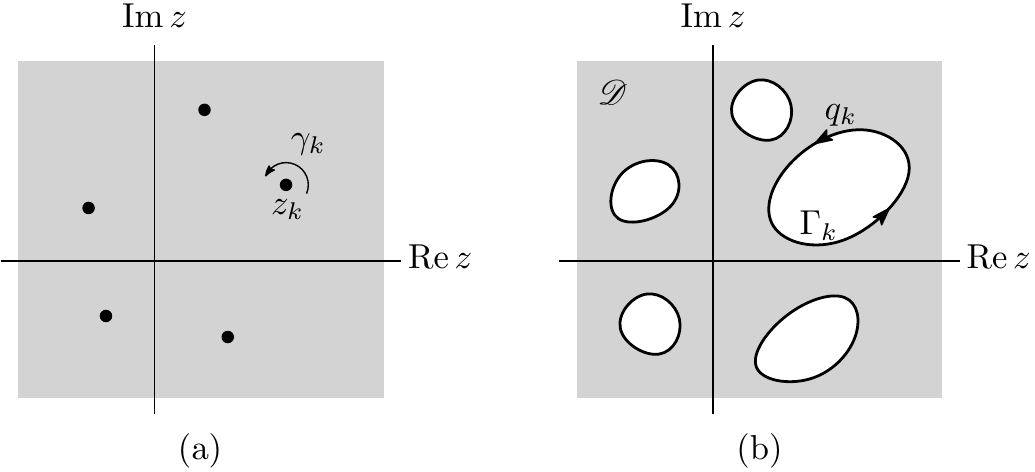}
  \caption{(a) A steady configuration of point vortices with locations $z_k$ and circulations $\gamma_k$. (b) A steady configuration of hollow vortices with boundaries $\Gamma_k$ along which the fluid velocity has constant magnitude $q_k$. The circulation remains $\gamma_k$.}
  \label{point vortex figure}
\end{figure}

Incompressibility and irrotationality imply that the (complexified) velocity field is antiholomorphic\footnote{If the velocity field is denoted $(\mathfrak{u}, \mathfrak{v}) \in \mathbb{R}^2$, some authors define the complex velocity by $\mathfrak{u} - i \mathfrak{v}$, as this is a holomorphic function.  To keep the complexification of $\mathbb{R}^2$ vectors consistent, however, in this paper we prefer to instead call $\mathfrak{u} + i \mathfrak{v}$ the complex velocity field, which is thus antiholomorphic.}  on $\fluidD$.  Thus we may introduce a  complex velocity potential 
\begin{subequations}
\label{intro hollow vortex problem}
\begin{equation}
\label{intro definition w}
  w = w(z) \colon {\fluidD} \to \mathbb{C} \qquad \textrm{holomorphic}
\end{equation}
such that $\overline{\partial_z w}$ is the velocity field, which must admit a continuous extension to $\partial\fluidD$. As $\fluidD$ is not simply connected, $w$ will be multivalued and the circulation around $\Gamma_k$ is given by
\begin{equation}
\label{circulation condition}
	\int_{\Gamma_k} \partial_z w \, dz = \gamma_k \in \R.
\end{equation}

Shifting to the appropriate moving reference frame, we obtain the relative complex velocity potential
\[
	W = W(z) \colonequals  w + i \frac{\Omega}{2} |z|^2 - c z. 
\]
A distinguishing feature of the rotating case $\Omega \neq 0$ is that the corresponding $W$ is not holomorphic.  On each vortex boundary, we impose the \emph{kinematic condition} that the relative velocity is purely tangential:
\begin{equation}
  \label{intro kinematic condition}
  \imagpart{W} = m_k \qquad \textrm{on } \Gamma_k,
\end{equation}
for some constant fluxes $(m_1,\ldots, m_M) \in \mathbb{R}^M$.  Finally, the \emph{dynamic condition} states that the pressure is continuous across the free boundaries.  For hollow vortices, each connected component of $\mathbb{C} \setminus \fluidD$ is a vacuum region at constant pressure, and so through Bernoulli's law, this becomes a requirement on the magnitude of the relative velocity field:
\begin{equation}
  \label{intro dynamic condition}
  |\partial_z W|^2 = q_k^2 \qquad \textrm{on } \Gamma_k,
\end{equation}
\end{subequations}
for some constant vector $q = (q_1, \ldots, q_M) \in \mathbb{R}^M$.  Note that here we can easily incorporate conservative body forces such as gravity by adding an appropriate function of $z$ to the right hand side of \eqref{intro dynamic condition}.

\subsection{Steady point vortices}
\label{point vortex section}

To study flows with extremely concentrated vorticity, one can imagine contracting each vortex boundary $\Gamma_k$ to a point $z_k$ while maintaining the circulation.  Intuitively, one expects the limiting system to be a collection of \emph{point vortices}, that is, to exhibit vorticity of the form $\sum_k \gamma_k \delta_{z_k}$,  where $\delta_{z_k}$ is the unit mass Dirac measure centered at $z_k$; see~Figure~\ref{point vortex figure}(a).  However, this does not constitute a weak solution of the Euler equations.  Near $z_k$, the conjugate of the velocity takes the form $(2\pi i)^{-1}\gamma_k (\placeholder - z_k)^{-1} + v_{\mathcal{H}}$ for a locally holomorphic function $v_{\mathcal{H}}$.  But the two-dimensional vorticity equation mandates that the vorticity is transported, which is not possible here as the vector field diverges as we approach $z_k$. 

Nonetheless, one can find an appropriate weakening of the Euler equations that does admit solutions of this form.  Physically, the vortices should not be able to self-advect, which suggests that the position of the center $z_k$ should be transported not by the full velocity, but only the non-singular part $\overline{v_{\mathcal{H}}}$.  This leads to the classical point vortex model of Helmholtz--Kirchhoff, which has been the subject of countless works in fluid mechanics and mathematics.  It can be stated rather elegantly as a system of complex ODEs.  Consider first the dynamical problem in the stationary frame.  At time $t \in \mathbb{R}$, let the vortex centers be located  $z_1(t), \ldots, z_M(t) \in \mathbb{C}$.  Then the positions of the vortex centers evolve according to 
\begin{equation}
  \label{helmholtz kirchhoff model}
  \partial_t \overline{z_k} = \sum_{j \neq k} \frac{\gamma_j}{2 \pi i} \frac{1}{z_j - z_k} \qquad \textrm{for } k = 1, \ldots, M.
\end{equation}
There are many ways to justify \eqref{helmholtz kirchhoff model} rigorously.  It arises, for instance, as the effective equation governing the limit of a sequence of solutions to the full Euler system where the support of the vorticity shrinks to the centers $\{ z_1, \ldots, z_M\}$; see, for example, \cite{marchioro1993vortices,marchioro1994book}.  The results we obtain in the present paper provide another justification in terms of hollow vortices.  
  
Continuing our earlier convention, we say a collection of point vortices is \emph{translating} provided there exists some wave speed $c \in \mathbb{R}$ such that $\partial_t \overline{z_k} = c$ for $k = 1, \ldots, M$. The point vortices are said to be \emph{rotating} about the origin with angular velocity $\Omega \in \mathbb{R}$ if $\partial_t \overline{z_k} = -i \Omega \overline{z_k}$ for $k = 1, \ldots, M.$ When combined with \eqref{helmholtz kirchhoff model}, both scenarios can be unified into the system of algebraic equations
\begin{equation}
  \label{definition steady point vortex}
  \begin{aligned}
    \mathcal{V}_k(\gamma_1,\, \ldots, \,  \gamma_M, \, z_1, \, \ldots, \, z_M, \, c, \, \Omega) 
 	 \colonequals  \sum_{j \neq k} \frac{\gamma_j}{2 \pi i} \frac{1}{z_j - z_k} - c + i \Omega \overline{z_k}
    = 0
 \end{aligned}
\end{equation}
for $k=1, \ldots, M$.  Observe that if $\Omega \neq 0$, then making the change of variables $z \mapsto z + i c/\Omega$ shows that \eqref{definition steady point vortex} corresponds to a rigidly rotating configuration centered at $z = -ic/\Omega$.  Thus we will assume without loss of generality that either $c = 0$ or $\Omega = 0$.  If both vanish, the configuration is said to be \emph{stationary}.  Finally, a collection of point vortices will be called \emph{steady} or in \emph{relative equilibrium} provided it satisfies \eqref{definition steady point vortex}, as this means it is time-independent in the moving frame. For more on point vortex equilibria we refer the reader to~\cite{newton2001nvortex,aref2003crystals,aref2007playground,newton2014postaref} and the references therein.

The idea of vortex desingularization is to reverse the limiting process described above: beginning at a steady point vortex configuration, we pry open each vortex center to obtain a steady hollow vortex solution to the full Euler equations. This is a local bifurcation result in the sense that the vortex boundary $\Gamma_k$ approximates a streamline of the point vortex system near $z_k$.  By means of global bifurcation theory, however, the family of hollow vortices can be extended far outside the perturbative regime.

\subsection{Conformal variables}\label{conformal variables section}

\begin{figure}
  \centering
  \includegraphics[scale=1.1]{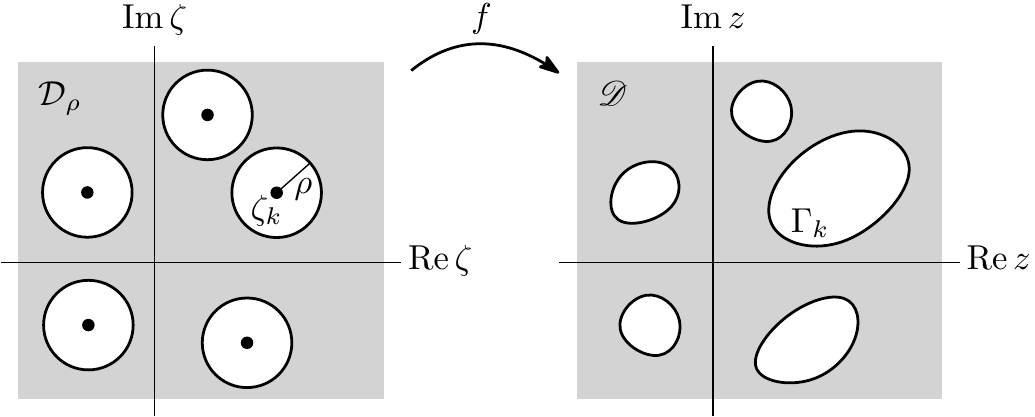}
  \caption{The conformal mapping $f$ maps the conformal domain $\mathcal D_\rho$ in the $\zeta$-plane  to the physical fluid domain $\mathscr D$ in the $z$-plane.}
  \label{conformal figure}
\end{figure}

One of the central challenges in vortex desingularization is to find a formulation of the problem that encapsulates both the Euler equations \eqref{intro hollow vortex problem} and the Helmhotz--Kirchhoff system \eqref{definition steady point vortex}.  In particular, it must allow the vortex boundaries to contract smoothly to points so that the (analytic) implicit function theorem can be applied.  

As a first step in that direction, our approach is to describe the fluid domain $\fluidD$ in the $z$-plane as the image under a conformal mapping of a geometrically simpler set.  Denote by $\zeta = \xi + i \eta$ the variables in this conformal plane, and for $\rho_1, \ldots, \rho_M > 0$ let
\begin{equation}
  \label{conformal domain definition} 
   \confD_{\rho_1,\ldots, \rho_M}(\zeta_1, \ldots, \zeta_M) \colonequals  \mathbb{C} \setminus \overline{B_{\rho_1}(\zeta_{1}) \cup \cdots \cup B_{\rho_M}(\zeta_{M})}.
\end{equation} 
Here, $\zeta_1, \ldots, \zeta_{M}$ are the pre-images of the vortex centers $z_1, \ldots, z_M$, respectively, and $\partial B_{\rho_k}(\zeta_k)$ is the pre-image of $\Gamma_k$.    Note that $\confD_{\rho_1, \ldots, \rho_M}$ falls in the broader class of \emph{circular domains}.  By an analogue of the Riemann mapping theorem due to Koebe \cite{koebe1918abhandlungen}, any $M$-connected domain can be mapped conformally to a circular domain, and  this description is unique up to normalization; see Section~\ref{circular domains section}.  To simplify the analysis, in our construction we will take all the radii to be the same, that is, $\rho_1 = \cdots = \rho_M \equalscolon  \rho$.  We then write $\confD_\rho$ for the corresponding domain in the $\zeta$-plane; see Figure~\ref{conformal figure}. Let
\[
	\fullparam = \left(\gamma_1,\, \ldots, \,  \gamma_M, \, \zeta_1, \, \ldots, \, \zeta_M, \, c, \, \Omega \right) \in \mathbb{R}^M \times \mathbb{C}^M \times \mathbb{R} \times \mathbb{R} \equalscolon  \fullPspace
\]
denote the complete set of parameters describing the point vortex system in this conformal domain.  Thus \eqref{definition steady point vortex} can be written as $\mathcal{V}(\fullparam) = 0$ where $\mathcal{V} = (\mathcal{V}_1,\ldots, \mathcal{V}_M) \colon \fullPspace \to \mathbb{C}^M$.

For the hollow vortex problem, we will work with two unknowns: 
\[
	f = f(\zeta) \in C^{\ell+\alpha}(\overline{\confD_\rho}), \qquad w = w(\zeta) \in C^{\ell+\alpha}(\overline{\confD_\rho}),
\]
where $\ell \geq 1$ and $\alpha \in (0,1)$ are fixed but arbitrary.  Here $f$ is the conformal mapping that defines the geometry of the hollow vortices, and we now view $w$ as the complex velocity potential in the $\zeta$-plane.  The fluid domain in the $z$-plane is then $\fluidD \colonequals  f(\confD_\rho)$, with $\Gamma_k = f(\partial B_\rho(\zeta_k))$, and the complex conjugate of the velocity field is 
\[
	\partial_z w= \frac{{\partial_\zeta w}}{{\partial_\zeta f}} \in C^{\ell-1+\alpha}(\overline{\confD_\rho}).
\]
Observe that in order for this reformulation to be valid, it is necessary that $f$ is both conformal and injective on $\overline{\confD_\rho}$.  Clearly, we must also have that $w$ is holomorphic on $\confD_\rho$ with $\partial_z w$ single-valued. We further require that $f$ is asymptotic to the identity map as $\zeta \to \infty$.  Under these conditions, the kinematic \eqref{intro kinematic condition} and dynamic conditions \eqref{intro dynamic condition} on $\Gamma_k$ can be restated in terms of $(w,f)$ on $ \partial B_{\rho}(\zeta_k)$; see Section~\ref{formulation section}.  

Lastly, from \eqref{definition steady point vortex} we can derive the leading-order asymptotics of the point vortex velocity field in a neighborhood of each $\zeta_k$.  In particular, one finds that the Bernoulli constants $q_k$ diverge like $1/\rho$ as $\rho \searrow 0$.  We therefore introduce normalized constants $Q = (Q_1, \ldots, Q_M) \in \mathbb{R}^M$ defined by 
\begin{equation}
  \label{q_k Q_k relation}
   \rho^2 q_k^2 - \frac{\gamma_k^2}{4 \pi^2 } \equalscolon  \rho Q_k, 
\end{equation}
which will satisfy $Q = O(\rho)$ in the limit $\rho \searrow 0$.

\subsection{Statement of results}
\label{statement section}

Let us now describe the main contributions of the paper.  We begin by demonstrating the vortex desingularization and global continuation procedure on the three important examples of steady point vortices shown in Figure~\ref{point vortex examples figure}.  

\begin{figure}
  \centering
  \includegraphics[scale=1.1]{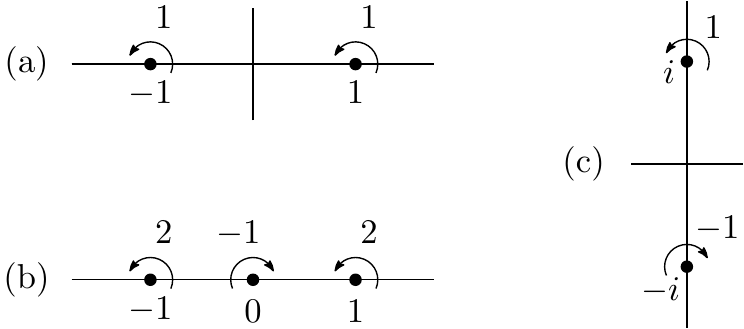}
  \caption{The basic point vortex equilibria which are desingularized in Theorems~\ref{intro rotating theorem}--\ref{intro pocklington theorem} below. (a) A symmetric pair of same signed vortices, steadily rotating about the origin with angular velocity $\Omega=\frac 1{4\pi}$. (b) A symmetric tripole, which is stationary. (c) A symmetric pair of opposite signed vortices, which steadily translate along the $x$-axis with velocity $c=\frac 1{4\pi}$.}
  \label{point vortex examples figure}
\end{figure}

First, consider a pair of point vortices with same-signed vortex strength, which will result in a rotating configuration. One such solution is
\[
	\fullparam_{\textup{r}}^0 \colonequals  \Big(\gamma_1 = 1, \, \gamma_2=1,\,  \zeta_1 = 1, \, \zeta_2 = -1, \, c = 0, \, \Omega = \frac{1}{4\pi}\Big),
\]
as can be easily verified from \eqref{definition steady point vortex}.  Our first theorem states that there exists a smooth curve of co-rotating hollow vortices that emanates from $\fullparam_{\textup{r}}^0$.  For reasons that will be elaborated upon later, only a subset of the components of $\fullparam$ will vary along the hollow vortex curve.  In the case of the co-rotating pair, we can exploit symmetries to keep all but $\Omega$ fixed. Our techniques can also treat asymmetric pairs for which $\gamma_1 \neq \gamma_2$, with a slightly larger set of varying parameters.

\begin{figure}
  \centering
  \includegraphics[scale=1.1]{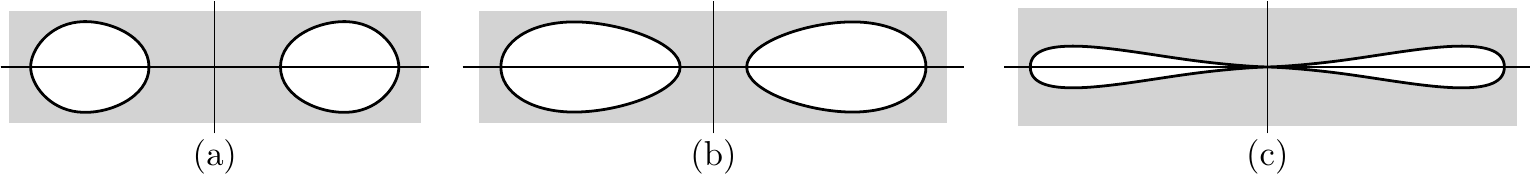}
  \caption{Scenario for the rotating pair predicted numerically in \cite{nelson2021corotating}, which involves both conformal and velocity degeneration.} 
  \label{pair pinching figure}
\end{figure}

As we follow the hollow vortex family to its extreme, one or more limiting singularities develop.  The first possibility is that the conformal description of the domain degenerates or the vortex boundaries (self-)intersect.  This \emph{conformal degeneracy} alternative is captured by the blowup of the quantity 
\begin{equation}
  \label{definition conformal blowup norm}
  \Nconf(f) \colonequals  \sup_{\partial \confD_{\rho}}   | \partial_\zeta f|  + \sup_{\substack{ \zeta, \zeta^\prime \in \partial\confD_{\rho} \\ \zeta \neq \zeta^\prime}} \frac{|\zeta - \zeta^\prime|}{|f(\zeta)-f(\zeta^\prime)|} + \frac{1}{\displaystyle \min_{j \neq k} \dist{(\Gamma_j, \Gamma_k)}}.
\end{equation}
The second term above is a chord-arc measure of the vortex boundaries, while the third gauges the distance between the vortices in the physical domain.  
Blowup of the first term in \eqref{definition conformal blowup norm} could indicate a loss of boundary regularity, for instance the development of an inward-pointing corner. Because the conformal description of $\fluidD$ is unique modulo a certain class of automorphisms, the breakdown we observe is not specific to the choice of coordinates but an intrinsic feature of the solutions.

The second alternative is \emph{velocity field degeneration}, wherein the relative velocity along the boundary becomes unbounded in magnitude or limits to stagnation.  This scenario is equivalent to the blowup of the quantity
\begin{equation}
  \label{definition velocity blowup norm}
  \Nvel(f,w) \colonequals  \sup_{\partial \confD_{\rho}} \left(  |U| + \frac{1}{|U|} \right),
\end{equation}  
where $U \colonequals  \partial_z W = \partial_\zeta W/\partial_\zeta f$ is the conjugate of the relative velocity field in the conformal domain.  Notice that due to the dynamic boundary condition \eqref{intro dynamic condition}, if stagnation occurs at a point on $\Gamma_k$, then it occurs at \emph{every} point on $\Gamma_k$.  However, there is no prohibition against having stagnation points in the interior, and indeed, these are expected to be present.  

For the co-rotating pair, there is only one other possibility: blowup of the (excess) angular momentum, which is given by
\begin{equation}
\label{definition angular momentum}
	L(f,w) \colonequals  \imagpart{\int_{\fluidD} z \partial_z \left(  w - w^0 \right) \, dz} = \imagpart{ \int_{\confD_\rho} f \overline{\partial_\zeta f} \partial_\zeta \left(  w - w^0 \right)  \, d\zeta }, 
\end{equation}
where $w^0$ is the potential for the corresponding point vortex configuration.  Our first result is then as follows.

\begin{theorem}[Rotating pair]
\label{intro rotating theorem}
There exists a family $\cm_{\textup{r}}$ of co-rotating hollow vortex pairs admitting the $C^0$ parameterization
\[
	\cm_{\textup{r}} = \{ \left(f(s), w(s), Q(s), \Omega(s), \rho(s) \right) : s \in [0,\infty) \},
\]
where the remaining components of $\fullparam$ are fixed to their values in $\fullparam_{\textup{r}}^0$.  The curve $\cm_{\textup{r}}$ exhibits the following properties.
\begin{enumerate}[label=\rm(\alph*)]
	\item The curve bifurcates from the configuration
		\[
			f(0) = \id, \quad w(0) = \frac{1}{2\pi i} \log{\Big( (\placeholder-1)(\placeholder+1) \Big)}, \quad Q(0) = 0, \quad \Omega(0) = \frac{1}{4\pi}, \quad \rho(0) = 0,
		\]
		which corresponds to the co-rotating pair of point vortices in Figure~\ref{point vortex examples figure}(a).
	\item  The solutions on $\cm_{\textup{r}}$ are hollow vortices in that $\rho(s) > 0$ for $s \neq 0$.  Moreover, in the limit $s \to \infty$, one of the following alternatives occurs.
		\begin{enumerate}[label=\rm(\roman*)]
			\item \textup{(Conformal degeneracy)} The unique conformal description of the fluid domain degenerates in that
        \begin{equation}
          \label{conformal degeneracy alternative}
          \limsup_{s \to \infty} \Nconf(f(s)) = \infty ,
        \end{equation}
			where $\Nconf$ is the quantity defined by \eqref{definition conformal blowup norm}.
			\item \textup{(Velocity degeneracy)} The fluid velocity degenerates along a vortex boundary in that
        \begin{equation}
          \label{velocity degeneracy alternative}
           \limsup_{s \to \infty} \Nvel(f(s),w(s))  = \infty,
        \end{equation}
			where $\Nvel$ is the quantity defined in \eqref{definition velocity blowup norm}.
		\item \textup{(Angular momentum blowup)} The excess angular momentum \eqref{definition angular momentum} is unbounded in that
      \begin{equation}
        \label{angular momentum alternative}
        \limsup_{s \to \infty} L(f(s), w(s))  =  \infty \quad \textrm{and} \quad \limsup_{s \to \infty} |\Omega(s)| = \infty.
      \end{equation}
	\end{enumerate}
	\item \label{intro rotating symmetry part} For each solution on $\cm_{\textup{r}}$, the fluid domain is even with respect to the real and imaginary axes, the relative horizontal velocity is even over the imaginary axis and odd over the real axis, while the relative vertical velocity is odd with respect to both real and the imaginary axes.
	\item At each point on $\cm_{\textup{r}}$, the curve admits a local real-analytic reparameterization.
\end{enumerate}
\end{theorem}

The above theorem represents the first rigorous construction of large, co-rotating hollow vortices.  Numerical work by Nelson, Krishnamurthy, and Crowdy \cite{nelson2021corotating} suggests that in the limit, the two vortex cores pinch off and merge: the vortex boundaries approach one another, while developing a cusp at the incipient points of contact which also tends to stagnation. A sequence of vortex boundaries are sketched in Figure~\ref{pair pinching figure}. This would correspond to both conformal and velocity field degeneration occurring simultaneously. 

It bears mentioning that some careful analysis is needed to winnow the alternatives down to just \eqref{conformal degeneracy alternative}, \eqref{velocity degeneracy alternative}, and \eqref{angular momentum alternative}, as a priori there are many other ways for degeneracy to manifest.  For example, we might imagine that in the limit the pre-images of the vortex boundaries come into contact in the conformal domain. However, the first two terms in $\Nconf(f)$ and the maximum modulus principle allow us to control $|\partial_\zeta f|$ on $\overline{\confD_\rho}$ from above and below, and thus the distance between vortices in the physical domain and their pre-images in the conformal domain are comparable.  Likewise, the argument principle and a continuity argument allow us to infer injectivity of $f(s)$ in the interior of $\confD_{\rho(s)}$ from its injectivity on the boundary.  A nontrivial uniform regularity result is also needed to reduce blowup to unboundedness in $L^\infty$, rather than in higher regularity H\"older or Sobolev norms.  The fact that $|\Omega(s)| \to \infty$ occurs in conjunction with the blowup of the excess angular momentum is a consequence of a novel identity for hollow vortices; see Proposition~\ref{momentum identity proposition}.

Next consider the stationary point vortex tripole 
\[ 
	\fullparam_{\textup{s}}^0 \colonequals  \Big(\gamma_1 = 2,\,  \gamma_2 = -1, \,  \gamma_3=2, \, \zeta_1 = 1,\,  \zeta_2=0, \, \zeta_3=-1, \, c=0,\, \Omega=0\Big)
\]
depicted in Figure~\ref{point vortex examples figure}(b).  Here, we take the varying component of $\fullparam$ to be just the circulation $\gamma_2$ around the central vortex.  A second application of the vortex desginularization and continuation machinery then yields the following result.

\begin{theorem}[Stationary tripole]
\label{intro tripole theorem}
There exists a family $\cm_{\textup{s}}$ of stationary hollow vortex tripoles admitting the $C^0$ parameterization
\[
	\cm_{\textup{s}} = \{ \left(f(s), w(s), Q(s), \gamma_2(s), \rho(s) \right) : s \in [0,\infty) \},
\]
with the remaining components of $\fullparam_{\textup{s}}$ fixed to their values in $\fullparam_{\textup{s}}^0$.  The curve $\cm_{\textup{s}}$ exhibits the following properties.
\begin{enumerate}[label=\rm(\alph*)]
	\item The curve bifurcates from the configuration
		\[
			f(0) = \id, \quad w(0) = \frac{1}{2\pi i} \log{\left(\frac{(\placeholder - 1)^2 (\placeholder + 1)^2}{\placeholder } \right)}, \quad Q(0) = 0, \quad \gamma_2(0) = -1, \quad \rho(0) = 0,
		\]
		which corresponds to the stationary point vortex tripole in Figure~\ref{point vortex examples figure}(b).  
	\item For all $s > 0$, we have $\rho(s) > 0$, and in the limit $s \to \infty$ either conformal degeneracy \eqref{conformal degeneracy alternative} or velocity degeneracy \eqref{velocity degeneracy alternative} occurs:
		\[
			\limsup_{s \to \infty} \Big( \Nconf(f(s))+\Nvel(f(s),w(s)) \Big) =  \infty.
		\]
	\item \label{intro stationary symmetry part} For each solution on $\cm_{\textup{s}}$, the fluid domain is even with respect to the real and imaginary axes, the relative horizontal velocity is even over the imaginary axis and odd over the real axis, while the relative vertical velocity is odd over the imaginary axis and even over the real axis.
	\item At each point on $\cm_{\textup{t}}$, the curve admits a local real-analytic reparameterization.
\end{enumerate}
\end{theorem}

As a final concrete example, we investigate the classical case of a translating pair of point vortices:
\[
	\fullparam_{\textup{t}}^0 \colonequals \Big(\gamma_1 = 1,\,  \gamma_2 = -1, \, \zeta_1 = i, \, \zeta_2 = -i, \, c = \frac{1}{4\pi}, \, \Omega = 0  \Big).
\]
Explicit families translating hollow vortex pairs were first discovered by Pocklington \cite{pocklington1896} and then later revisited by Crowdy, Llewellyn Smith, and Freilich \cite{crowdy2013translating}.  For our family, the varying parameter will be solely the wave speed $c$. 

\begin{theorem}[Pocklington vortices]
\label{intro pocklington theorem}
There exists a family $\cm_{\textup{t}}$ of translating pairs of hollow vortices admitting the $C^0$ parameterization
\[
	\cm_{\textup{t}} = \{ \left(f(s), w(s), Q(s), c(s), \rho(s) \right) : s \in [0,\infty) \},
\]
with the remainder of the components of $\fullparam_{\textup{t}}$ fixed to their values in $\fullparam_{\textup{t}}^0$.  The curve $\cm_{\textup{t}}$ exhibits the following properties.
\begin{enumerate}[label=\rm(\alph*)]
	\item The curve bifurcates from the configuration
		\[
			f(0) = \id, \quad w(0) = \frac{1}{2\pi i} \log{\left(\frac{\placeholder - i}{\placeholder + i} \right)}, \quad Q(0) = 0, \quad c(0) = \frac{1}{4\pi}, \quad \rho(0) = 0,
		\]
		which corresponds to the co-translating point vortex pair in Figure~\ref{point vortex examples figure}(c).
	\item For all $s > 0$, we have $\rho(s) > 0$, and in the limit $s \to \infty$ either conformal degeneracy \eqref{conformal degeneracy alternative} or velocity degeneracy \eqref{velocity degeneracy alternative} occurs:
		\[
			\limsup_{s \to \infty} \Big( \Nconf(f(s))+\Nvel(f(s),w(s)) \Big) =  \infty.
		\]
	\item \label{intro pocklington symmetry part} For each solution on $\cm_{\textup{t}}$, the fluid domain is even with respect to the real and imaginary axes, the relative horizontal velocity is even over the imaginary axis and odd over the real axis, while the relative vertical velocity is odd over the imaginary axis and odd over the real axis.

	\item At each point on $\cm_{\textup{t}}$, the curve admits a local real-analytic reparameterization.
\end{enumerate}
\end{theorem}

Theorems~\ref{intro rotating theorem}, \ref{intro tripole theorem}, and \ref{intro pocklington theorem} are specific applications of a general machinery that allows \emph{any} non-degenerate steady point vortex configuration to be desingularized and continued to a global $C^0$ curve of hollow vortices.  We will postpone the rigorous definition of non-degeneracy until Section~\ref{local bifurcation point vortex section}.  For now, it suffices to say that it is a condition on the rank of the Jacobian of the steady point vortex system \eqref{definition steady point vortex} that ensures all nearby solutions can be parameterized by some subcollection $\param^\prime$ of the parameters.  An assumption of this type is universal in vortex desingularization arguments, though often made implicitly.  When it holds, we may perform a generalized form of the decomposition seen in the examples, writing the full set of parameters as $\fullparam = (\param, \param^\prime)$, where $\param$ will be allowed to vary along the hollow vortex solution curve while $\param^\prime$ will remain fixed.  

Our first theorem in the general setting is the following vortex desingularization result.  

\begin{theorem}[Vortex desingularization]
\label{intro local desingularization theorem}
	Let $\fullparam_0 = (\param_0, \param_0^\prime)$ be a non-degenerate steady point vortex configuration.  There exists $\rho_1 > 0$ and a family $\cm_\loc$ of hollow vortices that admits the real-analytic parameterization 
\[
	\cm_{\loc} = \left\{  (f^\rho, w^\rho, Q^\rho, \param^\rho, \rho) : \quad \rho \in [0, \rho_1) \right\},
\]
and bifurcates from the point vortex configuration $\fullparam_0$ in that 
\[
		f^0 = \id, \quad w^0 = \sum_{k} \frac{\gamma_k}{2\pi i} \log{(\placeholder - \zeta_k)}, \quad Q^0 = 0, \quad (\param^0,\param_0^\prime) = \fullparam_0.
\]
\end{theorem}

The curve $\cm_\loc$ is obtained by using the (local) implicit function theorem which also gives detailed descriptions of the hollow vortex geometry and velocity field. Specifically, we find that the streamlines in a neighborhood of a vortex are asymptotically circular.  The presence of the other vortices is felt at the next order of development as an effective ambient straining flow.   That is, the conformal map and the complex conjugate of the velocity field take the form
\begin{equation}
\label{informal asymptotics}
  \begin{aligned}
    f^\rho(\zeta) & = \zeta + \rho^4 \frac{8  \pi i \overline{S_k}}{ \gamma_k} \frac{1}{\zeta - \zeta_k} + O(\rho^5) \\
    \partial_\zeta w^\rho(\zeta) & = \frac{\gamma_k}{2\pi i} \frac{1}{\zeta-\zeta_k} + 2S_k (\zeta-\zeta_k) - 2 \rho^4 \overline{S_k} \frac{1}{(\zeta-\zeta_k)^3} +O(\rho^5)
  \end{aligned}
\end{equation}
on the pre-image of the $k$-th vortex boundary $\partial B_\rho(\zeta_k)$, where
\[
	S_k \colonequals  -\frac{1}{2} \sum_{j \neq k} \frac{\gamma_j}{2\pi i} \frac{1}{(\zeta_j - \zeta_k)^2}.
\]
The leading-order terms in \eqref{informal asymptotics} agree with the explicit formulas of Llewellyn Smith and Crowdy \cite{llewellyn2012structure} for a single hollow vortex suspended in a velocity field that limits to $2S_k z$ as $z \to \infty$.  For the precise statement of Theorem~\ref{intro local desingularization theorem}, see Theorem~\ref{hollow vortex local bifurcation theorem}.

A major strength of Theorem~\ref{intro local desingularization theorem} is its generality: we are able treat in a unified way translating, stationary, and rotating vortices.  Traizet \cite{traizet2015hollow}, for example, constructs small solutions using a remarkable correspondence between the hollow vortex problem and certain minimal surface equations \cite{traizet2014classification}.  This equivalence, however, only applies to the stationary case and further supposes that the Bernoulli constants are all the same. Several authors have discovered explicit families of hollow vortices. This includes the classical, 19th century work of Pocklington \cite{pocklington1896} where relevant conformal mappings are expressed via Jacobi elliptic functions as well as more modern work based on Schottky--Klein prime functions \cite{crowdy2013translating,crowdy2014hollow,crowdy2021hstates}.  Not all hollow vortex configurations are expected to have explicit formulas of this type, however, even relatively simple configurations such as a rotating pair~\cite{nelson2021corotating}. Nevertheless, through Theorem~\ref{intro local desingularization theorem} (and Theorem~\ref{intro global continuation theorem} below), we are able to infer the rigorous existence of (global) curves of hollow vortices as soon as the non-degeneracy of the point vortex equilibrium is known. Our argument is quite robust, and in particular can be readily adapted to periodic configurations such as vortex streets or arrays, or to allow configurations consisting of both point vortices and hollow vortices. Of course the solutions we find are far less explicit; the local curves are obtained using a fixed point argument, while the global curves are obtained using non-constructive global bifurcation theory.

There is also a large and rapidly growing literature on the closely related question of desingularizing point vortices into vortex patches.  These are solutions to the Euler equations in the entire plane, where the vortex cores are modeled as regions of constant vorticity rather than vacuum.  In some respects, this situation is analytically simpler than the hollow vortex system: the fluid domain is simply connected and the vortex boundaries are streamlines but not lines of constant pressure. The latter fact means that while the kinematic condition \eqref{intro kinematic condition} is imposed on $\Gamma_k$, one is relieved of the need to satisfy the dynamic condition \eqref{intro dynamic condition}, which is fully nonlinear. 
On the other hand, for vortex patches one must study the flow both inside and outside the vortex regions, which complicates the use of conformal mappings. 
The earliest rigorous results on desingularization into steady vortex patches are due to Turkington, who used variational methods~\cite{turkington1983steady,turkington1985corotating}; also see \cite{keady1985streets,burton1987rings,smets2010desingularization,cao2021pairs} and the references therein. More recently, Hmidi and Mateu~\cite{hmidi2017pairs} developed a novel desingularization of the contour dynamics equations for symmetric pairs, allowing for the direct application of the implicit function theorem. This method has since been extended to other specific configurations~\cite{garcia2020streets,garcia2021choreography,hassainia2021asymmetric}, and very recently to general configurations~\cite{hassainia2022multipole}. It is worth mentioning that many of these works also treat other models such as the generalized surface quasi-geostrophic equations.

Our last result analytically continues the local curve $\cm_\loc$ to obtain a maximal curve $\cm$ that may contain solutions drastically different from the initial point vortex configuration.  In the limit along $\cm$, it is once again possible that either conformal degeneracy \eqref{conformal degeneracy alternative} or velocity degeneracy \eqref{velocity degeneracy alternative} occur.  However, in general there is a third alternative: the varying parameters may be unbounded.

\begin{theorem}[Global continuation]
\label{intro global continuation theorem}
Let $\fullparam_0 = (\param_0, \param_0^\prime)$ be a non-degenerate steady point vortex configuration and $\cm_\loc$ the local curve of hollow vortices furnished by Theorem~\ref{intro local desingularization theorem}.  Assume that at least one of the circulations is part of $\param_0^\prime$.  There exists a curve $\cm$ of hollow vortices admitting the $C^0$ parameterization
  \begin{equation}
    \label{global curve parameterization}
    \cm = \{ (f(s), w(s), Q(s), \param(s), \rho(s)) : s \in [0,\infty) \}
  \end{equation}
with $\cm_\loc \subset \cm$.  The curve $\cm$ is locally real analytic and exhibits the limiting behavior
  \begin{equation}
    \label{general continuation blowup alternative}
    \limsup_{s \to \infty} \Big( \Nconf(s) + \Nvel(s) + |\lambda(s)| \Big) = \infty.
  \end{equation}
\end{theorem}

For vortex patches, by contrast, the only available global bifurcation results treat either a single rotating patch~\cite{hassainia2020global} --- for which no desingularization is necessary --- or translating and rotating pairs~\cite{garcia2022global}. By combining some of the ideas in the present paper with \cite{hassainia2022multipole} and \cite{garcia2022global}, however, one might hope to prove a vortex-patch analogue of Theorem~\ref{intro global continuation theorem}. We recall that there are known explicit hollow vortex solutions for a single rotating vortex~\cite{crowdy2021hstates} and a translating pair~\cite{pocklington1896,crowdy2013translating}, but \emph{not} for a rotating pair~\cite{nelson2021corotating}.

\subsection{Plan of the paper}
\label{outline section}

Let us now outline the structure of the argument.  The conformal domain $\confD_\rho$ provides a canonical representation for the fluid domain $\fluidD$, but it too is varying with the parameters $\rho$ and $\fullparam$ and hence not a suitable functional analytic setting.  Instead, we consider the traces of the unknowns on the $M$ boundary components of $\confD_\rho$.  Through layer potentials, $f$ and $w$ can be represented via real-valued densities $\mu, \nu \in C^{\ell+\alpha}(\mathbb{T})^M$ defined on $\mathbb{T}$, the unit circle in the complex plane.  This change of unknowns fixes the domain at the price of rendering the problem highly nonlocal. The corresponding system can then be written as the abstract operator equation
\[
	\F(\mu,\nu,Q,\param, \rho) = 0,
\]
where $\F$ is a real-analytic mapping found by expressing the kinematic and dynamic conditions in terms of the densities $\mu$ and $\nu$.  This novel formulation captures both the point vortex system when $\rho = 0$ and the hollow vortex problem when $\rho > 0$.  Moreover, its linearization at a steady point vortex configuration has a block diagonal structure that causes the kernel and cokernel of the linearized hollow vortex and point vortex systems to have the same dimensions.  

As in~\cite{hassainia2022multipole}, non-degeneracy becomes crucial at this stage. Recall that \eqref{definition steady point vortex} can be written $\mathcal{V}(\fullparam) = 0$.  It is well known that the components of $\mathcal{V}$ satisfy two linear equations \eqref{pv identities}, and consequently its Jacobian $D_\fullparam \mathcal{V}$ will never be full rank; see the discussion in Section~\ref{local bifurcation point vortex section}. A non-degenerate steady configuration $\fullparam_0$ is one for which this lack of surjectivity can be attributed precisely to the identities \eqref{pv identities}.  In that case, we observe that one can systematically split the parameters $\fullparam$ into $(\param,\param^\prime)$, so that a modified implicit function argument can be applied to solve \eqref{definition steady point vortex} for $\param = \param(\param^\prime)$ locally. This procedure is detailed in Sections~\ref{abstract local bifurcation section} and \ref{local bifurcation point vortex section}.

To apply the same reasoning to the hollow vortex problem requires finding the identities responsible for the deficiency of the linearized operator $D_{(\mu,\nu,Q,\param)} \F$ at the steady point vortex; this is the subject of Section~\ref{hollow vortex identities section}.  The time-dependent point vortex system \eqref{helmholtz kirchhoff model} is Hamiltonian, and the identities for $\mathcal{V}$ can be understood as generated by its  invariance with respect to translation and rotation in space.   We therefore derive a formal Hamiltonian formulation for the time-dependent hollow vortex problem, and carry out an analogous (though far more complicated) procedure.  The result is a set of identities that directly generalize those for the point vortex system.  To the best of our knowledge, these have not been previously reported and are thus of independent interest.

In Section~\ref{local section}, we are then able to prove Theorem~\ref{intro local desingularization theorem} through an implicit function theorem argument.  Key to the analysis is the block diagonal structure of the linearized operator, which allows the dynamic and kinematic conditions on each vortex boundary to be decoupled at leading order.  For the co-rotating pair, stationary tripole, and translating pair, we can take advantage of additional symmetries to reduce the number of parameters that vary along the resulting bifurcation curve.  This argument is given in Section~\ref{desingularization with symmetry section}.

Finally, in Section~\ref{global section}, we extend the local curve $\cm_\loc$ to the non-perturbative regime using techniques from analytic global bifurcation theory.  Some background material on this subject is collected in Section~\ref{global ift section} for the reader's convenience.  Applying the general machinery yields a global curve $\cm$, but gives a considerably larger set of alternatives for the limiting behavior along it than what is claimed in Theorem~\ref{intro global continuation theorem}. In particular, it includes the possibilities that $\mu$ or $\nu$ is unbounded in $C^{\ell+\alpha}(\mathbb{T})^M$; the parameters are unbounded; the zero-set of $\F$ is not locally pre-compact; and that the formulation degenerates in that $f$ loses injectivity. We therefore establish linear and nonlinear a priori estimates for the relevant class of Riemann--Hilbert problems that imply local properness of $\F$ and allow us to derive uniform regularity bounds controlling the $C^{\ell+\alpha}$ norm of the densities in terms of $\Nconf$ and $\Nvel$.  Shifting between the local and nonlocal formulations at opportune times substantially simplifies this analysis.   Through these estimates, we ultimately confirm that either conformal degeneracy \eqref{conformal degeneracy alternative}, velocity degeneracy \eqref{velocity degeneracy alternative}, or blowup of $|\lambda|$ must occur, proving Theorem~\ref{intro global continuation theorem}.  The general theory is then applied to the construct global families of co-rotating hollow vortex pairs, stationary hollow vortex tripoles, and Pocklington vortices.  By exploiting concrete symmetries for these cases, the limiting behavior along the solution curves can be refined down to the alternatives enumerated in Theorems~\ref{intro rotating theorem}, \ref{intro tripole theorem}, and \ref{intro pocklington theorem}.  A priori estimates on the wave speed, excess angular momentum, and circulations play a key role for this part of the argument.  Several of these bounds appear to be new and may have broader applications.

\subsection*{Notation}

Here we lay out some notational conventions for the remainder of the paper.  Let $\mathbb{T}$ be the unit circle in the complex plane.  For $\ell \geq 0$ and $\alpha \in (0,1)$, we denote by $C^{\ell+\alpha}(\mathbb{T})$ the space of real-valued H\"older continuous functions of order $\ell$, exponent $\alpha$, and having domain $\mathbb{T}$.  

Every $\varphi \in C^{\ell+\alpha}(\mathbb{T})$ admits a unique power series representation 
\[
	\varphi(\tau) = \sum_{m \in \mathbb{Z}} \widehat \varphi_m \tau^m, \qquad \textrm{where} \quad \widehat\varphi_m \colonequals  \frac{1}{2\pi} \int_{\mathbb{T}} \varphi(\tau) \overline{\tau^m} \, d\theta.
\] 
Here and elsewhere, we write $\tau = e^{i\theta}$ when parameterizing $\mathbb{T}$ by arc length.  Note that because these are real-valued functions, the coefficients must obey $\widehat\varphi_{-m} = \overline{\widehat\varphi_m}$.  For $m \geq 0$, let $\proj_m$ denote the projection
\begin{equation}
  \label{definition projection operator}
  \proj_m \colon C^{\ell+\alpha}(\mathbb{T}) \to C^{\ell+\alpha}(\mathbb{T}), \qquad 
  (P_m\varphi)(\tau) \colonequals  
    \left\{ \begin{aligned}
      \widehat\varphi_m \tau^m + \widehat\varphi_{-m} \tau^{-m} & \qquad  \textrm{if }m \neq 0 \\
       \widehat\varphi_0 & \qquad  \textrm{if } m = 0.
    \end{aligned}
    \right.
\end{equation}
and set
\[ 
	\proj_{\leq m} \colonequals  \proj_0 + \cdots + \proj_m, \qquad \proj_{> m} \colonequals  1-\proj_{\leq m}.
\]
We will often work with the space
\[
	\mathring{C}^{\ell+\alpha}(\mathbb{T}) \colonequals  \proj_{> 0} C^{\ell+\alpha}(\mathbb{T}),
\]
which corresponds to elements of $C^{\ell+\alpha}(\mathbb{T})$ normalized to have mean $0$.

Finally, throughout the paper, we will use the Wirtinger derivative operators
\[
	\partial_z \colonequals  \frac{1}{2} \left( \partial_x - i \partial_y \right), \qquad
	\partial_{\overline z} \colonequals  \frac{1}{2} \left( \partial_x + i \partial_y \right).
\]
In particular, a function $f = f(z,\bar z)$ is holomorphic precisely when $\partial_{\bar z} f = 0$ and antiholomorphic provided $\partial_z f = 0$.  When there is no risk of confusion, these may also be denoted using subscripts, thus $f_z = \partial_z f$ and so on.  Primes will be reserved for denoting Wirtinger derivatives of functions with domain $\mathbb{T}$.

\section{Tools from bifurcation theory}
\label{bifurcation theory section}

\subsection{An abstract local bifurcation result}
\label{abstract local bifurcation section}

In our analysis of steady point vortices, we will find that the linearized operator $D\mathcal{V}(\fullparam_0)$ at a steady configuration $\fullparam_0$ is not an isomorphism due to certain structural features of the system \eqref{definition steady point vortex}.  An abstract prototype for this situation can be formulated as follows.

Let $\mathscr{W}$, $\mathscr{V}$, and $\mathscr{Z}$ be Banach spaces, $\mathcal{O} \subset \mathscr{W} \times \mathscr{V}$ an open set containing a point $(w_0, v_0)$. Suppose that $\mathscr{G} = \mathscr{G}(w,v) \colon \mathcal{O} \to \mathscr{Z}$ is real analytic with
\begin{equation}
\label{ift with identity hypothesis G}
	\mathscr{G}(w_0, v_0) = 0, \qquad \kernel{\mathscr{G}_w(w_0,v_0)}= \{ 0 \}, \qquad \cod\range{\mathscr{G}_w(w_0,v_0)} \equalscolon  n
\end{equation}
 for some $n \geq 1$.  Thus we cannot directly apply the implicit function theorem to conclude the existence of nearby solutions.  However, let us assume that the lack of surjectivity is the result of a set of $n$ identities satisfied by the full nonlinear operator $\mathscr{G}$.  These can be expressed in terms of a real-analytic mapping
\[
	\phi = \phi(Z,w,v) \colon \mathscr{Z} \times \mathscr{W} \times \mathscr{V} \to \mathbb{R}^n, 
\]
satisfying
\begin{subequations}
\label{phi assumptions}
\begin{align}
	\phi(\mathscr{G}(w,v), w,v) & = 0 \label{phi 0 at G} \\
	\phi(0,w,v) & = 0 \label{phi 0 at 0} \\
	\phi_Z(0,w_0,v_0) & \textrm{ is surjective} \label{phi surjective}
\end{align}
\end{subequations}
for all $(w,v)$ a neighborhood of $(w_0, v_0)$.

The following lemma, which is a mild generalization of the implicit function theorem, appears in~\cite[lemma~2.6]{hassainia2022multipole}. For completeness and later reference we include a short proof below. 
\begin{lemma}[Implicit function theorem with identities]
\label{local ift phi theorem}
There exists a neighborhood $\mathcal{V} \subset \mathscr{V}$ of $v_0$, a neighborhood $\mathcal{W} \subset \mathscr{W}$ of $w_0$, and a real-analytic mapping $\tilde w \colon \mathcal{V} \to \mathscr{W}$ such that 
\[
	\mathscr{G}(\tilde w(v),v) = 0 \qquad \textrm{for all } v \in \mathcal{V}.
\]
Moreover, if $(w,v) \in \mathscr{G}^{-1}(0) \cap \mathcal{W} \times \mathcal{V}$, then $w = \tilde w(v)$.
\end{lemma}
\begin{proof}
  In view of \eqref{phi surjective}, there exists an $n$-dimensional subspace $\mathscr{Z}_1 \subset \mathscr{Z}$ such that 
  \[
    \phi_Z(0,w_0,v_0)|_{\mathscr{Z}_1} \colon \mathscr{Z}_1 \to \mathbb{R}^n \qquad \textrm{is an isomorphism}.
  \]
  Consider then the augmented map
  \[
    \mathscr{H} \colon \mathscr{Z}_1 \times \mathcal{O} \to \mathscr{Z} \qquad \mathscr{H}(Z_1, w,v) \colonequals  \mathscr{G}(w,v) - Z_1.
  \]
  Notice that \eqref{ift with identity hypothesis G} and \eqref{phi surjective} together imply that $\range{\mathscr{G}_w(w_0,v_0)} = \kernel{\phi_Z(0,w_0,v_0)}$.  Thus by \eqref{phi surjective} and our choice of $\mathscr{Z}_1$, the linearized operator
  \[
    \mathscr{H}_{(Z_1,w)}(0,w_0,v_0)(\dot Z_1, \dot w) = \mathscr{G}_w(w_0,v_0)\dot w-\dot Z_1 
  \]
  is invertible $\mathscr{Z}_1 \times \mathscr{W} \to \mathscr{Z}$.  Applying the implicit function theorem, we infer the existence of real-analytic mappings $\tilde Z_1 \colon \mathcal{V} \to \mathscr{Z}_1$ and $\tilde w \colon \mathcal{V} \to \mathscr{W}$ defined on a neighborhood $\mathcal{V} \subset \mathscr{V}$ of $v_0$ and such that in some neighborhood of $(0,w_0,v_0)$ it holds that
  \[
    \mathscr{H}( Z_1,  w, v) = 0 \qquad \textrm{if and only if} \qquad (Z_1, w) = (\tilde Z_1(v), \tilde w(v)).
  \]

  We claim that in fact $\tilde Z_1$ must vanish identically, and hence $\mathscr{G}(\tilde w(v), v) = 0$.  To this end, consider the restriction 
  \[
    \psi \colonequals \phi|_{\mathscr{Z}_1 \times \mathcal{O}} \colon \mathscr{Z}_1 \times \mathcal{O} \to \mathbb{R}^n.
  \]
  Since $\psi_{Z_1}(0,w_0,v_0) = \psi_Z(0,w_0,v_0)$ is invertible by construction, the implicit function theorem and \eqref{phi 0 at 0} imply that, in a neighborhood of $(0,w_0,v_0)$, 
  \[
    \psi(Z_1, w, v) = 0 \qquad \textrm{if and only if} \qquad Z_1 = 0. 
  \]
  But from \eqref{phi 0 at G}, we have that
  \[
    0 = \phi(\mathscr{G}(\tilde w(v), v), \tilde w(v), v) = \phi(\tilde Z_1(v), \tilde w(v), v) = \psi( \tilde Z_1(v), \tilde w(v), v).
  \]
  Perhaps shrinking $\mathcal{V}$, it follows then that $\tilde Z_1(v) \equiv 0$.  The proof is therefore complete.
\end{proof}

\subsection{Local bifurcation of steady point vortices}
\label{local bifurcation point vortex section} 
We now apply the abstract result Lemma~\ref{local ift phi theorem} to study the existence of steady point vortex configurations.  Recall that 
\[
	\fullparam = \left(\gamma_1,\, \ldots, \,  \gamma_M, \, \zeta_1, \, \ldots, \, \zeta_M, \, c, \, \Omega \right) \in \mathbb{R}^M \times \mathbb{C}^M \times \mathbb{R} \times \mathbb{R} \equalscolon  \fullPspace
\]
denotes the complete set of physical parameters describing the system, where $\zeta_k$ will be the preimage of the center $z_k$ in the conformal domain.  We view $\fullPspace$ as a $(3M+2)$-dimensional vector space over $\mathbb{R}$.  Comparing the Helmholtz--Kirchhoff model \eqref{helmholtz kirchhoff model} to \eqref{definition steady point vortex}, we arrive at a rather complicated algebraic restriction on what parameters $\fullparam$ can possibly represent a steady vortex configuration.  Written abstractly, it takes the form $\mathcal{V}(\fullparam) = 0$, where 
\[
	\mathcal{V} = (\mathcal{V}_1, \ldots, \mathcal{V}_M) \colon \mathbb{P} \to \mathbb{C}^M
\]
is the mapping
\begin{equation}
  \label{dWdzeta point vortex algebra}
  \mathcal{V}_k(\fullparam) \colonequals  \sum_{j \neq k} \frac{\gamma_j}{2\pi i} \frac{1}{\zeta_k - \zeta_j} - c + i \Omega \overline{\zeta_k} \qquad \textrm{for } k = 1, \ldots, M.
\end{equation}
Thus $\fullparam_0 \in \mathbb{P}$ is steady provided that $\mathcal{V}(\fullparam_0) = 0$.

The components of $\mathcal{V}$ satisfy two identities:
\begin{subequations}
\label{pv identities}
\begin{align}
	\sum_k \gamma_k \mathcal{V}_k  & = -c \sum_k \gamma_k + i \Omega \sum_k \gamma_k \overline{\zeta_k} \label{pv translation identity} \\
	\sum_k \gamma_k \zeta_k \mathcal{V}_k & = \frac{1}{2\pi i} \sum_{j < k} \gamma_j \gamma_k - c \sum_k \gamma_k \zeta_k + i \Omega \sum_k \gamma_k |\zeta_k|^2. \label{pv rotation identity}
\end{align}
\end{subequations}
These can be verified directly from \eqref{dWdzeta point vortex algebra}, but we will see that they can also be thought of as the consequence of the translation invariance and rotation invariance of the system, respectively \cite{oneil1987stationary}.  In principle, \eqref{pv identities} restricts the maximum rank of the Jacobian of $\mathcal{V}$.  To say a steady point vortex configuration $\fullparam_0$ is non-degenerate means essentially that $D\mathcal{V}(\fullparam_0)$ fails to be surjective precisely due to \eqref{pv identities}.  One can then find a family of steady configurations bifurcating from $\fullparam_0$ with some appropriate subset of the parameters $\param$  determined by the remaining parameters $\param^\prime$.   How this is accomplished will depend on the particular application.  The identities in \eqref{pv identities} correspond to four real equations, which can potentially lead to $\range{D\mathcal{V}(\fullparam_0)}$ having codimension $4$.  Often, this can be reduced by imposing symmetries, so that we only need a subset of the identities to explain the lack of surjectivity.  

\begin{definition}[Non-degeneracy] \label{definition nondegenerate configuration}
Let $\fullparam_0 \in \mathbb{P}$ be a steady point vortex configuration.
\begin{enumerate}[label=(\roman*)]
\item \label{def translating} $\fullparam_0$ represents a \emph{non-degenerate translating configuration} provided $c\neq0$, $\Omega=0$, and
\begin{equation}
\label{codimension translating case}
  \begin{aligned}
    & \cod \range D \mathcal{V}(\fullparam_0)\big|_{\mathbb{P} \cap \{ \Omega = 0 \}\phantom{,c=0}} & = 1.
  \end{aligned}
\end{equation}
\item \label{def rotating} $\fullparam_0$ represents a \emph{non-degenerate rotating configuration} provided $c = 0$, $\Omega \neq 0$, and
\begin{equation}
\label{codimension rotating case}
  \begin{aligned}
    & \cod \range D \mathcal{V}(\fullparam_0)\big|_{\mathbb{P} \cap \{ c = 0 \}\phantom{\Omega=0,}} & = 1.
  \end{aligned}
\end{equation}
\item \label{def stationary} $\fullparam_0$ represents a \emph{non-degenerate stationary configuration} provided $c = \Omega = 0$, and
\begin{equation}
\label{codimension stationary case}
  \begin{aligned}
    & \cod \range D \mathcal{V}(\fullparam_0)\big|_{\mathbb{P} \cap \{ c=0, \Omega = 0 \}} & = 3.
  \end{aligned}
\end{equation}
\end{enumerate}
\end{definition}

The dimension counts in \eqref{codimension translating case}--\eqref{codimension stationary case} are motivated by \eqref{pv identities} as follows.  Notice that if we assume $\Omega = 0$, then the right-hand side of \eqref{pv translation identity} is purely real.  Thus the range of $D \mathcal{V}(\fullparam_0)$ restricted to $\mathbb{P} \cap \{ \Omega = 0 \}$ has codimension at least $1$.  In \eqref{codimension translating case}, we ask that this is the only source of degeneracy.  Likewise, for rotating configurations, the right-hand side of \eqref{pv rotation identity} is purely imaginary, so the minimal codimension of $\range{D\mathcal{V}(\fullparam_0)}$ when restricted to this subspace is again $1$.  Finally, in the stationary case, the right-hand side of \eqref{pv translation identity} and the real part of the right-hand side of \eqref{pv rotation identity} vanish, hence the minimal codimension is $3$, as required in \eqref{codimension stationary case}.

Having this in mind, we introduce the mappings
\begin{equation}
\label{pv phi functions}
  \begin{aligned}
    \phi_{\rm t}^0 &: \mathbb{C}^M \times \mathbb{P} \to \mathbb{R}   &\quad ({V}, \fullparam) & \mapsto \imagpart{\sum_k \gamma_k {V}_k }  \\
    \phi_{\rm r}^0 &: \mathbb{C}^M \times \mathbb{P} \to \mathbb{R}  &\quad ({V}, \fullparam) & \mapsto \realpart{\sum_k \gamma_k \zeta_k {V}_k }  \\
    \phi_{\rm s}^0 &: \mathbb{C}^M \times \mathbb{P} \to \mathbb{C} \times \mathbb{R} & \quad ({V}, \fullparam) & \mapsto \Big( \sum_k \gamma_k {V}_k , ~ \realpart{\sum_k \gamma_k \zeta_k {V}_k } \Big).
  \end{aligned}
\end{equation}
As discussed above, these are associated to the identities that cause the degeneracy of a translating, rotating, and stationary point vortex configuration, respectively.  Abusing notation slightly, we will occasionally  write $\phi^0$ to stand for $\phi_{\rm t}^0$, $\phi_{\rm r}^0$, or $\phi_{\rm s}^0$.  The next lemma shows that $\phi^0$ satisfies the hypotheses \eqref{phi assumptions} of the higher codimension implicit function theorem recorded in Lemma~\ref{local ift phi theorem}
\begin{lemma}
\label{Dphi pv range lemma}
	Let $\fullparam_0$ be a non-degenerate steady point vortex configuration that is either translating, rotating, or stationary, and let $\phi^0$ be the corresponding map from \eqref{pv phi functions}.  Then $D_\fullparam \phi^0(0,\fullparam_0)$ is surjective.
\end{lemma}
\begin{proof}
For the translating and rotating point vortex configurations, this follows immediately from the definition of $\phi^0$.  Consider then the stationary case.  To prove that $D_V \phi_{\rm s}^0(0,\fullparam_0)$ is surjective, it suffices to consider the minor of the Jacobian $D_{(V_1,V_2)} \phi_{\rm s}^0(0,\fullparam_0)$.  Identifying the domain $\mathbb{C}^2 \cong \mathbb{R}^4$ and codomain $\mathbb{C} \times \mathbb{R} \cong \mathbb{R}^3$, this can be written as 
\[
	D_{(V_1,V_2)} \phi_{\rm s}^0(0,\fullparam_0) \begin{pmatrix} \dot V_1 \\ \dot V_2 \end{pmatrix} 
		= 
	\begin{pmatrix}
		\gamma_1^0 & 0  & \gamma_2^0 & 0 \\
		0 & \gamma_1^0 & 0 & \gamma_2^0 \\
		\gamma_1^0 \realpart \zeta_1^0 & -\gamma_1^0 \imagpart \zeta_1^0   & \gamma_2^0 \realpart \zeta_2^0 & -\gamma_2^0 \imagpart \zeta_2^0 
	\end{pmatrix}
	\begin{pmatrix}
		\realpart{\dot{{V}}_1} \\
		\imagpart{\dot{{V}}_1}\\
		\realpart{\dot{{V}}_2} \\
		\imagpart{\dot{{V}}_2}
	\end{pmatrix}.
\]
First note that none of the vortex strengths can vanish.  Thus, if $\realpart{\zeta_1^0} \neq \realpart{\zeta_2^0}$, then the first three columns span $\mathbb{R}^3$.  On the other hand, if the real parts are the same, then because the vortex centers are distinct, we must have that $\imagpart{\zeta_1^0} \neq \imagpart{\zeta_2^0}$.  In that case, the last three columns span $\mathbb{R}^3$.   
\end{proof}

Suppose now that $\fullparam_0$ is a non-degenerate steady configuration in the sense of Definition~\ref{definition nondegenerate configuration}.  Let $\mathbb{P}_0$ stand for $\mathbb{P} \cap \{ \Omega = 0\}$, $\mathbb{P} \cap \{ c = 0 \}$, or $\mathbb{P} \cap \{ c=\Omega = 0\}$, according to whether $\fullparam_0$ is translating, rotating, or stationary, respectively.  Set $n \colonequals  \cod\range{D\mathcal{V}(\fullparam_0)|_{\mathbb{P}_0}}$.  Then there exists a $(2M-n)$-dimensional subspace $\Pspace$ of $\mathbb{P}_0$ with
\begin{equation}
  \label{definition non-degnerate subspace}
  \cod\range{D \mathcal{V}(\fullparam_0)\big|_{\mathcal{P}}} = n, \qquad \dimension\kernel{D\mathcal{V}(\fullparam_0)\big|_{\mathcal{P}}} = 0.
\end{equation}
Say $\mathcal{P}^\prime$ is the direct complement of $\mathcal{P}$ in $\mathbb{P}_0$, so that we can identify $\mathcal{P} \times \mathcal{P}^\prime \cong \mathbb{P}_0$.  Equivalently, we imagine partitioning the parameters into two: a $(2M-n)$-dimensional vector $\param$ and an $(M+2+n)$-dimensional vector $\param^\prime$.    In view of Definition~\ref{definition nondegenerate configuration} and Lemma~\ref{Dphi pv range lemma}, an appeal to Lemma~\ref{local ift phi theorem} yields the following.

\begin{proposition}
\label{point vortex manifold proposition}
Suppose that $\fullparam_0 = (\param_0,\param_0^\prime) \in \mathcal{P} \times \mathcal{P}^\prime$ is a non-degenerate steady vortex configuration in the sense of Definition~\ref{definition nondegenerate configuration} with $\mathcal{P}$ satisfying \eqref{definition non-degnerate subspace}.   There exists a neighborhood $\mathcal{N}^0 \subset \mathcal{P}^\prime$ of $\param_0^\prime$ and a manifold $\mathscr{M}_\loc^0$ of steady point vortex configurations of the same type as $\fullparam_0$.  In particular, $\mathscr{M}_\loc^0$ takes the form of a graph:
\[
	\mathscr{M}_\loc^0 \colonequals  \left\{ (\param(\param^\prime), \param^\prime) \in \mathcal{P} \times \mathcal{P}^\prime : \param^\prime \in \mathcal{N}^0 \right\} \subset \mathcal{V}^{-1}(0),
\]
and the coordinate map $\param^\prime \mapsto \param(\param^\prime)$ is real analytic $\mathcal{N}^0 \to \mathcal{P}$, with $\param(\param_0^\prime) = \param_0$.  In a neighborhood of $\fullparam_0$, the manifold $\mathscr{M}_\loc^0$ contains all steady point vortex configurations of the same type.
\end{proposition}

For the hollow vortex problem, we will ultimately group the parameters in $\param$ with the unknowns describing the geometry of the fluid domain and the velocity field; these will vary along the family of solutions to be constructed.  On the other hand, the parameters collected in $\param^\prime$ will be fixed along the entire curve.

Definition~\ref{definition nondegenerate configuration} can be directly verified for many of the point vortex equilibria appearing in the literature. Below we content ourselves with the three examples related to Theorems~\ref{intro rotating theorem}--\ref{intro pocklington theorem}.

\begin{example}[Translating pair] 
\label{translating pair example}
Consider first the case of two point vortices ($M = 2$) that are translating ($\Omega = 0$) parallel to the real axis.  Grouping the parameters into
\[
	\param \colonequals  (\gamma_1,\zeta_2) \qquad \param^\prime \colonequals  (\gamma_2, \zeta_1, c, \Omega) 
\]
we see that $\fullparam_0 \colonequals  (\param_0, \param_0^\prime)$ defined by the values 
\[
	(\gamma_1^0, \gamma_2^0) = (1,-1), \qquad (\zeta_1^0,\zeta_2^0) = ( i, -i), \qquad (c,\Omega) = \left(\frac{1}{4\pi}, 0 \right)
\]
is a steady translating configuration. Moreover, identifying the domain $\mathbb{R} \times \mathbb{C} \cong \mathbb{R}^3$ and codomain $\mathbb{C}^2 \cong \mathbb{R}^4$, we see that
\[
	D_\param \mathcal{V}(\fullparam_0) = 
		-\frac{1}{8\pi}
		\begin{pmatrix}
			0 & 0 &  1  \\
			0 & 1 &  0  \\
			2  & 0 & 1  \\
			0& 1 &  0 
		\end{pmatrix},
\]
  which has rank $3$.  The codimension is thus $1$, and so it is non-degenerate in the sense of Definiton~\ref{definition nondegenerate configuration}\ref{def translating}. Applying Lemma~\ref{local ift phi theorem}, we can locally solve \eqref{dWdzeta point vortex algebra} for $\param = \param(\param')$.
\end{example}

\begin{example}[Rotating pair] \label{rotating pair example} Suppose now that we have a pair ($M = 2$) of point vortices having the same signed circulating, which will result in a rotating configuration ($c = 0$).  We take
\[
	\param = (\realpart{\zeta_1}, \zeta_2), \quad \param^\prime \colonequals  (\gamma_1, \gamma_2, \imagpart{\zeta_1}, c, \Omega)
\]
Then $\fullparam_0 = (\param_0, \param_0^\prime)$ is a steady configuration for
\[
	(\gamma_1^0, \gamma_2^0) \colonequals  ( 1, 1), \qquad (\zeta_1^0, \zeta_2^0) \colonequals  (1, -1), \qquad (c, \Omega) = \left(0, \frac{1}{4\pi} \right).
\]
Identifying the domain $\mathbb{R} \times \mathbb{C} \cong \mathbb{R}^3$ and codomain $\mathbb{C}^2 \cong \mathbb{R}^4$, we see that the Jacobian at $\fullparam_0$ is given by
\[
	D_\param \mathcal{V}(\fullparam_0) = 
		\frac{1}{8\pi}
		\begin{pmatrix}
			0 & 0 & 1 \\
			3 & -1 & 0 \\
			0 & 0 & 1 \\
			-1 & 3 & 0
		\end{pmatrix}.
\]
This clearly has three-dimensional range, and hence $\cod{\range{D_\param \mathcal{V}(\fullparam_0)}} = 1$.  Recalling \eqref{codimension rotating case}, we conclude that $\fullparam_0$ is non-degenerate in the sense of Definition~\ref{definition nondegenerate configuration}\ref{def rotating}
\end{example}

\begin{example}[Stationary tripole] \label{stationary tripole example} Finally, we consider a stationary ($c = \Omega = 0$) configuration of three point vortices ($M = 3$), grouping the remaining parameters as
\[
	\param = (\gamma_1, \zeta_3), \qquad \param^\prime = (\gamma_2, \gamma_3, \zeta_1, \zeta_2, c, \Omega).
\]
One solution of this form is $\fullparam_0 = (\param_0, \param_0^\prime)$ for
\[ 
	(\gamma_1^0, \gamma_2^0, \gamma_3^0) \colonequals  (2,-1,2),\qquad (\zeta_1^0, \zeta_2^0, \zeta_3^0) \colonequals  (1,0,-1), \qquad (c,\Omega) = (0,0).
\]
We then compute that
\[
	D_\param \mathcal{V}(\fullparam_0) 
		= \frac{1}{4\pi} \begin{pmatrix}
			0 & 0 & 1 \\
			0 & -1 & 0 \\
			0 & 0 & 4 \\
			2 & -4 & 0 \\
			0 & 0 & 1 \\
			1 & -1 & 0
		\end{pmatrix},
\]
where we have again identified the domain $\mathbb{R} \times \mathbb{C} \cong \mathbb{R}^3$ and codomain $\mathbb{C}^3 \cong \mathbb{R}^6$ in the usual way.  The columns above are clearly linearly independent, hence the range of the Jacobian above is three dimensional, meaning the codimension is $3$ as required by \eqref{codimension stationary case}.   Thus this is a non-degenerate stationary configuration as in Defintion~\ref{definition nondegenerate configuration}\ref{def stationary}.  
\end{example}

\subsection{Global implicit function theorem}
\label{global ift section}

The continuation argument underlying Theorems~\ref{intro rotating theorem}, \ref{intro tripole theorem}, \ref{intro pocklington theorem}, and \ref{intro global continuation theorem} is based on analytic global bifurcation theory, which was first introduced by Dancer \cite{dancer1973bifurcation,dancer1973globalstructure} in the 1970s, and then refined by Buffoni and Toland \cite{buffoni2003analytic}.  Here we give a version of these results that is stated in the form of a global implicit function theorem.  Following the philosophy from \cite{chen2018existence}, we have avoided making certain compactness assumptions a priori, preferring instead to treat their violation as potential alternatives for the limiting behavior along the solution curve.  

\begin{theorem}[Global implicit function theorem] \label{homoclinic global ift} 
Let $\mathscr{W}$ and $\mathscr{Z}$ be Banach spaces, $\mathcal{O} \subset \mathscr{W} \times \mathbb{R}$ an open set containing a point $(w_0, \param_0)$. Suppose that  $\mathscr{G} \colon \mathcal{O} \to \mathscr{Z}$ is real analytic and satisfies 
  \begin{equation}
    \mathscr{G}(w_0, \lambda_0) = 0, \qquad \mathscr{G}_w(w_0, \lambda_0) \colon \mathscr{W} \to \mathscr{Z} \quad \textrm{is an isomorphism}.  \label{homoclinic global ift assumptions} 
  \end{equation}
 Then there exist a curve $\mathscr{K}$ that admits the global $C^0$ parameterization  
 \[ \mathscr{K} \colonequals  \left\{ (w(s), \param(s)) : s \in \mathbb{R}  \right \} \subset \mathscr{G}^{-1}(0) \cap \mathcal{O},\]
  and satisfies the following.  
  \begin{enumerate}[label=\rm(\alph*)]
  \item \label{K well behaved} At each $s \in \mathbb{R}$, the linearized operator $\mathscr{G}_w(w(s), \param(s)) \colon \mathscr{W} \to \mathscr{Z}$ is Fredholm index $0$.
  \item \label{K alternatives} One of the following alternatives holds as $s \to \infty$ and $s \to -\infty$.
    \begin{enumerate}[label=\rm(A\arabic*$'$)]
    \item  \label{K blowup alternative}
      \textup{(Blowup)}  The quantity 
      \begin{align}
        \label{K blowup}
        N(s)\colonequals  \n{w(s)}_{\mathscr{W}} + |\param(s)| +\frac 1{\dist((w(s),\param(s)), \, \dell \mathcal{O})}
      \end{align}
    \item \label{K loss of compactness alternative} \textup{(Loss of compactness)} There exists a sequence $s_n \to \pm\infty$ with $\sup_n N(s_n) < \infty$, but $( w(s_n), \param(s_n) )$ has no convergent subsequence in $\mathscr{W} \times \mathbb{R}$.  
      \item \label{K loss of fredholmness alternative} \textup{(Loss of Fredholmness)}    There exists a sequence $s_n \to \pm\infty$
        with $\sup_{n} N(s_n) < \infty$ and so that $(w(s_n), \param(s_n)) \to (w_*, \param_*) \in \mathcal{W}$ in $\mathscr{W} \times \mathbb{R}$, however $\mathscr{G}_w(w_*, \param_*)$ is not semi-Fredholm.  
   \item \label{K loop} \textup{(Closed loop)} There exists $T > 0$ such that $(w(s+T), \param(s+T)) = (w(s), \param(s))$ for all $s \in (0,\infty)$. 
        \end{enumerate}
          \item \label{K reparam} Near each point $(w(s_0),\param(s_0)) \in \mathscr{K}$, we can locally reparametrize $\mathscr{K}$ so that $s\mapsto (w(s),\param(s))$ is real analytic.
 \item \label{K maximal part} The curve $\mathscr{K}$ is maximal in the sense that, if $\mathscr{J} \subset \mathscr{G}^{-1}(0) \cap \mathcal{W}$ is a locally real-analytic curve containing $(w_0,\param_0)$ and along which  $\mathscr G_w$ is Fredholm index 0, then $\mathscr{J} \subset \mathscr{K}$. 
  \end{enumerate}
\end{theorem}
This theorem follows from a straightforward adaptation of \cite[Theorem 6.1]{chen2018existence}.  A sketch of the argument can be found in \cite[Appendix B]{chen2020global}.  Note that here we have slightly rephrased \ref{K loss of fredholmness alternative} to be the loss of semi-Fredholmness, which by continuity of the index is equivalent to the index ceasing to be $0$ as stated in \cite{chen2020global}.

Assume that the space $\mathscr{V}$ admits the decomposition $\mathscr{V} = \mathbb{R} \times \mathscr{V}^\prime$, and let elements of $\mathscr{V}$ be accordingly denoted $(\param, \param^\prime)$.  We will think of $\param \in \mathbb{R}$ as the varying direction with the components in $\mathscr{V}^\prime$ fixed.    Under the hypotheses of Lemma~\ref{local ift phi theorem}, there then exists a curve $\mathscr{K}_\loc \subset \mathscr{W} \times \mathbb{R}$ admitting the real-analytic parameterization
\[
	\mathscr{K}_\loc = \{ ( \tilde w(\param,\param_0^\prime), \lambda,\lambda_0^\prime) : |\lambda - \lambda_0| \ll 1 \} \subset \mathscr{G}^{-1}(0)
\]
where $v_0 = (\lambda_0, \lambda_0^\prime)$.  

We now argue as in the proof of Theorem~\ref{homoclinic global ift} to extend this curve globally.  Because $\lambda^\prime$ will be fixed to $\lambda_0^\prime$ throughout this procedure, it will be convenient to suppress the dependence of all quantities on it.  Thus, $\mathcal{O}$ will now be the corresponding open subset of $\mathscr{W} \times \mathbb{R}$,  $\mathscr{G}$ is viewed as a a mapping $\mathscr{G} = \mathscr{G}(w,\lambda) \colon \mathcal{O} \to \mathscr{Z}$, and $\mathscr{K}_\loc \subset \mathcal{O}$. 

\begin{corollary}[Global implicit function theorem with identities]
\label{global ift with identities corollary}
There exists a curve $\mathscr{K} \subset \mathcal{O}$ of solutions that admits the global $C^0$ parameterization
\[
	\mathscr{K} = \{ (w(s), \lambda(s)) : s \in \mathbb{R} \} \subset \mathscr{G}^{-1}(0) 
\]
with $\mathscr{K}_\loc \subset \mathscr{K}$ and satisfying the following.
\begin{enumerate}[label=\rm(\alph*)]
\item \label{K alternatives with identities} As $s \to \infty$ and $s \to -\infty$, one of the alternatives \ref{K blowup alternative}, \ref{K loss of compactness alternative}, \ref{K loss of fredholmness alternative}, or \ref{K loop} occurs.     
\item \label{K reparam with identities} Near each point $(w(s_0),\lambda(s_0)) \in \mathscr{K}$, we can locally reparametrize $\mathscr{K}$ so that $s\mapsto (w(s),\lambda(s))$ is real analytic.
\item \label{K maximal part with identities} The curve $\mathscr{K}$ is maximal in the sense that, if $\mathscr{J} \subset \mathscr{G}^{-1}(0) \cap \mathcal{W}$ is a locally real-analytic curve containing $(w_0,\lambda_0)$ and along which  $\mathscr G_w$ is semi-Fredholm, then $\mathscr{J} \subset \mathscr{K}$. 
\end{enumerate}
\end{corollary}
\begin{proof}
Let $\mathscr{H} : \mathscr{Z}_1 \times \mathcal{O} \to \mathscr{Z}$ be given as in the proof of Lemma~\ref{local ift phi theorem}.  By the same argument, we have that $D_{(Z_1,w)} \mathscr{H}(0,w_0,\lambda_0)$ is an isomorphism $\mathscr{Z}_1 \times \mathscr{W} \to \mathscr{Z}$, and so we can apply Theorem~\ref{homoclinic global ift} to infer the existence of a solution curve 
\[
	\tilde{\mathscr{K}} = \{ (Z_1(s), w(s), \lambda(s)) : s \in \mathbb{R} \}  \subset \mathscr{Z}_1 \times \mathscr{W} \times \mathbb{R}.
\]  
that is globally $C^0$ and locally real analytic.  But, in the proof of Lemma~\ref{local ift phi theorem}, it was also shown that in a neighborhood of $(0,w_0, \lambda_0)$, the zero-set of $\mathscr{H}$ lies in the subspace $\{ Z_1 = 0\}$.  Local real-analyticity and maximality imply that $Z_1(\placeholder)$ vanishes identically.  Setting $\mathscr{K}$ to be the projection of $\tilde{\mathscr{K}}$ into $\mathscr{W} \times \mathbb{R}$, we see that $\mathscr{K} \subset \mathscr{G}^{-1}(0)$.  Local uniqueness shows $\mathscr{K}_\loc \subset \mathscr{K}$, and the rest of the properties claimed above are inherited from the corresponding statements about $\tilde{\mathscr{K}}$.
\end{proof}

\section{Tools from complex analysis}
\label{complex analysis section}

\subsection{Circular domains} 
\label{circular domains section}

Multiply connected subsets of $\mathbb{C}$ can be represented canonically in a variety of ways.  A \emph{circular domain} is a subset of $\mathbb{C}$ whose boundary is composed of a finite disjoint union of circles.  The set $\confD_{\rho_1, \ldots,\rho_m}(\fullparam)$ introduced in \eqref{conformal domain definition} is one example.  The following deep theorem of Koebe generalizes the Riemann mapping theorem to multiply connected domains. Here it is stated in a slightly less general form that is best suited to our purposes.

\begin{theorem}[Koebe, \cite{koebe1918abhandlungen}]
  \label{koebe theorem}
  For every $M$-connected domain $\fluidD \subset \mathbb{C}$, there exists  an $M$-connected circular domain $\confD \subset \mathbb{C}$ and an injective conformal mapping $f \maps \fluidD \to \mathbb{C}$ such that $f(\fluidD) = \confD$.  Moreover, this mapping and domain are unique provided we require in addition that
  \begin{equation}\label{koebe theorem normalization}
    f(z) = z+O\left(\frac{1}{z}\right) \textrm{ as } z \to \infty.
  \end{equation}
\end{theorem}

See also \cite[Chapter V \S 6]{goluzin1969geometric}.  In particular, Theorem~\ref{koebe theorem} shows that the onset of conformal degeneracy is fundamental rather than an artifact of our choice of conformal domain. 

For the set $\confD_\rho(\zeta_1,\ldots, \zeta_M)$ to actually be a circular domain, the distance between any two vortex centers must be at least $2\rho$.  This restriction forces us to work within a certain subset of parameter space.  It will also be convenient for the implicit function theorem argument to allow $\rho \leq 0$, with the convention that $\confD_\rho = \confD_{|\rho|}$ for $\rho < 0$, and $\confD_0 \colonequals  \mathbb{C} \setminus \{ \zeta_1, \ldots, \zeta_M\}$.  We therefore introduce the open sets
\begin{equation}
  \label{definition of U neighborhood}
  \mathcal{U}_\delta \colonequals  \big\{ (\rho,\fullparam) \in \mathbb{R} \times \mathbb{P} : \min_{j\neq k} |\zeta_j - \zeta_k| > 2|\rho|+\delta \big\}, \qquad \mathcal{U} \colonequals  \bigcup_{\delta \geq 0} \mathcal{U}_\delta.
\end{equation}

Finally, let us note that by selecting one of the centers $\zeta_k$ and applying a spherical inversion through $\partial B_{\rho_k}(\zeta_k)$, we can transform $\confD_\rho$ to a (bounded) circular domain lying in the interior of $B_{\rho_k}(\zeta_k)$.  The Green's function for the Laplacian on sets of this form can be determined explicitly in terms of Shlottky--Klein prime functions, and hence these types of circular domains are often preferred by authors making use of such techniques; see, for example, \cite{crowdy2020book}.  As we will instead perform vortex desingularization via the implicit function theorem, it is simpler to formulate the problem so that the conformal mapping $f$ is near identity.  Because we anticipate the streamlines of the small hollow vortices being asymptotically circular, using $\confD_{\rho_1,\ldots, \rho_M}$ as the conformal domain with $\zeta_k \approx z_k$ is a natural choice.

\subsection{Injectivity and conformality}
\label{injectivity conformality section}

While the methods from the previous subsection allow us to represent mappings $f \in C^{\ell+\alpha}(\overline{\confD_\rho},\mathbb C)$ which are holomorphic in $\confD_\rho$ and satisfy the normalization \eqref{koebe theorem normalization}, it remains to check when such mappings are univalent conformal mappings onto their images.

First we give a simple sufficient condition for the derivative $\partial_\zeta f$ to be non-vanishing.
\begin{lemma}\label{lem basic nondeg f_zeta}
  Suppose that $f \in C^2(\overline{\confD_\rho},\mathbb C)$ is holomorphic with $\partial_\zeta f \ne 0$ on $\partial \confD_\rho$. If $f$ satisfies the normalization condition \eqref{koebe theorem normalization} and the integral condition
  \begin{align}
    \label{eqn no winding f_zeta}
    \frac 1{2\pi i}\int_{\partial \confD_\rho} \frac{\partial_\zeta^2 f}{\partial_\zeta f}\, d\zeta = 0, 
  \end{align}
  then $\partial_\zeta f \ne 0$ in $\confD_\rho$.
  \begin{proof}
    As $\partial_\zeta f$ has no poles in $\confD_\rho$, this follows from Cauchy's argument principle, i.e.~from calculating the left hand side of \eqref{eqn no winding f_zeta} using the calculus of residues.
  \end{proof}
\end{lemma}

Next, we give a sufficient condition for $f$ to be injective.
\begin{lemma}\label{lem basic injective}
  In the setting of Lemma~\ref{lem basic nondeg f_zeta}, suppose further that $f|_{\partial \confD_\rho}$ is injective, and that none of the simple curves $\Gamma_k \colonequals  f(\partial B_\rho(\zeta_k))$ enclose one another. Then $f$ is globally injective.
  \begin{proof}
    Using the Jordan curve theorem and basic properties of the winding number, there is a unique unbounded domain $\fluidD$ with boundary $\partial \fluidD = \bigcup_k \Gamma_k$, given by
    \begin{align*}
      \fluidD \colonequals   \left\{ z \in \mathbb C : \frac 1{2\pi i}\sum_{k=1}^M \int_{\Gamma_k} \frac{d\tilde z}{\tilde z - z} = 0 \right\}.
    \end{align*}
    Changing variables inside the integral, we see that $z \in \fluidD$ if and only if 
    \begin{align*}
      \frac 1{2\pi i}\int_{\partial \confD_\rho} \frac{\partial_\zeta f(\zeta)}{f(\zeta)-z}\, d\zeta = 0.
    \end{align*}
    Using Cauchy's argument principle and the calculus of residues, this simplifies to $N(z) - 1 = 0$, where $N(z)$ is the number of roots $\zeta \in \confD_\rho$ of the equation $f(\zeta)=z$ and the $-1$ comes from the residue at infinity. Thus $z \in \fluidD$ if and only if $N(z)=1$, i.e.~$f$ is bijective $\confD_\rho \to \fluidD$.
  \end{proof}
\end{lemma}

\subsection{Layer-potential representations}
\label{layer potential section}

As outlined in Section~\ref{introduction section}, our strategy will be to reformulate the hollow vortex system \eqref{intro hollow vortex problem} in terms of boundary traces of our unknowns in the conformal domain.  Because we wish to smoothly collapse the vortices to points, it will be advantageous to work with quantities defined on $M$ copies of $\mathbb{T}$ rather than $\partial\confD_\rho$.

With that in mind, for $(\rho, \fullparam) \in \mathcal{U}$ and real-valued densities $ \mu = (\mu_1, \ldots, \mu_M) \in \mathring{C}^{\ell+\alpha}(\mathbb{T})^M$, we define the layer-potential operator
\begin{equation}
  \label{definition Z operator}
   \mathcal{Z}^\rho(\fullparam)[\mu](\zeta) \colonequals  \sum_{k=1}^M \frac{1}{2\pi i} \int_{\mathbb{T}} \frac{\mu_k(\sigma)}{\rho \sigma + \zeta_k - \zeta} \, \rho d\sigma.
\end{equation}
It is clear that the right-hand side of \eqref{definition Z operator} defines a single-valued holomorphic function in $\confD_\rho$ that vanishes at infinity.  Moreover, for fixed $(\rho,\fullparam) \in \mathcal{U}$ with $\rho > 0$, we have by Privalov's theorem that 
\[
	\mathcal{Z}^\rho(\fullparam) \quad \textrm{is bounded} \quad C^{\ell+\alpha}(\mathbb{T})^M \to C^{\ell+\alpha}(\overline{\confD_\rho},\mathbb C).
\]
Due to the commutation identity $\partial_\zeta \left( \mathcal{Z}^\rho \mu \right) = \frac{1}{\rho} \mathcal{Z}^\rho\mu^\prime$, which holds for $0 < |\rho| \ll 1$, these bounds will be highly $(\rho,\Lambda)$ dependent, though they are uniform on compact subsets of $\mathcal{U} \setminus \{ \rho = 0\}$.  Of course, the right-hand side of \eqref{definition Z operator} also gives a single-valued holomorphic function on $\mathbb{C} \setminus \overline{\confD_\rho}$, but we will use $\mathcal{Z}^\rho$ exclusively to denote the function defined on the (unbounded) external domain.  We also note that the formula \eqref{definition Z operator} remains well-defined for $-1 \ll \rho < 0$, and indeed 
\[
	\mathcal{Z}^\rho\mu = \mathcal{Z}^{-\rho}[\mu(-\placeholder)] \qquad \textrm{for all } \mu \in C^{\ell+\alpha}(\mathbb{T})^M.
\]

Both the kinematic and dynamic conditions are posed on the boundary of $\fluidD$, and hence we will need to be able to evaluate the traces of $\mathcal{Z}^\rho(\fullparam) \mu$ on the components of $\partial\confD_\rho$.  Towards that end, we introduce the operators
\begin{equation}
 \label{trace Z operator}
  \begin{aligned}
    \mathcal{Z}_k^\rho[\mu](\tau) & \colonequals  \mathcal{Z}^\rho[\mu](\zeta_k + \rho \tau) \\ 
    & = \frac{1}{2\pi i} \int_{\mathbb{T}} \frac{\mu_k(\sigma) - \mu_k(\tau)}{\sigma - \tau} \, d\sigma + \sum_{j \neq k} \frac{1}{2\pi i} \int_{\mathbb{T}} \frac{\mu_j(\sigma)}{\rho (\sigma - \tau) + \zeta_j - \zeta_k} \rho \, d\sigma.
  \end{aligned}
\end{equation}
The second line results from the Sokhotski--Plemelj formula.  Here we are continuing the convention of suppressing dependence on $\fullparam$ when there is no risk of confusion.  It follows from standard potential theory that for fixed $\rho$ and $\Lambda$, 
\[
	\mathcal{Z}^\rho_k(\Lambda) \quad \textrm{is bounded} \quad C^{\ell+\alpha}(\mathbb{T})^M \to C^{\ell+\alpha}(\mathbb{T},\mathbb C).
\]
Moreover, because these trace operators in fact commute with differentiation,
\begin{equation}
  \label{Zk derivative commutation}
  \partial_\tau \mathcal{Z}_k^\rho = \mathcal{Z}_k^\rho[ \partial_\tau \placeholder],
\end{equation}
it can be easily verified that 
\[
	(\rho,\fullparam) \mapsto \mathcal{Z}_k^\rho(\fullparam) \quad \textrm{is real analytic} \quad \mathcal{U} \to \mathcal{L}(C^{\ell+\alpha}(\mathbb{T})^M, C^{\ell+\alpha}(\mathbb{T},\mathbb C)).
\]
 In particular, $\mathcal{Z}_k^0 \mu = \mathcal{C} \mu_k$, where $\mathcal{C}$ is the Cauchy-type integral operator
\begin{equation}
  \label{definition C operator}
  \mathcal{C}[\mu_k](\tau) \colonequals  \frac{1}{2 \pi i} \int_{\mathbb{T}} \frac{\mu_k(\sigma) - \mu_k(\tau)}{\sigma - \tau} \, d\sigma.
\end{equation}
One can verify that $\mathcal{C} \tau^m = 0$ for $m \ge 0$ and $\mathcal{C} \tau^m = -\tau^m$ for $m < 0$. Equivalently, $2\mathcal{C} = i \mathcal{H} + P_0 -1$, where $\mathcal{H}$ is the Hilbert transform on the circle.

Finally, it is useful to note that one can invert the layer-potential representation \eqref{trace Z operator} to express $\mu_k$ in terms of $\mathcal{Z}_k^\rho \mu$.    In particular, a function $g \in C^{\alpha}(\mathbb{T},\mathbb{C})$ is the trace of a holomorphic function on $B_1(0)$ if and only if $g \in \ker{\mathcal{C}}$. Assuming that $\mu_k \in \mathring{C}^{\alpha}(\mathbb{T})$, we then find from \eqref{trace Z operator} that
\[
	 \mathcal{Z}_k^\rho \mu =  \mathcal{C} \mu_k  + g_k,
\]
for an explicit function $g_k$ that is holomorphic on $B_1(0)$, whence 
\[
	\frac{1-i\mathcal{H}}{2} \mu_k = \left( \frac{1-i\mathcal{H}-\proj_0}{2} \right)^2 \mu_k = \mathcal{C} \mathcal{Z}_k^\rho \mu.
\]
Because $\mu_k$ is real valued, this leads to the following inversion formula
\[
\begin{aligned}
	\mu_k & =  \frac{1-i\mathcal{H}}{2} \mu_k +  \frac{1+i\mathcal{H}}{2} \mu_k  =   \frac{1-i\mathcal{H}}{2} \mu_k + \overline{ \frac{1-i\mathcal{H}}{2} \mu_k} = 2 \realpart{\left(\mathcal{C}  \mathcal{Z}_k^\rho \mu \right)}.  
\end{aligned}
\]
Thanks to the commutation identity \eqref{Zk derivative commutation}, the same argument applied to $\partial_\tau^\ell \mu_k$ yields
\begin{equation}
  \label{S-1 mu identity}
  \partial_\tau^\ell \mu_k = 2 \realpart{\left( \mathcal{C}\partial_\tau^\ell \mathcal{Z}_k^\rho \mu \right)} \qquad \textrm{for all } \ell \geq 0.
\end{equation}
The next lemma is an immediate consequence of \eqref{S-1 mu identity} and the boundedness of $\mathcal{C}$.

\begin{lemma}
\label{layer potential bounds lemma}
Suppose that $\mu_k \in \mathring{C}^\alpha(\mathbb{T})$.  For all $\ell \geq 0$ and $1 < p < \infty$, it holds that
  \begin{equation}
    \label{layer potential schauder}
    \| \partial_\tau^\ell \mu_k \|_{C^\alpha(\mathbb{T})} \leq C_\alpha \| \partial_\tau^\ell \mathcal{Z}_k^\rho \mu \|_{C^\alpha(\mathbb{T})}, \qquad \| \partial_\tau^\ell \mu_k \|_{L^p(\mathbb{T})} \leq C_p \| \partial_\tau^\ell \mathcal{Z}_k^\rho \mu \|_{L^p(\mathbb{T})}
  \end{equation}
for constants $C_\alpha, C_p > 0$ depending only on $\alpha$ and $p$ respectively.  
\end{lemma}

\subsection{Symmetries with respect to the axes} \label{symmetries subsection}

We say a function $f : \mathcal{U} \subset \mathbb{C} \to \mathbb{C}$ is \emph{real on real} provided $\imagpart{f}$ vanishes identically on $\mathbb{R} \cap \mathcal{U}$.  By the Schwarz reflection principle, if $\mathcal{U}$ is symmetric with respect to the real axis and $f$ is holomorphic, then $f$ is real on real if and only if $f = f^*$, where
\[
	f^*(\zeta) \colonequals  \overline{f(\overline{\zeta})}
\]
is the \emph{Schwarz conjugate} of $f$.  Similarly, we say $f$ is \emph{imaginary on imaginary} if $\realpart{f}$ vanishes on $i\mathbb{R} \cap \mathcal{U}$ and \emph{real on imaginary} if $\imagpart{f}$ vanishes there.  When $\mathcal{U}$ is symmetric with respect to the imaginary axis and $f$ is holomorphic, these are equivalent to having $f^* = -f(-\placeholder)$ and $f^* = f(-\placeholder)$, respectively. 

In our later analysis, it will be useful to understand what assumptions on the densities $\mu$ cause the layer-potential $\mathcal{Z}^\rho \mu$ given by \eqref{definition Z operator} to fall into one or more of these categories.  The following elementary lemma answers this question.

\begin{lemma} \label{symmetry over axes lemma}
Consider the conformal mapping $\mathcal{Z}^\rho(\Lambda) \mu$ given by \eqref{definition Z operator}.
\begin{enumerate}[label=\rm(\alph*\rm)]
	\item If for every $1 \leq k \leq M$ there exists $1 \leq k^\prime \leq M$ such that
	\[
		\overline{\zeta_k} = \zeta_{k^\prime} \qquad \textrm{and} \qquad \mu_k^* =  \mu_{k^\prime},
	\] 
	then $\mathcal{Z}^\rho \mu$ is real on real.  
	\item If for every $1 \leq k \leq M$ there exists $1 \leq k^{\prime\prime} \leq M$ such that
	\[
		-\overline{\zeta_k} = \zeta_{k^{\prime\prime}} \qquad \textrm{and} \qquad \mu_k^* = -\mu_{k^{\prime\prime}}(-\placeholder),
	\]
	then $\mathcal{Z}^\rho \mu$ is imaginary on imaginary.
	\item If for every $1 \leq k \leq M$ there exists $1 \leq k^{\prime\prime} \leq M$ such that
	\[
		-\overline{\zeta_k} = \zeta_{k^{\prime\prime}} \qquad \textrm{and} \qquad \mu_k^* = \mu_{k^{\prime\prime}}(-\placeholder),
	\]
	then $\mathcal{Z}^\rho \mu$ is real on imaginary.
\end{enumerate}
\end{lemma}
\begin{proof}
From \eqref{definition Z operator} we compute that
\[
	\mathcal{Z}^\rho[\mu]^*(\zeta) = \sum_{k=1}^M \frac{1}{2\pi i} \int_{\mathbb{T}} \frac{\mu_k(\overline{\sigma})}{\rho \sigma + \overline{\zeta_k} - \zeta} \rho \, d\sigma = \sum_{k=1}^M \frac{1}{2\pi i} \int_{\mathbb{T}} \frac{\mu_k^*(\sigma)}{\rho \sigma + \overline{\zeta_k} - \zeta} \rho \, d\sigma.
\]
The statements in all three parts are immediate consequences.
\end{proof}

A brute force way to impose these types of symmetry requirements in the bifurcation theory is to work with densities that lie in certain subspaces derived from  Lemma~\ref{symmetry over axes lemma}.  For example, suppose that $\mathcal{Z}^\rho \mu$ is imaginary on imaginary and there is a center $\zeta_k \in i\mathbb{R} \setminus \{0\}$ so that by Lemma~\ref{symmetry over axes lemma} the corresponding density obeys $\mu_k^* = -\mu_k(-\placeholder)$.  Expanding, we find 
\[
	-\mu_k(-\sigma) = \sum_{m \in \mathbb{Z}} (-1)^{m+1} \widehat{\mu}_k^m \sigma^m, \qquad \mu_k^*(\sigma) = \overline{\mu_k(\overline{\sigma})} = \sum_{m \in \mathbb{Z}} \overline{\widehat{\mu}_k^m} \sigma^{m}.
\]
Equating the two, we infer that  
\begin{equation}
  \label{mu even imaginary}
  \mu_k \in C_{\imagimag}^{\ell+\alpha}(\mathbb{T})  \colonequals   \left\{ \varphi \in C^{\ell+\alpha}(\mathbb{T}) : i^m \widehat\varphi_m \in i\mathbb{R} \textrm{ for all } m \in \mathbb{Z} \right\}.
\end{equation} 
On the other hand, to have $\mathcal{Z}^\rho \mu$ be real on imaginary requires instead that for every $\zeta_k \in i \mathbb{R} \setminus \{ 0\}$, the corresponding density satisfies $\mu_k^* = \mu_k(-\placeholder)$, that is 
\begin{equation}
  \label{mu odd imaginary}
  \mu_k \in C_{\realimag}^{\ell+\alpha}(\mathbb{T})  \colonequals   \left\{ \varphi \in C^{\ell+\alpha}(\mathbb{T}) : i^m \widehat\varphi_m \in \mathbb{R}   \textrm{ for all } m \in \mathbb{Z} \right\}.
\end{equation}
By the same reasoning, in order to be real on real, we must have that for every center $\zeta_k \in \mathbb{R} \setminus \{ 0\}$, the density satisfies
\begin{equation}
  \label{mu even real}
  \mu_k \in C_{\realreal}^{\ell+\alpha}(\mathbb{T}) \colonequals   \left\{ \varphi \in C^{\ell+\alpha}(\mathbb{T}) :  \widehat\varphi_m \in \mathbb{R} \textrm{ for all } m \in \mathbb{Z} \right\},
\end{equation}
whereas to be imaginary on real requires membership in the space
\begin{equation}
  \label{nu odd imaginary}
  \mu_k \in C_{\imagreal}^{\ell+\alpha}(\mathbb{T})  \colonequals   \left\{ \varphi \in C^{\ell+\alpha}(\mathbb{T}) :  \widehat\varphi_m \in i\mathbb{R}   \textrm{ for all } m \in \mathbb{Z} \right\}.
\end{equation}
In the special case $\zeta_k = 0$, then $\mathcal{Z}^\rho \mu$ is real on real and imaginary on imaginary only if 
\begin{equation}
  \label{mu even real and imaginary}
  \mu_k \in C_{\realreal,\imagimag}^{\ell+\alpha}(\mathbb{T}) \colonequals  C_{\realreal}^{\ell+\alpha}(\mathbb{T}) \cap C_{\imagimag}^{\ell+\alpha}(\mathbb{T}),
\end{equation}
while to be real on real and real on imaginary we must have that 
\[
  \mu_k \in C_{\realreal,\realimag}^{\ell+\alpha}(\mathbb{T})  \colonequals  C_{\realreal}^{\ell+\alpha}(\mathbb{T}) \cap C_{\realimag}^{\ell+\alpha}(\mathbb{T}),
\]
and to be imaginary on both real and imaginary requires 
\begin{equation}
  \label{mu odd real and imaginary}
  \mu_k \in C_{\imagreal,\imagimag}^{\ell+\alpha}(\mathbb{T})  \colonequals  C_{\imagreal}^{\ell+\alpha}(\mathbb{T}) \cap C_{\imagimag}^{\ell+\alpha}(\mathbb{T}),
\end{equation}

\section{Formulation of the hollow vortex problem} 
\label{formulation section}

In this section, we present a formulation of the hollow vortex problem that is amenable to desingularization via the (global) implicit function theorem.  The main difficulty is to find a way to describe the fluid velocity and the geometry of the domain so that (i) the equation smoothly collapses to the point vortex system at the point of bifurcation, and (ii) the unknowns are elements of a fixed function space.  

\subsection{Layer-potentials for the conformal mapping and complex potential}
\label{conformal map complex potential section}

Recall that our basic strategy begins by resetting the hollow vortex problem \eqref{intro hollow vortex problem} on a circular domain
\[
	\confD_\rho = \confD_\rho(\fullparam) \colonequals  \mathbb{C} \setminus \overline{B_\rho(\zeta_{1}) \cup \cdots \cup B_\rho(\zeta_{M})},
\]
in the $\zeta$-plane.  The task becomes then to construct a conformal mapping $f \in C^{\ell+\alpha}(\overline{\confD_\rho})$
defining the geometry of the hollow vortices, and a complex velocity potential $w \in C^{\ell+\alpha}(\overline{\confD_\rho})$
describing the flow. The physical domain in the $z$-plane will be $\fluidD \colonequals  f(\confD_\rho)$, and the velocity field 
\[
	\overline{\partial_z w}= \frac{\overline{\partial_\zeta w}}{\overline{\partial_\zeta f}} \in C^{\ell-1+\alpha}(\overline{\confD_\rho}).
\]
Observe that in order for this reformulation to be valid, it is necessary that $f$ is both conformal and globally injective on $\overline{\confD_\rho}$. We further require that $f$ is asymptotic to the identity map as $\zeta \to \infty$.  The geometric parameters $(\rho,\fullparam)$ must also lie in the neighborhood $\mathcal{U}$ introduced in \eqref{definition of U neighborhood}.   

It is advantageous to use a double-layer potential representation for both $f$ and $w$, so that they can be written in terms of their traces on the boundary components of $\confD_\rho$.  For the conformal mapping, we impose the ansatz
\begin{equation}
  \label{f ansatz}
  f = \id + \rho^2 \mathcal{Z}^\rho \mu,
\end{equation}
where recall the operator $\mathcal{Z}^\rho = \mathcal{Z}^\rho(\fullparam)$ is defined by \eqref{definition Z operator}, and the real-valued densities $\mu = (\mu_1, \ldots, \mu_M) \in \mathring{C}^{\ell+\alpha}(\mathbb{T})^M$ are to be determined.  The factor of $\rho^2$ anticipates the scaling of the governing equations as $\rho \searrow 0$.  In the interest of readability, the dependence of $\mathcal{Z}^\rho$ on $(\param, \param^\prime)$ will be suppressed when there is no risk of confusion.

We also wish to arrange that $w$ converges (in a sufficiently smooth way) to the velocity potential $w^0$ for the point vortex problem in the plane
\begin{equation}\label{w0 formula}
  w^0 = w^0(\fullparam)(\zeta) \colonequals  \sum_{k=1}^M \frac{\gamma_k}{2\pi i} \log{\left(\zeta - \zeta_k \right)}.
\end{equation}
An important feature of the problem is that when $\Omega \neq 0$, the relative (complex) velocity field will no longer be holomorphic, as viewed in the rotating frame, it has a constant vorticity $\Omega$.  To make this distinction explicit, we denote the relative complex velocity potential by
\[
	W = W(\fullparam) =  w + i \frac{\Omega}{2} |f|^2 -cf 
\]
and likewise for $W^0$.

In accordance with the Helmholtz--Kirchhoff model, the vortex centers and strengths must satisfy the algebraic constraint \eqref{dWdzeta point vortex algebra}, which implies $W^0$ exhibits the asymptotics
\begin{equation}
  \label{dWdzeta point vortex}
  \partial_\zeta W^0 = \frac{\gamma_k}{2\pi i} \frac{1}{\zeta -\zeta_k} - \sum_{j \neq k} \frac{\gamma_j}{2\pi i} \frac{\zeta - \zeta_k}{(\zeta_j - \zeta_k)^2} + O(|\zeta - \zeta_k|^2) \qquad \textrm{as } \zeta \to \zeta_k.
\end{equation}
Thus, the circulation $\gamma_k$ for the velocity field derived from $w$ around $\partial B_\rho(\zeta_k)$ is independent of $\rho$.  

As with $f$, our approach will be to use a layer-potential representation for $w$ that can be expanded to arbitrary order near $\rho = 0$.  With that in mind, we impose the ansatz
\begin{equation}
  \label{w ansatz}
  w = w^0 + \rho \mathcal{Z}^\rho \nu 
\end{equation}
for real densities $\nu = (\nu_1, \ldots, \nu_M) \in \mathring{C}^{\ell+\alpha}(\mathbb{T})^M$.  Thus the ``trivial'' point vortex solution $w^0$ corresponds to $(\mu,\nu, \rho) = (0,0,0)$.  Notice that this choice implies that $w-w^0$ is single-valued and holomorphic in $\confD_\rho$.

\begin{remark}
\label{negative rho remark}
Observe that from the layer-potential representations for $f$ and $w$ in \eqref{f ansatz} and \eqref{w ansatz}, it follows that $(\mu,\nu,\rho)$ and $(\mu(-\placeholder), -\nu(-\placeholder), -\rho)$ give identical conformal mappings and complex potentials.  Physically, they represent the same hollow vortex configuration, continuing our convention that the corresponding conformal domain $\confD_\rho = \confD_{|\rho|}$ for $|\rho| > 0$ and $\confD_0 \colonequals  \mathbb{C} \setminus \{ \zeta_1, \ldots, \zeta_M\}$.
\end{remark}

\subsection{Kinematic condition}

The kinematic condition \eqref{intro kinematic condition} states that the vortex boundaries are (relative) streamlines, meaning that the relative velocity field is purely tangential along each connected component of $\partial\fluidD$.  Written in the conformal domain, this requirement takes the form
\begin{equation}
  \label{local kinematic condition}
  0 = \realpart{ \left( \tau f_\zeta \left( \frac{w_\zeta}{f_\zeta} + i \Omega \overline{f} -c \right) \right)\Big|_{\zeta = \rho \tau + \zeta_k}} \qquad \textrm{for all } \tau \in \mathbb{T}, \quad k = 1, \ldots, M,
\end{equation}
which follows from the fact that $\tau f_\zeta$ is the normal vector along the vortex boundary, while $w_\zeta/f_\zeta + i \Omega \overline f - c$ is the conjugate of the relative velocity field there.  Using the layer-potential representations for $f$ and $w$, and supposing that $\rho \neq 0$, we then arrive at 
\begin{equation}
  \label{preliminary layer potential kinematic}
  0 = \realpart{\left(  \tau \left( w_\zeta^0(\rho \tau + \zeta_k) + \mathcal{Z}_k^\rho\nu^\prime  +  (1+ \rho \mathcal{Z}_k^\rho\mu^\prime) \left(  i \Omega\overline{(\zeta_k + \rho \tau + \rho^2 \mathcal{Z}_k^\rho\mu)} -c \right) \right) \right)}
\end{equation}
for all $\tau \in \mathbb{T}$ and $k = 1, \ldots, M$.  Recall that $\mathcal{Z}_k^\rho$ is the integral operator defined in \eqref{trace Z operator}. Using \eqref{w0 formula} and the asymptotics for $W_\zeta^0$ in \eqref{dWdzeta point vortex}, the terms in \eqref{preliminary layer potential kinematic} can be regrouped as
\begin{equation}
 \label{kinematic condition}
  \begin{aligned}
    0 & = \realpart{\left( \tau \left( \mathcal{Z}_k^\rho \nu^\prime + \Vrho_k^\rho + ( i \Omega \overline{( \zeta_k + \rho \tau)} -c) \rho \mathcal{Z}_k^\rho \mu^\prime + i\Omega \rho^2 \overline{\mathcal{Z}_k^\rho \mu}  + i \Omega \rho^3 \mathcal{Z}_k^\rho [\mu^\prime] \overline{\mathcal{Z}_k^\rho\mu} \right) \right)} 
  \end{aligned}
\end{equation}
where
\begin{equation}
 \label{definition Lambda}
  \begin{aligned}
    \Vrho_k^\rho = \Vrho_k^\rho(\fullparam)(\tau) & \colonequals   \left( \partial_\zeta W^0 - \frac{\gamma_k}{2 \pi i} \frac{1}{\zeta - \zeta_k} \right) \Big|_{\zeta = \zeta_k + \rho \tau}  \\
    &= \mathcal{V}_k( \fullparam) + \sum_{m=1}^\infty \sum_{j \neq k} \frac{(-1)^m}{m!} \frac{\gamma_j}{2 \pi i} \frac{\rho^m \tau^m}{(\zeta_j - \zeta_k)^{m+1}}.
  \end{aligned}
\end{equation}
Observe that $(\rho,\fullparam) \mapsto \Vrho_k^\rho(\fullparam)$ is real analytic $\mathcal{U} \to C^{\ell+\alpha}(\mathbb{T})$ for any $\ell \geq 0$.  
As it must be, $(\nu,\mu,\rho) = (0,0,0)$ is a trivial solution to \eqref{kinematic condition} when $(\param,\param^\prime)$ is steady vortex configuration.  Lastly, we note that the right-hand side of \eqref{kinematic condition} necessarily lies in $\mathring{C}^{\ell-1+\alpha}$ when $\mu, \nu \in \mathring{C}^{\ell+\alpha}$.

\subsection{Bernoulli condition}

The dynamic condition \eqref{intro dynamic condition} requires that the pressure along each hollow vortex boundary is constant. In the absence of body forces and surface tension, this is equivalent by Bernoulli's principle to the relative velocity field having constant modulus on each connected component of $\partial \fluidD$.  Thus
\begin{equation}
  \label{bernoulli condition patch}
   \left| \frac{w_\zeta}{f_\zeta} + i \Omega \overline{f} -c \right|^2 = q_k^2 \qquad \textrm{on } \partial B_\rho(\zeta_k) 
\end{equation}
for some constant vector $q = (q_1, \ldots, q_M) \in \mathbb{R}^M$. 

Recall that if $\gamma_k \neq 0$, then the trace of $|w_\zeta|^2$ (and hence $q_k^2$) should diverge like $O(1/\rho^2)$ as $\rho \searrow 0$ due to \eqref{dWdzeta point vortex}.    With that in mind, we multiply the left-hand side of \eqref{bernoulli condition patch} by $\rho^2$, evaluate at $\zeta = \rho \tau + \zeta_k$, and then expand to find
\begin{equation}
 \label{definition Bernoulli operator}
  \begin{aligned}
    & \rho^2 \left| \frac{ w_\zeta}{f_\zeta} + i \Omega \overline{f} -c \right|^2  \\
    & \quad =  \frac{\left|\dfrac{\gamma_k}{2\pi i \tau} + \rho \left( \mathcal{Z}_k^\rho \nu^\prime + \Vrho_k^\rho \right) + \rho^2    \left(  i \Omega \overline{( \zeta_k + \rho \tau)} -c\right)   \mathcal{Z}_k^\rho \mu^\prime + \rho^3  i \Omega \left( \overline{\mathcal{Z}_k^\rho \mu}  + \rho \mathcal{Z}_k^\rho [\mu^\prime] \overline{\mathcal{Z}_k^\rho\mu} \right) \right|^2}{\left|1+ \rho \mathcal{Z}_k^\rho \mu^\prime\right|^2}   \\
    & \quad \equalscolon  \frac{\gamma_k^2}{4\pi^2} + \rho \mathcal{B}_k^\rho.
  \end{aligned}
\end{equation}
This defines an explicit operator $\mathcal{B}_k^\rho = \mathcal{B}_k^\rho(\mu, \nu, \fullparam)$ that is real analytic in a neighborhood of $(\mu,\nu,\rho) = (0,0,0)$.  It follows that the Bernoulli condition can be formulated in terms of the densities via
the requirement
\begin{equation}
  \label{reformulated bernoulli condition} 
  \mathcal{B}^\rho(\mu, \nu, \fullparam) = Q,
\end{equation}
where $Q = (Q_1,\ldots, Q_M) \in \mathbb{R}^M$ are unknown constants.  Notice that \eqref{reformulated bernoulli condition} is equivalent to \eqref{bernoulli condition patch} for $\rho > 0$, but it is well-defined even for $|\rho| \ll 1$.  Moreover,
\begin{equation}
  \label{linear Bernoulli operator}
  \mathcal{B}_k^\rho(\mu,\nu,\fullparam) = \frac{\gamma_k^2}{2  \pi^2} \realpart{\left( \frac{2\pi i \tau}{\gamma_k} \left( \mathcal{V}_k(\fullparam) + \mathcal{C}\nu^\prime \right) - \mathcal{C} \mu^\prime\right)} + O(\rho) \qquad \textrm{in } C^{\ell-1+\alpha}(\mathbb{T}) 
\end{equation}
for any $\mu, \nu \in C^{\ell+\alpha}(\mathbb{T})$ and $\ell \geq 1$.

\subsection{Abstract operator equation}

Assume now that $\fullparam_0$ is a vortex configuration that is non-degenerate in the sense of Definition~\ref{definition nondegenerate configuration}.  We can then decompose the parameter space $\mathbb{P} = \Pspace \times \Pspace^\prime$ so that $\fullparam_0 = (\param_0,\param_0^\prime)$ and \eqref{definition non-degnerate subspace} holds.  For the local and global hollow vortex families we construct, we fix the parameter values $\param^\prime = \param_0^\prime$ and allow $\param$ to vary.  For notational simplicity, the functions $\Vrho^\rho$ and $\mathcal{B}$ will now be considered to have domain $\Pspace$.  Likewise, the neighborhoods $\mathcal{U}$ and $\mathcal{U}_\delta$ are redefined in the obvious way.    

Because $(\mu_k, \nu_k, Q_k)$ describes the behavior near the $k$-th vortex center, it will be convenient to introduce the space
\[  
	(\mu, \nu, Q) \in \Wspace \colonequals   \mathring{C}^{\ell+\alpha}(\mathbb{T})^M \times  \mathring{C}^{\ell+\alpha}(\mathbb{T})^M \times \mathbb{R}^M.
\]
Here $\ell \geq 1$ is fixed.  The regularity is of little importance to the local bifurcation argument, but it will enter into the global analysis.  The hollow vortex problem \eqref{kinematic condition} and \eqref{reformulated bernoulli condition} can be written as the abstract operator equation
\[
	\F( u;\, \rho) = 0,
\]
for 
\[
	\F = (\A, \B) \colon \mathcal{O} \subset \Xspace \times \mathbb{R} \to \Yspace
\]
the real-analytic mapping between the spaces
\begin{align*}
	\Xspace   \colonequals  \Wspace \times \Pspace, \qquad
	\Yspace  \colonequals   \mathring{C}^{\ell-1+\alpha}(\mathbb{T})^M \times C^{\ell-1+\alpha}(\mathbb{T})^M,
\end{align*}
whose components $\A_k$ and $\B_k$ enforce the kinematic and Bernoulli conditions on the $k$-th vortex boundary, respectively.  Explicitly, they are     
\begin{equation}
 \label{abstract operator}
  \begin{aligned}
    \mathscr{A}_k(u, \rho) 
    & \colonequals  \realpart{\left( \tau \left( \mathcal{Z}_k^\rho \nu^\prime + \Vrho_k^\rho + ( i \Omega \overline{( \zeta_k + \rho \tau)} -c) \rho \mathcal{Z}_k^\rho \mu^\prime + i\Omega \rho^2 \overline{\mathcal{Z}_k^\rho \mu}  + i \Omega \rho^3 \mathcal{Z}_k^\rho [\mu^\prime] \overline{\mathcal{Z}_k^\rho\mu} \right) \right)}   \\
    \mathscr{B}_k(u, \rho)	& \colonequals  \mathcal{B}_k^\rho(\mu,\nu,\param) - Q_k,
  \end{aligned}
\end{equation}
where recall that the function $\Vrho^\rho$ is given by \eqref{definition Lambda} and the operator $\mathcal{B}^\rho$ is defined in \eqref{definition Bernoulli operator}.  
To define the domain $\mathcal{O}$ of $\F$, we first define $\widetilde{\mathcal{O}}$ by
\begin{equation}
  \label{definition O delta}
  \mathcal{O}_\delta \colonequals  \left\{ (\mu,\nu,Q,\param,\rho) \in \Wspace \times \mathcal{U}_\delta : \inf_k \inf_{\mathbb{T}}{ |1+ \rho \mathcal{Z}_k^\rho \mu_k^\prime |} > \delta \right\}, \quad \widetilde{\mathcal{O}} \colonequals  \bigcup_{\delta > 0} \mathcal{O}_\delta,
\end{equation}
where $\mathcal{U}_\delta$ was defined in \eqref{definition of U neighborhood}.  Membership in $\mathcal{O}$ ensures that the corresponding mapping $f$ constructed via \eqref{f ansatz} has nonvanishing derivative on $\partial \confD_\rho$. The trivial solution corresponding to the point vortex configuration now takes the form
\[
	(\mu,\nu,Q,\param,\rho)  = (0,0,0,\param_0,0) \equalscolon  (u^0,0) \in \mathcal{O}.
\]
In order for $(u,\rho) \in \F^{-1}(0)$ to represent a physical solution to the problem, it is additionally necessary that $f$ is univalent on $\overline{\confD_\rho}$.  This will follow for the small hollow vortex solutions by construction, as $f$ will be near identity.  For the global solutions, we will use a continuity argument based on the argument principle.

\begin{remark}
\label{negative rho operator equation remark}
When $(u, \rho) = (u^0,0)$, the operator equation coincides exactly with the point vortex problem $\mathcal{V}(\param, \param^\prime) = 0$.  It is not hard hard to verify, moreover, that $\F$ exhibits the symmetry 
  \begin{equation}
    \label{symmetries abstract operator}
    \F(\mu, \, \nu, \, Q, \,  \param; \,  \rho) = \F(\mu(-\placeholder), \, -\nu(-\placeholder),  \, Q, \, \param; \, -\rho).
  \end{equation}
 Recalling Remark~\ref{negative rho remark}, we see that these represent the same physical solution.  
\end{remark}

\begin{lemma}[Uniqueness]
\label{uniqueness lemma}
Suppose that $(u,0) = (\mu,\nu,Q,\param,0) \in \mathcal{O}$ satisfies $\F(u; 0) = 0$.  Then $\mu,\nu,Q = 0$ and $\mathcal{V}(\param,\param_0^\prime) = 0$.  That is, $(u,0)$ represents a steady point vortex configuration.
\end{lemma}
\begin{proof}
Let $(u,0) \in \F^{-1}(0)$ be given.  From the kinematic condition $\A_k(u;0) = 0$ we have
\[
	\realpart{( \tau \mathcal{C}\nu_k^\prime )} = -\realpart{\left( \tau \Vrho_k(\param,\param_0^\prime)\right)} \qquad \textrm{for } k = 1, \ldots, M,
\]
which can be inverted using \eqref{multiplier formulas} to find 
\[
	\nu_k(\tau) =  \tau \Vrho_k(\param,\param_0^\prime) + \frac{1}{\tau} \overline{\Vrho_k(\param,\param_0^\prime)}.
\]
Similarly, the Bernoulli condition $\B_k(u;0) = 0$ takes the explicit form
\[
	0 = \frac{\gamma_k^2}{2\pi^2} \realpart{\left( \frac{2\pi i \tau}{\gamma_k} \left( \mathcal{C} \nu_k^\prime + \Vrho_k(\param,\param_0^\prime) \right) - \mathcal{C} \mu_k^\prime \right)} - Q_k.
\]
From our formula for $\nu_k$, we then find that
\[
	\realpart{\left( i \tau \mathcal{C} \nu_k^\prime \right)} = \frac{i \tau}{2} \Vrho_k(\param,\param_0^\prime) - \frac{i}{2\tau} \overline{\Vrho_k(\param,\param_0^\prime)} = \realpart{\left( i \tau \Vrho_k(\param,\param_0^\prime) \right)},
\]
and thus the Bernoulli condition can be reexpressed as
\[
	0 = \frac{2 \gamma_k}{\pi} \realpart{\left( i \tau \Vrho_k(\param,\param_0^\prime)\right)} - \frac{\gamma_k^2}{2\pi^2} \realpart{\mathcal{C} \mu_k^\prime} - Q_k.
\]
Finally, applying the projections $\proj_0$, $\proj_1$, and $\proj_{>1}$, we conclude that each of the terms above must vanish, which completes the proof.
\end{proof}

\section{Identities for the hollow vortex problem}
\label{hollow vortex identities section}

The purpose of this section is to find a set of identities for the hollow vortex problem that generalize \eqref{pv identities} for point vortices.  As in the proof of Proposition~\ref{point vortex manifold proposition}, these will allow us to overcome the lack of surjectivity of the linearized operator $\F_u$ at the point vortex, and thereby infer the existence of neighboring hollow vortex solutions.  

One way to derive \eqref{pv identities} on physical grounds is to exploit the Hamiltonian formulation of the dynamical point vortex problem.  To find corresponding identities for hollow vortices, it will also be helpful to reframe the system in Hamiltonian terms.  It is sufficient to do this rather formally, and then verify the resulting identities at the end of the analysis.   In fact, the identities that we find recover \eqref{pv identities} in the point vortex limit; see Appendix~\ref{identities appendix}.

\subsection{Formal variational arguments}

Suppose that the conformal mapping $f$ is injective, so that we can work in the physical domain $\fluidD \colonequals  f(\confD_\rho)$. We consider the Hamiltonian $H$ to formally consist of the kinetic energy $E$, the linear momentum $p$, and the angular momentum $L$, where
\begin{align*}
E \colonequals  & \frac12 \int_{\fluidD} |w_z|^2\,dz = \frac12 \int_{\confD_\rho} \left| w_z(f(\zeta)) \right|^2 |f_\zeta|^2 \,d\zeta, \\
p \colonequals  & \realpart \int_{\fluidD} w_z \,dz = \realpart\int_{\confD_\rho} w_z(f(\zeta)) |f_\zeta|^2 \,d\zeta, \\
L \colonequals  & \imagpart \int_{\fluidD} zw_z \,dz = \imagpart\int_{\confD_\rho} f(\zeta) w_z(f(\zeta)) |f_\zeta|^2 \,d\zeta.
\end{align*}
As we will be taking variations, for the moment we do not concern ourselves with whether these integrals are convergent, and we do not distinguish between the angular momentum and the excess angular momentum defined in \eqref{definition angular momentum}. We also depart temporarily from the convention elsewhere in the paper and view $w$ as a function of $z$ rather than $\zeta$. It is also worth noting that this is one of the only places in the paper where the absence of a body force such as gravity simplifies the analysis in a nontrivial way.    

The Hamiltonian depends on $w$ and $f$ and reads
\begin{equation*}
H = E - cp - \Omega L.
\end{equation*}
Using dots to denote variations and $\delta$ the Gateaux derivative, one formally computes 
\begin{equation}
\label{formal variations E p L}
  \begin{aligned}
    \delta E(w,f)[(\dot w, \dot f)] = & -\frac12 \imagpart \int_{\p \confD_\rho} \LC w_\zeta \overline{\dot w} + |w_z|^2 f_\zeta \overline{\dot f} \RC \,d\zeta, \\
    \delta p(w,f)[(\dot w, \dot f)] = & -\frac12 \imagpart \int_{\p \confD_\rho} \LC f_\zeta \overline{\dot w} + \LC w_z + \overline{w_z} \RC f_\zeta \overline{\dot f} \RC \,d\zeta, \\
    \delta L(w,f)[(\dot w, \dot f)] = & \imagpart \int_{\confD_\rho} \left( f |f_\zeta|^2 \dot w_z - f_\zeta \overline{(f w_z \dot f)_\zeta} + f w_z f_\zeta \overline{\dot f_\zeta} \right) \,d\zeta.
  \end{aligned}
\end{equation}
Set $U \colonequals  w_z + i \Omega \overline f - c$ to be the conjugate of the relative velocity. The dynamic boundary conditions \eqref{bernoulli condition patch} and kinematic boundary conditions \eqref{local kinematic condition} on $\p B_\rho(\zeta_k)$ become
\begin{equation*}
\left\{
\begin{aligned}
	\Func_1 &\colonequals  \realpart\LB (\zeta - \zeta_k) f_\zeta U \RB = 0 \\
	\Func_2 & \colonequals  |U|^2 - q_k^2 = 0.
\end{aligned} \right.
\qquad \textrm{on } \partial B_\rho(\zeta_k).
\end{equation*}
Observe, moreover, that from the first of these equations it follows that
\begin{equation}\label{bar xi}
	U = \frac{2\Func_1}{(\zeta - \zeta_k) f_\zeta} - \frac{\overline{(\zeta - \zeta_k) f_\zeta}}{(\zeta - \zeta_k) f_\zeta} \overline U \qquad \textrm{on } \partial B_\rho(\zeta_k).
\end{equation}
Recalling the analysis in Section~\ref{point vortex section}, we should think of $\mathcal{F} = (\mathcal{F}_1, \mathcal{F}_2)$ as playing an analogous role to the components of $\mathcal{V}$.  The main task is thus to find the correct generalization of \eqref{pv identities}, which should naturally be expressed in terms of $\mathcal{F}$.  

\subsection{Translating equilibria}\label{subsec transl id}
Consider first the case when $\Omega = 0$ and $c \in \R$. We study variations of the form 
\begin{equation*}
\dot f  = \sigma, \qquad \dot w = -w_z \dot f = -\sigma w_z,
\end{equation*}
where $\sigma \in \mathbb C$ is constant.  Observe that these corresponds to the invariance of the system under translations group $z \mapsto z + s \sigma$ for $s \in \mathbb{R}$. We then find that
\begin{align*}
\left\{\begin{aligned}
\delta E(w,f)[(\dot w, \dot f)] = & -\frac12 \imagpart \int_{\p \confD_\rho} \LB -\overline\sigma \overline{w_z} w_\zeta + \overline\sigma |w_z|^2 f_\zeta \RB\,d\zeta, \\
\delta p(w,f)[(\dot w, \dot f)] = & -\frac12 \imagpart \int_{\p \confD_\rho} \LB -\overline\sigma \overline{w_z} f_\zeta + \overline\sigma \LC w_z + \overline{w_z} \RC f_\zeta \RB \,d\zeta.
\end{aligned}\right.
\end{align*}

Taking $\sigma = 1$, we compute the corresponding directional derivative of the Hamiltonian  $H$ to be
\begin{align*}
\delta_f H(w,f)[1]  = & -\frac12 \imagpart \int_{\p \confD_\rho} \LC |U|^2 - c^2 \RC f_\zeta \,d\zeta
= -\frac12 \imagpart \LC \int_{\p \confD_\rho} \Func_2 f_\zeta \,d\zeta + \sum_k \int_{\p B_\rho(\zeta_k)} (c^2 + q_k^2) f_\zeta \,d\zeta \RC \\
= & -\frac12 \imagpart \int_{\p \confD_\rho} \Func_2 f_\zeta \,d\zeta.
\end{align*}
Note here that because $\confD_\rho$ is an exterior domain, the positive orientation of $\partial \confD_\rho$ is clockwise, while as usual we take the positive orientation of $\partial B_\rho(\zeta_k)$ to be counter clockwise.  Recalling the definition of $\Func_1$ and using \eqref{bar xi}, we then find
\begin{align*}
\int_{\p B_\rho(\zeta_k)} |U|^2 f_\zeta \,d\zeta = & \int_{\p B_\rho(\zeta_k)} U \overline U f_\zeta \,d\zeta = \int_{\p B_\rho(\zeta_k)} \LB \frac{2 \overline\xi \Func_1}{\zeta - \zeta_k} - \frac{\overline{(\zeta - \zeta_k) f_\zeta U^2}}{\zeta - \zeta_k} \RB \,d\zeta.
\end{align*}
Since on $\p B_\rho(\zeta_k)$ we have $d\overline\zeta = - \frac{\overline{\zeta - \zeta_k}}{\zeta - \zeta_k}\,d\zeta$, it follows that
\begin{align*}
-\int_{\p B_\rho(\zeta_k)} \frac{\overline{(\zeta - \zeta_k) f_\zeta U^2}}{\zeta - \zeta_k} \,d\zeta = &  \int_{\p B_\rho(\zeta_k)} \overline{U^2 f_\zeta} \,d\overline\zeta 
= \overline{\int_{\p B_\rho(\zeta_k)} \frac{w_\zeta^2}{f_\zeta}  \,d\zeta} - 2c \overline{\int_{\p B_\rho(\zeta_k)} f_\zeta U \,d\zeta}.
\end{align*}
Using \eqref{bar xi} again yields
\begin{align*}
\overline{\int_{\p B_\rho(\zeta_k)} f_\zeta U \,d\zeta} = \overline{\int_{\p B_\rho(\zeta_k)} \frac{2\Func_1}{\zeta - \zeta_k}\,d\zeta} + \int_{\p B_\rho(\zeta_k)} f_\zeta U \,d\zeta = -4\pi i \Func_1 + \int_{\p B_\rho(\zeta_k)} f_\zeta U \,d\zeta,
\end{align*}
and hence
\[
\imagpart {\int_{\p B_\rho(\zeta_k)} f_\zeta U \,d\zeta} = 2\pi \Func_1.
\]

Combining these facts, we find
\begin{align*}
\imagpart \int_{\p B_\rho(\zeta_k)} |U|^2 f_\zeta \,d\zeta = \imagpart \int_{\p B_\rho(\zeta_k)} \frac{2 \overline{w_z} \Func_1}{\zeta - \zeta_k} \,d\zeta + \imagpart \overline{\int_{\p B_\rho(\zeta_k)} \frac{w_\zeta^2}{f_\zeta}  \,d\zeta}.
\end{align*}
The asymptotics of $f$ in \eqref{f ansatz} and $w$ in \eqref{w ansatz} imply that as $\zeta \to \infty$
\begin{equation}\label{fw asymptotics}
f(\zeta) = \zeta + O\LC \tfrac{1}{\zeta} \RC, \quad f_\zeta(\zeta)= 1 + O\LC \tfrac{1}{\zeta^2} \RC, \quad w_\zeta(\zeta)= w^0_\zeta(\zeta) + O\LC \tfrac{1}{\zeta^2} \RC,
\end{equation}
and thus
\[
\int_{\p \confD_\rho} \frac{w_\zeta^2}{f_\zeta}  \,d\zeta = 0.
\]
Therefore, we arrive at the following identity
\begin{equation}\label{id transl}
{\frac12 \imagpart \int_{\p \confD_\rho} \Func_2 f_\zeta \,d\zeta - \imagpart \int_{\p \confD_\rho} \frac{\overline{w_\zeta} \Func_1}{(\zeta - \zeta_k) \overline{f_\zeta}} \,d\zeta = 0.}
\end{equation}

\subsection{Rotating equilibria}\label{subsec rot id}

Similarly, for rotating configurations ($c = 0$, $\Omega \neq 0$), we invoke the invariance of the system under rotations and scalings $z \mapsto e^{s \sigma} z$ with $\sigma \in \mathbb{C}$ constant and $s \in \mathbb{R}$.  This suggests we consider variations of the form
\begin{equation*}
\dot f  = \sigma f, \qquad \dot w = -w_z \dot f = -\sigma w_z f.
\end{equation*}
Evaluating \eqref{formal variations E p L} with this choice of $(\dot w, \dot f)$ and applying the complex Green's theorem leads to
\begin{align*}
\left\{\begin{aligned}
\delta E(w,f)[(\dot w, \dot f)] = & -\frac12 \imagpart \int_{\p \confD_\rho} \LB -\overline\sigma |w_z|^2 f_\zeta \overline f + \overline\sigma |w_z|^2 f_\zeta \overline f \RB\,d\zeta, \\
\delta L(w,f)[(\dot w, \dot f)] 
= & -\imagpart \int_{\confD_\rho} \sigma f \overline{f_\zeta} \LC w_z f \RC_\zeta \,d\zeta + \frac12 \realpart \int_{\p \confD_\rho} \overline f f_\zeta \LC \overline{\sigma w_z f} + \sigma w_z f \RC \,d\zeta.
\end{aligned}\right.
\end{align*}

Next we compute the variation of the Hamiltonian in the direction $\sigma = i$ and re-express it in terms of $\Func_2$.  First, observe that
\begin{align*}
\delta_f H(w,f)[i f]  & = \frac12 \realpart \int_{\p \confD_\rho} \LB |U|^2 - \Omega^2|f|^2 \RB \overline f f_\zeta \,d\zeta \\
& = \frac12 \realpart \LC \int_{\p \confD_\rho} \Func_2 \overline f f_\zeta \,d\zeta - \sum_k \int_{\p B_\rho(\zeta_k)} \LC \Omega^2|f|^2 - q_k^2 \RC \overline f f_\zeta \,d\zeta \RC,
\end{align*}
where again the orientation of $\partial B_\rho(\zeta_k)$ is counter clockwise.  Consider now the two last terms on the right-hand side above.  Using the complex Green's theorem, we compute that
\begin{align*}
  \int_{\p B_\rho(\zeta_k)} \LC \Omega^2|f|^2 - q_k^2 \RC \overline f f_\zeta \,d\zeta 
  = 
  \int_{\Gamma_k} \LC \Omega^2 z \overline z^2 - q_k^2 \overline z \RC  \,dz
  = 2i \int_{R_k} \LC 2\Omega^2 |z|^2 - q_k^2 \RC \,dz
\end{align*}
is purely imaginary, where here $R_k$ is the region enclosed by $\Gamma_k$. Thus 
\[
\delta_f H(w,f)[i f] = \frac12 \realpart \int_{\p \confD_\rho} |U|^2 \overline f f_\zeta \,d\zeta = \frac12 \realpart \int_{\p \confD_\rho} \Func_2 \overline f f_\zeta \,d\zeta.
\]

As in the previous subsection, we will rewrite this identity using $\Func_1$.  Observe that
\begin{align*}
\frac12 \realpart \int_{\p B_\rho(\zeta_k)} |U|^2 \overline f f_\zeta \,d\zeta & = \frac12 \realpart \int_{\p B_\rho(\zeta_k)} \LB \frac{2 \overline{U f} \Func_1}{\zeta - \zeta_k} - \frac{\overline{(\zeta - \zeta_k) f f_\zeta U^2}}{\zeta - \zeta_k} \RB \,d\zeta \\
& = \realpart \int_{\p B_\rho(\zeta_k)} \frac{\overline{w_z f} \Func_1}{\zeta - \zeta_k} \,d\zeta + \frac12 \realpart \overline{\int_{\p B_\rho(\zeta_k)} \frac{w_\zeta^2 f}{f_\zeta} \,d\zeta}.
\end{align*}
From this we derive the corresponding identity
\begin{equation}\label{id rot 1}
\frac12 \realpart \int_{\p \confD_\rho} \Func_2 \overline f f_\zeta \,d\zeta - \realpart \int_{\p \confD_\rho} \frac{\overline{w_\zeta f} \Func_1}{(\zeta - \zeta_k) \overline{f_\zeta}} \,d\zeta = \frac12 \realpart \overline{\int_{\p \confD_\rho} \frac{w_\zeta^2 f}{f_\zeta} \,d\zeta}.
\end{equation}
To compute the right-hand side above, recall the asymptotics \eqref{fw asymptotics}, which implies that as $\zeta \to \infty$,
\begin{align*}
\frac{w_\zeta^2 f}{f_\zeta} & = (w^0_\zeta)^2 \zeta + O \LC \tfrac{1}{\zeta^2} \RC = \LC \sum^M_{k=1} \frac{\gamma_k}{2\pi i (\zeta - \zeta_k)} \RC^2 \zeta + O \LC \tfrac{1}{\zeta^2} \RC 
= \sum^M_{k, j} \frac{\gamma_k \gamma_j}{(2\pi i)^2 (\zeta - \zeta_j)} + O \LC \tfrac{1}{\zeta^2} \RC.
\end{align*}
We can then use the residual theorem to compute
\[
\realpart{ \int_{\p \confD_\rho} \frac{w_\zeta^2 f}{f_\zeta} \,d\zeta} = -\sum_k \realpart{ \int_{\partial B_\rho(\zeta_k)} \frac{w_\zeta^2 f}{f_\zeta} \,d\zeta } =  \realpart{ \frac{1}{2\pi i} \LC \sum_k \gamma_k \RC^2} = 0.
\]
Therefore \eqref{id rot 1} becomes
\begin{equation}\label{id rot}
{\frac12 \realpart \int_{\p \confD_\rho} \Func_2 \overline f f_\zeta \,d\zeta - \realpart \int_{\p \confD_\rho} \frac{\overline{w_\zeta f} \Func_1}{(\zeta - \zeta_k) \overline{f_\zeta}} \,d\zeta = 0.}
\end{equation}

\subsection{Stationary equilibria}\label{subsec stat id}
Finally we consider the case for stationary equilibria ($c = \Omega = 0$). From similar arguments to the previous two subsections, we are ultimately able to obtain the following identities
\begin{equation}\label{id stationary1-2}
{\frac12 \int_{\p \confD_\rho} \Func_2 f_\zeta \,d\zeta - \int_{\p \confD_\rho} \frac{\overline{w_\zeta} \Func_1}{(\zeta - \zeta_k) \overline{f_\zeta}} \,d\zeta = 0,}
\quad
{\frac12 \realpart \int_{\p \confD_\rho} \Func_2 \overline f f_\zeta \,d\zeta - \realpart \int_{\p \confD_\rho} \frac{\overline{w_\zeta f} \Func_1}{(\zeta - \zeta_k) \overline{f_\zeta}} \,d\zeta = 0.}
\end{equation}

\subsection{The mapping \texorpdfstring{$\phi$}{phi}}
\label{hollow vortex phi section}
What we have showed in Section \ref{subsec transl id}--\ref{subsec stat id} is that a hollow vortex solution $(u, \rho) \in \Xspace \times \mathbb{R}$ would automatically satisfy certain identities, which causes degeneracy of $D_u \F$. For this, we introduce the mappings
\begin{align*}
	\phi_{\rm t} &= \phi_{\rm t}(A,B, u, \rho) : \Yspace \times \Xspace \times \mathbb{R} \to \mathbb{R}, \\
	\phi_{\rm r} &= \phi_{\rm r}(A,B, u, \rho) : \Yspace \times \Xspace \times \mathbb{R} \to \mathbb{R}, \\
	 \phi_{\rm s} & = \phi_{\rm s} (A,B,u,\rho) : \Yspace \times \Xspace \times \mathbb{R} \to \mathbb{C} \times \mathbb{R},
\end{align*}
given by
\begin{equation}
\label{hv phi functions}
  \begin{aligned}
    \phi_{\rm t} & \colonequals  \imagpart \sum_k \int_{\mathbb T} \LC  \frac12 B_k (1 + \rho \mathcal{Z}^\rho_k \mu^\prime)- \frac{\overline{\rho\tau (w^0_\zeta + \mathcal{Z}^\rho_k \nu^\prime)}}{\overline{(1 + \rho \mathcal{Z}^\rho_k \mu^\prime)}} A_k \RC \overline{(\zeta_k + \rho\tau + \rho^2 \mathcal{Z}^\rho_k \mu)} \,d\tau,\\
    \phi_{\rm r} &\colonequals  \realpart \sum_k \int_{\mathbb T} \LC  \frac12 B_k (1 + \rho \mathcal{Z}^\rho_k \mu^\prime)- \frac{\overline{\rho\tau (w^0_\zeta + \mathcal{Z}^\rho_k \nu^\prime)}}{\overline{(1 + \rho \mathcal{Z}^\rho_k \mu^\prime)}} A_k \RC \overline{(\zeta_k + \rho\tau + \rho^2 \mathcal{Z}^\rho_k \mu)} \,d\tau, \\
    \phi_{\rm s}  &\colonequals  \begin{pmatrix}
      \displaystyle \sum_k \int_{\mathbb T} \LC  \frac12 B_k (1 + \rho \mathcal{Z}^\rho_k \mu^\prime)- \frac{\overline{\rho\tau (w^0_\zeta + \mathcal{Z}^\rho_k \nu^\prime)}}{\overline{(1 + \rho \mathcal{Z}^\rho_k \mu^\prime)}} A_k \RC \,d\tau \\
      \displaystyle \realpart \sum_k \int_{\mathbb T} \LC  \frac12 B_k (1 + \rho \mathcal{Z}^\rho_k \mu^\prime)- \frac{\overline{\rho\tau (w^0_\zeta + \mathcal{Z}^\rho_k \nu^\prime)}}{\overline{(1 + \rho \mathcal{Z}^\rho_k \mu^\prime)}} A_k \RC \overline{(\zeta_k + \rho\tau + \rho^2 \mathcal{Z}^\rho_k \mu)} \,d\tau
    \end{pmatrix}.
  \end{aligned}
\end{equation}
Notice that these are found writing \eqref{id transl}, \eqref{id rot}, and \eqref{id stationary1-2} in terms of the densities $(\mu,\nu)$.  They will serve an identical purpose to the mappings $\phi_{\rm t}^0$, $\phi_{\rm r}^0$, and $\phi_{\rm s}^0$ from \eqref{pv phi functions}.  The main result of this section is then the following.

\begin{theorem}[Hollow vortex identities]
\label{hollow vortex identities theorem}
	Fix $(c,\Omega)$ and let $\phi$ stand for $\phi_{\rm t}$, $\phi_{\rm r}$, or $\phi_{\rm s}$ depending on whether $c \neq 0, \Omega = 0$, or $c = 0, \Omega \neq 0$, or $c = \Omega = 0$, respectively.  Then
	\begin{equation}
	\label{hollow vortex phi properties}
    \begin{aligned}
      \phi(\A(u,\rho), \B(u,\rho), u, \rho) & = 0 \\
      \phi(0, u, \rho) & = 0  \\
      D_{(A,B)} \phi(0,u,\rho) & \textrm{ is surjective} 
    \end{aligned}
\end{equation}
	for all $(u,\rho) \in \mathcal{O}$ sufficiently small with the same value of $(c,\Omega)$.
\end{theorem}
\begin{proof}
The first equality in \eqref{hollow vortex phi properties} follows from the formal identities \eqref{id transl}, \eqref{id rot}, and \eqref{id stationary1-2}; it can be verified directly from the equation as well.  The second is immediate, so it remains only to prove the surjectivity of $D_{(A,B)} \phi$.

For $\dot A_j \in \mathring{C}^{\ell+\alpha}(\mathbb{T})$, we can write out its Fourier series representation
\[
\dot A_j(\tau) = \sum_{m \ne 0} \widehat{a}_{j,m} \tau^m.
\]
It is easy to check that
\begin{align*}
D_{A_j}\phi_{\rm t}(A, B, u, \rho)\begin{bmatrix} \dot A, \dot B, \dot u, \dot \rho \end{bmatrix} & = -\imagpart \int_{\mathbb T} \LC \frac{i\gamma_j}{2\pi} + O(\rho) \RC \dot A_j \,d\tau = \gamma_j \imagpart \widehat{a}_{j,-1} + O(\rho), \\
D_{A_j}\phi_{\rm r}(A, B, u, \rho)\begin{bmatrix} \dot A, \dot B, \dot u, \dot \rho \end{bmatrix} & = -\realpart \int_{\mathbb T} \LC \frac{i\gamma_j}{2\pi} \overline{\zeta_j} + O(\rho) \RC \dot A_j \,d\tau = \gamma_j \realpart \LC \overline{\zeta_j} \widehat{a}_{j,-1} \RC + O(\rho).
\end{align*}
Therefore for $\rho$ sufficiently small we know that $D_{(A, B)}\phi_{\rm t}(A, B, u, \rho)$ and $D_{(A, B)}\phi_{\rm r}(A, B, u, \rho)$ are both surjective.  Moreover,
\[
D_{A_j}\phi_{\rm s}(A, B, u, \rho)\begin{bmatrix} \dot A, \dot B, \dot u, \dot \rho \end{bmatrix} = \begin{pmatrix}
\displaystyle \gamma_j \widehat{a}_{j,-1} \\ \displaystyle \gamma_j \realpart \LC \overline{\zeta_j} \widehat{a}_{j,-1} \RC \end{pmatrix} + O(\rho).
\]
Identifying the codomain $\mathbb{C} \times \mathbb{R} \cong \mathbb{R}^3$ we find that
\begin{align*}
&D_{(A_1, A_2)}\phi_{\rm s}(A, B, u, \rho)\begin{bmatrix} \dot A, \dot B, \dot u, \dot \rho \end{bmatrix} \\
	 & \qquad = \LB \begin{pmatrix}
		\gamma_1 & 0  & \gamma_2 & 0 \\
		0 & \gamma_1 & 0 & \gamma_2 \\
		\gamma_1 \realpart \zeta_1 & \gamma_1 \imagpart \zeta_1   & \gamma_2 \realpart \zeta_2 & \gamma_2 \imagpart \zeta_2 
	\end{pmatrix} + O(\rho) \RB
	\begin{pmatrix}
		\realpart{\widehat{a}_{1,-1}} \\
		\imagpart{\widehat{a}_{1,-1}}\\
		\realpart{\widehat{a}_{2,-1}} \\
		\imagpart{\widehat{a}_{2,-1}}
	\end{pmatrix}.
\end{align*}
From this computation we see that when $\zeta_1 \ne \zeta_2$ and $\rho$ is sufficiently small, then $D_{(A, B)}\phi_{\rm s}(A, B, u, \rho)$ is surjective. 
\end{proof}

\section{Vortex desingularization}
\label{local section}

Having derived the necessary identities in the previous section, we are now prepared to prove the existence of solutions to the steady hollow vortex problem in a neighborhood of a non-degenerate steady point vortex configuration.  Recall from the discussion in Section~\ref{statement section} the quantity
\[
	S_k \colonequals  -\frac{1}{2} \sum_{j \neq k} \frac{\gamma_j^0}{2 \pi i} \frac{1}{(\zeta_j^0 - \zeta_k^0)^2},
\]
which we will show is the effective straining velocity field experienced by the $k$-th hollow vortex due to the influence of the other vortices.  Our main result on vortex desingularization is the following.

\begin{theorem}[Local desingularization] \label{hollow vortex local bifurcation theorem}
Let $\fullparam_0 = (\param_0, \param_0^\prime)$ be a non-degenerate steady vortex configuration.  There exists $\rho_1 > 0$ and a curve $\km_\loc$ of solutions to the hollow vortex problem that admits the parameterization 
\[
	\km_{\loc} = \left\{  (u^\rho,\rho) = (\mu^\rho, \nu^\rho, Q^\rho, \param^\rho, \rho)  \in \mathcal{O} : \quad |\rho| < \rho_1 \right\} \subset \F^{-1}(0),
\]
with $\rho \mapsto u^\rho$ real analytic as a mapping $(-\rho_1, \rho_1) \mapsto \Xspace$.  Moreover, $\km_\loc$ satisfies the following.
\begin{enumerate}[label=\rm(\alph*)]
	\item \label{local asymptotics part} 
	$\km_\loc$ bifurcates from the point vortex configuration $\fullparam_0$ in that $u^0 = (0, \param_0)$, and to leading order is given by
	\begin{equation}
	\label{leading order hollow vortex}
    \begin{aligned}
      \mu_k^\rho(\tau) & =   \frac{16\pi}{\gamma_k} \rho \realpart{ \left( i S_k \tau \right)}  + O(\rho^2) &  \textrm{in } \mathring{C}^{\ell+\alpha}(\mathbb{T}) \\
      \nu_k^\rho(\tau) & = -2 \rho   \realpart {\left( S_k \tau^2 \right)}  +O(\rho^2) &  \textrm{in } \mathring{C}^{\ell+\alpha}(\mathbb{T})\\
      Q^\rho & = O(\rho^2) \\ 
      \param^\rho & = \param_0 + O(\rho^2).
    \end{aligned}
\end{equation} 
	
	\item \label{local symmetry part} The parameterization exhibits the symmetry
    \begin{equation}
      \label{parameterization symmetry}
       (\mu^\rho(\tau), \, \nu^\rho(\tau), \, Q^\rho, \, \param^\rho) = (\mu^{-\rho}(-\tau), \, -\nu^{-\rho}(-\tau), \, Q^{-\rho}, \, \param^{-\rho} ).
    \end{equation}
		For $0 < |\rho| <\rho_1$, it follows that $(u^\rho,\rho)$ represents a steady hollow vortex configuration, with $(u^\rho, \rho)$ and $(u^{-\rho}, -\rho)$ giving rise to the same physical solution.
	\item \label{local uniqueness part} In a neighborhood of $(u,\rho) = (0,\param_0,0)$ in $\mathcal{O}$,  the curve $\km_\loc$ comprises all solutions to the steady point vortex and steady hollow vortex problems with $\param^\prime = \param_0^\prime$ and conformally equivalent to a circular domain with equal radii.
	\end{enumerate}
\end{theorem}

\begin{remark}
By an identical argument, we can see that in fact there is a real-analytic manifold of steady hollow vortices $\mathscr{M}_\loc$ taking the form of a graph over the variables $(\param^\prime, \rho) \in \mathcal{N}$, for some neighborhood $\mathcal{N}$ of $(\param_0^\prime, 0)$. The section of $\mathscr{M}_\loc$ with $\rho = 0$ coincides with the manifold $\mathscr{M}_\loc^0$ of point vortex configurations from Proposition~\ref{point vortex manifold proposition}.
\end{remark}

\subsection{Linear analysis} \label{local linear section}

Recall from Section~\ref{formulation section} that we have formulated the system as the abstract operator equation
\[
	\F(u,\rho) = (\A(u,\rho), \B(u,\rho)) = 0
\]
where $\mathscr{A} = (\mathscr{A}_1, \ldots, \mathscr{A}_M)$ and $\mathscr{B} = (\mathscr{B}_1, \ldots, \mathscr{B}_M)$ are given by
\[
\begin{aligned}
	\mathscr{A}_k(u, \rho) 
		& \colonequals  \realpart{\left( \tau \left( \mathcal{Z}_k^\rho \nu^\prime + \Vrho_k^\rho + ( i \Omega \overline{( \zeta_k + \rho \tau)} -c) \rho \mathcal{Z}_k^\rho \mu^\prime + i\Omega \rho^2 \overline{\mathcal{Z}_k^\rho \mu}  + i \Omega \rho^3 \mathcal{Z}_k^\rho [\mu^\prime] \overline{\mathcal{Z}_k^\rho\mu} \right) \right)}   \\
	\mathscr{B}_k(u, \rho)	& \colonequals  \mathcal{B}_k^\rho(\mu,\nu,\param) - Q_k,
\end{aligned}
\]
for $k = 1, \ldots, M$.  Let $\fullparam_0 = (\param_0,\param_0^\prime)$ be a non-degenerate steady point vortex configuration and consider the trivial solution $(u^0,\rho) = (0,0,\param_0, \rho)$. A simple calculation reveals that the Fr\'echet derivative there is
\[
	 D_u \F(u^0,0) \dot u = 
				\begingroup
				\renewcommand\arraystretch{1.5}
				\begin{pmatrix} 
					 0  & \A_\nu^0 & 0 &  \A_\param^0 \\
					 \B_\mu^0 & \B_\nu^0 & \B_Q^0 &  \B_\param^0  \\
				\end{pmatrix} \endgroup
				\begin{pmatrix} \dot \mu \\ \dot \nu \\ \dot Q \\ \dot \param \end{pmatrix},
\]
where we are abbreviating $\A_\nu^0 \colonequals  D_\nu \A(u^0,0)$ and so on.  Importantly, from \eqref{linear Bernoulli operator} and \eqref{abstract operator}, it follows that $\A_\nu^0$ and $\B_w^0$ have have block diagonal form 
\begin{equation}
\label{DA DB formulas}
\begin{aligned}
 	\A_{k\nu}^0 \dot\nu & = \realpart{\left( \tau \mathcal{C} \dot\nu_k^\prime \right)}, \\ 
  	\B_{k(\mu,\nu,Q)}^0\begin{pmatrix} \dot \mu \\  \dot\nu \\ \dot Q\end{pmatrix} & = \frac{(\gamma_k^0)^2}{2 \pi^2} \realpart{ \left(  \frac{2\pi i}{\gamma_k^0} \tau \mathcal{C} \dot\nu_k^\prime - \mathcal{C}\dot\mu_k^\prime \right)} -\dot Q_k,
\end{aligned}
\qquad \textrm{for } k = 1, \ldots, M.
\end{equation}

Recall that $\mathcal{C} \tau^m = 0$ for $m \ge 0$ and $\mathcal{C} \tau^m = -\tau^m$ for $m < 0$. Thus if
$\varphi \in C^{\ell+\alpha}(\mathbb{T})$ for $\ell \geq 1$ and has the Fourier series representation
\[
	\varphi(\tau) = \sum_{m \in \mathbb{Z}} \widehat\varphi_m \tau^m,
\]
then the following multiplier formulas hold
\begin{equation}
 \label{multiplier formulas}
  \begin{aligned}
    \realpart{ \mathcal{C}[\varphi^\prime](\tau)} & =  \frac{1}{2} \left( \sum_{m = -\infty}^{-1} |m| \widehat\varphi_m \tau^{m-1}  + \sum_{m = 1}^{\infty} |m| \widehat\varphi_m \tau^{m+1}   \right)  \\
    \realpart{\left(\tau \mathcal{C}[\varphi^\prime](\tau)\right)} & = \frac{1}{2} \sum_{m \in \mathbb{Z}} |m|   \widehat\varphi_m \tau^m \\
    \realpart{\left( i \tau \mathcal{C}[\varphi^\prime](\tau) \right)}  & = -\frac{1}{2} \sum_{m \in \mathbb{Z}} i m \widehat\varphi_m \tau^{m} .
  \end{aligned}
\end{equation}
As an immediate corollary, the operators appearing in \eqref{multiplier formulas} are injective as mappings  $\mathring{C}^{\ell+\alpha}(\mathbb{T}) \to C^{\ell-1+\alpha}(\mathbb{T})$.  Using the projection operators \eqref{definition projection operator}, their ranges can be characterized as
\begin{equation}
 \label{multiplier cokernel} 
  \begin{aligned}
    \range{\realpart{(\tau \mathcal{C} \partial_\tau)}} & = \range{\realpart{( i {\tau} \mathcal{C} \partial_\tau)}} = \proj_{> 0} C^{\ell-1+\alpha}(\mathbb{T}) = \mathring{C}^{\ell-1+\alpha}(\mathbb{T}) \\ 
     \range{\realpart{( \mathcal{C} \partial_\tau)}} & = \proj_{> 1} C^{\ell-1+\alpha}(\mathbb{T}). 
  \end{aligned}
\end{equation}
Note that because these are spaces of real-valued functions, the above sets have codimension $1$ and $2$, respectively.   In particular, $\realpart{(\tau \mathcal{C} \partial_\tau)}$ is invertible $\mathring{C}^{\ell+\alpha} \to \mathring{C}^{\ell-1+\alpha}$.  Because $\A_\nu^0$ is diagonal with precisely this operator along its diagonal, we infer that $\A_\nu^0$ is likewise an isomorphism $\mathring{C}^{\ell+\alpha}(\mathbb{T})^M \to \mathring{C}^{\ell-1+\alpha}(\mathbb{T})^M$.

\begin{lemma}[Kernel and range]
\label{kernel and range lemma}
Suppose that $\fullparam_0 = (\param_0, \param_0^\prime)$ represents a steady vortex configuration and let $u^0 \colonequals  (0,\param_0) \in \Xspace$.  Then the Fr\'echet derivative $D_u \F(u^0,0) \colon \Xspace \to \Yspace$ is Fredholm with
\begin{equation}
\label{Fu dimension counts}
  \begin{aligned}
    \dim\kernel{D_u \F(u^0,0)} & = \dim\kernel{D_{\param} \mathcal{V}(\fullparam_0)} \\ 
    \cod\range{D_u\F (u^0, 0)} &= \cod\range{D_{\param} \mathcal{V}(\fullparam_0)}. 
  \end{aligned}
\end{equation}  
\end{lemma}
\begin{proof}
Inverting $\A_\nu^0$, we consider the ``row-reduced'' operator
  \begin{equation}
    \label{definition row reduced operator}
    \mathscr{L} \begin{pmatrix} \dot \mu \\ \dot Q \\ \dot \param \end{pmatrix} \colonequals  \B_\mu^0\dot\mu + \B_Q^0 \dot Q + \left( \B_\param^0 - \B_\nu^0 (\A_\nu^0)^{-1} \A_\param^0\right) \dot\param. 
  \end{equation}
Clearly, $\F_u(u^0,0)$ is Fredholm if and only if
\[
	\mathscr{L} :  \mathring C^{\ell+\alpha}(\mathbb{T})^M \times \mathbb{R}^M  \times \Pspace \to C^{\ell-1+\alpha}(\mathbb{T})^M \quad \textrm{is Fredholm}.
\]
Moreover, the dimension of  $\kernel\mathscr{L}$ and $\kernel{D_u \F(u^0,0)}$ coincide, as do the codimension of their ranges.

Now, $\A_\param^0 \dot \param= \realpart{( \tau  \mathcal{V}_\param^0 \dot \param)}$, and hence \eqref{multiplier formulas} and \eqref{DA DB formulas} together give the explicit formulas
\begin{align*}
	(\A_\nu^0)^{-1} \A_\param^0 \dot \param & = - \realpart{(\tau \mathcal{V}_\param^0 \dot \param)} = -\frac{1}{2} \left( \tau \mathcal{V}_{\param}^0 \dot \param + \frac{1}{\tau} \overline{\mathcal{V}_{\param}^0 \dot \param} \right) \\
	 \B_\nu^0 (\A_\nu^0)^{-1} \A_{k\param}^0  \dot\param & = \frac{\gamma_k^0}{\pi} \realpart{\left( i \tau \mathcal{C} \partial_\tau (\A_\nu^0)^{-1} (\A_{k\param}^0 \dot \param) \right)} = -\frac{\gamma_k^0}{2\pi} {\left( i \tau  \mathcal{V}_{k\param}^0 \dot \param - i \frac{1}{\tau} \overline{\mathcal{V}_{k\param}^0 \dot \param} \right)} \\
	 & = -\frac{\gamma_k^0}{\pi} \realpart{\left( i \tau \Vrho_{k\lambda}^0 \dot\lambda \right)}.
\end{align*}
Using these facts, \eqref{linear Bernoulli operator}, and \eqref{DA DB formulas}, we find that $\mathscr{L}$ has $k$-th component taking the form
  \begin{equation}
    \label{Lk formula}
    \mathscr{L}_k \begin{pmatrix} \dot \mu \\ \dot Q \\ \dot \param \end{pmatrix} = -\frac{(\gamma_k^0)^2}{2\pi^2} \realpart{\mathcal{C}\dot\mu_k^\prime} - \dot Q_k +   \frac{2\gamma_k^0}{\pi} \realpart{\left( i \tau  \mathcal{V}_{k\param} \dot \param \right)}.
  \end{equation}
Written this way, the dimension counts in \eqref{Fu dimension counts} become clear.  Suppose first that $\dot u \in \kernel{\mathscr{L}}$.  Projecting \eqref{Lk formula} to the subspace of constants by applying $\proj_0$, we see that $\dot Q = 0$.  Likewise, applying $\proj_1$, we find that $\dot \param \in \kernel{D_\param \mathcal{V}(\fullparam_0)}$, and thus by \eqref{multiplier cokernel} we have $\dot\mu = 0$.  Therefore
  \begin{equation}
    \label{kernel characterization}
    \kernel{D_u \F(u^0,0)} = \{ 0 \} \times  \kernel{D_{\param} \mathcal{V}(\fullparam_0)} \subset \Xspace,
  \end{equation}
and so in particular the dimension of $\kernel{D_u \F(u^0,0)}$ agrees with that of $\kernel{D_\param \mathcal{V}(\fullparam_0)}$.

The same reasoning allows us to characterize the range of $\mathscr{L}$.   Indeed, we see that for $\varphi \in C^{\ell-1+\alpha}(\mathbb{T})^M$,
  \begin{equation}
    \label{range characterization}
    \varphi \in \range{\mathscr{L}} \quad \textrm{if and only if} \quad \int_{\mathbb{T}} \varphi \overline{\tau} \, d\theta \in \range{D_\param \mathcal{V}(\fullparam_0)}. 
  \end{equation}
Thus, the codimensions of $\range{D_u \F(u^0,0)}$ and $\range{D_\param \mathcal{V}(\fullparam_0)}$ coincide, which completes the proof. 
\end{proof}

\subsection{Proof of the local desingularization theorem}

Combining the results of Section~\ref{local linear section} with the identities we derived in Section~\ref{hollow vortex identities section}, we now prove the main result on the local desingularization of point vortices. 

\begin{proof}[Proof of Theorem~\ref{hollow vortex local bifurcation theorem}]
The existence and regularity of the curve $\km_\loc$ follows from the abstract implicit function theorem result Lemma~\ref{local ift phi theorem}, along with our characterization of the kernel and cokernel in  Lemma~\ref{kernel and range lemma}, and the identities derived in Theorem~\ref{hollow vortex identities theorem}.  Local uniqueness and the invariance of the equation observed in Remark~\ref{negative rho operator equation remark} ensure that the parameterization obeys the claimed symmetry.  It remains only to calculate the leading-order asymptotics \eqref{leading order hollow vortex}.  

Writing $u^\rho \equalscolon  \rho \dot u + O(\rho^2)$, we see that
\[
	\A_u^0 \dot u = -\A_\rho^0, \qquad \B_u^0 \dot u = -\B_\rho^0.
\] 
The formulas for $\A_w^0$ and $\B_w^0$ are already given in \eqref{DA DB formulas}.  For the remaining derivatives, we compute that
\begin{align*}
	 \A_{k\lambda}^0 &= \realpart{ \left( \tau  \mathcal{V}_{k\param} \right)} \\
	 \A_{k\rho}^0 & = \realpart{ \left(  \tau  \Vrho_{k\rho}^0\right)} = -\realpart{\left( \frac{1}{2\pi i} \tau^2 \sum_{j \neq k} \frac{\gamma_j^{0}}{(\zeta_j^0 - \zeta_k^0)^2}\right) } = 2\realpart{(S_k \tau^2)}.
\end{align*}
Then, from
  \begin{equation}
    \label{asymptotic kinematic}
    \A_{k\nu}^0 \dot\nu + \A_{k \param}^0 \dot \param = - \A_{k\rho}^0,
  \end{equation}
we see that
\[
	0 = \proj_1 \A_{k\param}^0 \dot\param = \frac{\tau}{2} \Vrho_{k\lambda} \dot\lambda+  \frac{1}{2\tau} \overline{\Vrho_{k\lambda} \dot\lambda}.
\]
But, recall that by construction $\Vrho_\lambda$ has a trivial kernel, and thus $\dot\lambda = 0$. Applying $\proj_{>2}$ to \eqref{asymptotic kinematic} yields 
\[
	\proj_{> 2}  \A_{k\nu}^0 \dot\nu = 0,
\]
and hence $\proj_{>2} \dot\nu = 0$.  Moreover, using \eqref{multiplier formulas} we can invert $\realpart{(\tau \mathcal{C} \partial_\tau \placeholder)}$ to obtain
\begin{align*}
	 \dot\nu_k &=   -\left( S_k \tau^2 + \overline{S_k} \frac{1}{\tau^2} \right) 
	 =-2 \realpart{\left(  S_k \tau^2 \right)}.
\end{align*}

Turning to the dynamic condition, we see that
  \begin{equation}
    \label{asymptotics bernoulli}
    -\frac{(\gamma_k^0)^2}{2 \pi^2} \realpart{ \mathcal{C} \dot\mu_k^\prime} + \frac{\gamma_k^0}{\pi} \realpart{(i \tau \mathcal{C} \dot\nu_k^\prime)} - \dot Q_k + \B_{k\param}^0 \dot\param = -\B_{k\rho}^0.
  \end{equation}
Expanding the operator $\mathcal{B}_k$ gives
\[ 
	\rho \mathcal{B}_k^\rho(0,0,\lambda) = \left| \frac{\gamma_k}{2\pi i \tau} + \rho \Vrho_k^\rho\right|^2 - \frac{\gamma_k^2}{4 \pi^2} = \rho^2 |\Vrho_k^\rho|^2 + 2 \rho \realpart{\left( \overline{\frac{\gamma_k}{2\pi i \tau}} \Vrho_k^\rho \right)},
\]
and thus
\begin{align*}
	\B_{k\rho}^0 &=  \frac{\gamma_k^0}{\pi } \realpart{\left(  i\tau \Vrho_{k\rho}^0 \right)}  = \frac{2\gamma_k^0}{\pi} \realpart{ (i S_k \tau^2  )}.
\end{align*}
In particular, applying the projection $\proj_0$ to \eqref{asymptotics bernoulli} reveals that $\dot Q_k = 0$.  Using \eqref{multiplier formulas} and our previous expression for $\dot\nu$, we then find that
\begin{align*}
	\frac{\gamma_k^0}{\pi} \realpart{(i \tau \mathcal{C} \dot\nu_k^\prime)} &=  \frac{2\gamma_k^0}{\pi} \realpart{\left( i S_k \tau^2 \right)}.
\end{align*}
Applying $\proj_2$ and $1-\proj_2$ to \eqref{asymptotics bernoulli}, we therefore obtain
\[
	\left\{
	\begin{aligned}
		(1-\proj_2) \realpart{ \mathcal{C} \dot\mu_k^\prime} & =  0 \\
		\proj_2 \realpart{ \mathcal{C} \dot\mu_k^\prime} & = \proj_2 \left(  \frac{2\pi^2}{(\gamma_k^0)^2} \B_{k\rho}^0 + \frac{2 \pi}{\gamma_k^0} \realpart{\left( i  \tau \mathcal{C}\dot\nu_k^\prime\right)} \right) = \frac{ 8 \pi}{\gamma_k^0} \realpart{ \left(  i S_k \tau^2 \right)}.
	\end{aligned}
	\right.
\]
Using \eqref{multiplier formulas}, these readily lead us to
\begin{align*}
	\dot\mu_k & = \frac{16 \pi}{\gamma_k^0} \realpart{ \left(  i S_k \tau \right)}.
\end{align*}
Thus, the asymptotics in \eqref{leading order hollow vortex} have been confirmed and the proof is complete.
\end{proof}

\subsection{Desingularization in the presence of symmetry}
\label{desingularization with symmetry section}

We next present a number of corollaries to Theorem~\ref{hollow vortex local bifurcation theorem} that demonstrate how the general vortex desingularization result can be strengthened in specific cases where additional symmetries are available.  In what follows, a subset of $\mathbb{C}$ is called \emph{symmetric} if it is invariant under even reflection over the real and imaginary axes.  For the fluid domain $\fluidD$, this holds provided the corresponding conformal domain $\confD_\rho$ is symmetric and the mapping $f$ is real on real and imaginary on imaginary.  We will consider two symmetry classes for the relative velocity field $\overline{\partial_z W}$: one for the rotating pair and stationary tripole examples, and another for translating pairs.  

The next lemma shows how these assumptions can be integrated into the formulation of the hollow vortex problem introduced in Section~\ref{formulation section}.  It can be easily generalized to other situations, but to keep the statement as simple as possible, we only present those necessary for our three applications.  

\begin{lemma}[Nonlinear symmetries] 
\label{nonlinear symmetries lemma}
Let densities $(\mu,\nu) \in \mathring C^{\ell+\alpha}(\mathbb{T})^M \times \mathring C^{\ell+\alpha}(\mathbb{T})^M$ be given and define the conformal mapping $f$ and velocity potential $w$ by \eqref{f ansatz} and \eqref{w ansatz}, respectively.  
\begin{enumerate}[label=\rm(\alph*)]
\item \label{rotating and tripole symmetries part} Suppose that $\confD_\rho$ is symmetric and for each $1 \leq k \leq M$ it holds that $\zeta_k \in \mathbb{R}$, and 
\[
	(\mu_k,\nu_k) \in \left\{ 
	\begin{aligned}
	 	\mathring C_\realreal^{\ell+\alpha}(\mathbb{T}) \times \mathring C_\imagimag^{\ell+\alpha}(\mathbb{T}) & \qquad   \textrm{if } \zeta_k \neq 0 \\
		\mathring C_{\realreal,\imagimag}^{\ell+\alpha}(\mathbb{T}) \times \mathring C_{\imagreal,\imagimag}^{\ell+\alpha}(\mathbb{T}) & \qquad  \textrm{if } \zeta_k = 0.
	\end{aligned} \right.
\]
Then $\fluidD$ is symmetric, the relative velocity $\overline{\partial_z W}$ is imaginary on real and real on imaginary, and
\begin{equation}
\label{codomain rotating and tripole}
	\F_k(u,\rho) \in \left\{ 	
  \begin{aligned}
    \mathring C_\imagreal^{\ell-1+\alpha}(\mathbb{T}) \times  C_\realreal^{\ell-1+\alpha}(\mathbb{T}) & \qquad   \textrm{if } \zeta_k \neq 0 \\
    \mathring C_{\imagreal,\imagimag}^{\ell-1+\alpha}(\mathbb{T}) \times  C_{\realreal,\realimag}^{\ell-1+\alpha}(\mathbb{T}) & \qquad  \textrm{if } \zeta_k = 0.
  \end{aligned} \right.
\end{equation}
\item \label{translating symmetry part} Suppose that $\confD_\rho$ is symmetric and for each $1 \leq k \leq M$ it holds that $\zeta_k \in i\mathbb{R} \setminus \{0\}$, and
\[
	(\mu_k,\nu_k) \in \mathring C_\imagimag^{\ell+\alpha}(\mathbb{T}) \times \mathring C_\realimag^{\ell+\alpha}(\mathbb{T}).
\]
Then $\fluidD$ is symmetric, the relative velocity $\overline{\partial_z W}$ is real on real and real on imaginary, and
\[
	\F_k(u,\rho) \in \mathring{C}_\realimag^{\ell-1+\alpha}(\mathbb{T}) \times  C_\realimag^{\ell-1+\alpha}(\mathbb{T}).
\]
\end{enumerate}
\end{lemma}
\begin{proof}
Consider first the situation in part~\ref{rotating and tripole symmetries part}.  The stated symmetry of $\fluidD$ and $\overline{\partial_z W}$ follow from the explicit formulas \eqref{f ansatz} and \eqref{w ansatz} together with Lemma~\ref{symmetry over axes lemma} and the surrounding discussion.  Verifying that the nonlinear operator $\F_k = (\mathscr{A}_k, \mathscr{B}_k)$ satisfies \eqref{codomain rotating and tripole} is best done through the local formulation.  Recall that for $|\rho| > 0$ and $\tau \in \mathbb{T}$, physically $\mathscr{A}_k = \mathscr{A}_k(u,\rho)(\tau)$ corresponds to the normal component of the relative velocity field at $\zeta = \zeta_k + \rho \tau$.  Based on the symmetry properties of $f$ and $w$, a careful examination of \eqref{local kinematic condition} shows that
\[
	\mathscr{A}_k^* = \mathscr{A}_k \quad \textrm{if } \zeta_k \neq 0, \qquad \mathscr{A}_k^* = \mathscr{A}_k = \mathscr{A}_k(-\placeholder) \quad \textrm{if } \zeta_k = 0.
\]
Expanding $\mathscr{A}_k$ as a Fourier series and arguing as in Section~\ref{symmetries subsection}, we find that necessarily $\mathscr{A}_k$ lies in the subspace indicated in \eqref{codomain rotating and tripole}. By real analyticity, this extends to $\rho = 0$ as well.  

In the same vein, observe that when $|\rho| > 0$, the second component of $\F_k$ satisfies 
\[ 
	\rho \mathscr{B}_k(u,\rho)(\tau) = \rho^2 | U(\zeta_k + \rho \tau) |^2 - \frac{\gamma_k}{4 \pi^2} - \rho Q_k,
\]
where, as before, $U$ is the conjugate of the relative velocity field.  We can again infer from this that
\[
	\mathscr{B}_k^* = \mathscr{B}_k \quad \textrm{if } \zeta_k \neq 0, \qquad \mathscr{B}_k^* = \mathscr{B}_k = \mathscr{B}_k(-\placeholder) \quad \textrm{if } \zeta_k = 0,
\]
and the same will hold for $\rho = 0$ by real analyticity.  Expanding as a Fourier series and following the same argument as in Section~\ref{symmetries subsection}, we infer that  $\mathscr{B}_k$ belongs to the subspace indicated in \eqref{codomain rotating and tripole}.  The proof of part~\ref{translating symmetry part} is an easy adaptation of these ideas, and hence we omit it.
\end{proof}
 
 First, consider the co-rotating vortex pair given by 
\[
	\fullparam^0 = \fullparam_{\textup{r}}^0 \colonequals  \Big(\gamma_1 = 1, \, \gamma_2=1,\,  \zeta_1 = 1, \, \zeta_2 = -1, \, c = 0, \, \Omega = \frac{1}{4\pi}\Big).
\]
In Example~\ref{rotating pair example}, we identified a non-degenerate splitting of the parameters with $\lambda = (\realpart{\zeta_1}, \zeta_2)$.  A direct application of Theorem~\ref{hollow vortex local bifurcation theorem} would then furnish a local curve of nearby hollow vortices along which the pre-image of the centers vary.  However, this scenario produces some unwanted analytical complications for the global continuation argument taken up in the next section.   Here we give an alternative construction that, by exploiting the symmetry of the configuration, permits us to fix the vortex centers and circulations taking instead $\lambda = \Omega$.  

\begin{corollary}[Small rotating pair]
\label{local rotating pair corollary}
Let $\fullparam_{\textup{r}}^0$ represent the pair of co-rotating point vortices given above.  There exists $\rho_1 > 0$ and a curve $\km_{\mathrm{r},\loc}$ of co-rotating hollow vortices that admits the parameterization 
\[
	\km_{\mathrm{r},\loc} = \left\{  (u^\rho,\rho) = (\mu^\rho, \nu^\rho, Q^\rho, \Omega^\rho, \rho)  \in \mathcal{O} : \quad |\rho| < \rho_1 \right\} \subset \F^{-1}(0),
\]
with $\rho \mapsto u^\rho$ real analytic as a mapping $(-\rho_1, \rho_1) \mapsto \Xspace$ and the remaining components of $\fullparam$ fixed to their values in $\fullparam_{\textup{r}}^0$.  
\begin{enumerate}[label=\rm(\alph*)]
\item $\km_{\mathrm{r},\loc}$ bifurcates from the point vortex configuration $\fullparam_{\textup{r}}^0$ and to leading order is given by \eqref{leading order hollow vortex} with $\lambda^\rho = \Omega^\rho$ and $\lambda_0 = 1/(4\pi)$.
\item The parameterization exhibits the symmetry \eqref{parameterization symmetry}.
\item Each solution on $\km_{\mathrm{r},\loc}$ is symmetric in that the fluid domain is even over both axes while the relative velocity is real on imaginary and imaginary on real.  In a neighborhood of $(0,\param_0,0)$ in $\mathcal{O}$,  the curve $\km_{\mathrm{r},\loc}$ comprises all solutions with this symmetry and such that $\param^\prime = \param_0^\prime$ and conformally equivalent to a circular domain with equal radii 
\end{enumerate}
\end{corollary}
\begin{proof}
As this configuration is symmetric with respect to both the imaginary and real axis, and the centers both lie on the real axis, it suffices to determine $u_1 \colonequals  (\mu_1,\nu_1,Q_1,\Omega)$ setting 
\[
	\mu_2 \colonequals  -\mu_1^*(-\placeholder), \qquad \nu_2 \colonequals  \nu_1^*(-\placeholder), \qquad Q_2 = Q_1
\]
and leaving the remaining parameters fixed.  By the symmetries of the nonlinear problem, moreover, we need only solve the kinematic and dynamic conditions on the first vortex boundary, assuming that $\mu_1$ and $\nu_1$ satisfy \eqref{mu even real} and \eqref{mu even imaginary}, respectively.  
Lemma~\ref{nonlinear symmetries lemma}\ref{rotating and tripole symmetries part} therefore justifies reformulating the problem as $\F_1(u_1,\rho) = 0$, where
\[
	\F_1 = (\mathscr{A}_1, \mathscr{B}_1) \colon \mathring C_\realreal^{\ell+\alpha}(\mathbb{T}) \times \mathring C_\imagimag^{\ell+\alpha}(\mathbb{T}) \times \mathbb{R}^3 \longrightarrow \mathring{C}_{\imagreal}^{\ell-1+\alpha}(\mathbb{T}) \times  C_\realreal^{\ell-1+\alpha}(\mathbb{T})
\]
encodes the Bernoulli and kinematic conditions on the first vortex boundary.  %Thus the vortex centers and strengths are now fixed, and the only component of $\fullparam$ allowed to vary is the angular velocity $\Omega$.  

The proof parallels that of Theorem~\ref{hollow vortex local bifurcation theorem} with only a few small modifications.  As in \eqref{DA DB formulas}, we compute that
\[
	D_{u_1} \F_1(u_1^0,0) = 
		\begin{pmatrix} 
			0 & \mathscr{A}_{1\nu_1}^0 & 0 & \mathscr{A}_{1\Omega}^0 \\ 
			\mathscr{B}_{1\mu_1}^0 & \mathscr{B}_{1\nu_1}^0 & \mathscr{B}_{1Q_1}^0 & \mathscr{B}_{1\Omega}^0
		\end{pmatrix} 
		\begin{pmatrix}
			\dot \mu_1 \\
			\dot \nu_1 \\
			\dot Q_1 \\
			\dot \Omega
		\end{pmatrix},
\]
with 
\[
	\mathscr{A}_{1\nu_1}^0 \dot \nu_1 = \realpart{\left( \tau \mathcal{C} \partial_\tau \dot\nu_1 \right)} \qquad 
	\mathscr{B}_{1 (\mu_1, \nu_1, Q_1)}^0 \begin{pmatrix} \dot \mu_1 \\ \dot \nu_1 \\ \dot Q_1 \end{pmatrix} = \frac{1}{2\pi^2}  \realpart{ \left( 2\pi i \tau \mathcal{C} \dot \nu_1^\prime - \mathcal{C} \dot \mu_1^\prime \right)} - \dot Q_1,
\]
the domain and codomain being changed as indicated above.  The characterization of the kernel \eqref{kernel characterization} still holds, and since $\mathcal{V}_{1\Omega} = i$ defines an injective mapping, we see that so too is $D_{u_1} \F_1(u_1,0)$.   Likewise, from the multiplier formulas \eqref{multiplier formulas}, we have that $\mathscr{A}_{1\nu_1}^0$ is an isomorphism $\mathring{C}_{\realreal}^{\ell+\alpha}(\mathbb{T}) \to \mathring{C}_{\realreal}^{\ell-1+\alpha}(\mathbb{T})$. Thus to establish surjectivity it suffices to consider the  row-reduced operator $\mathscr{L}_1$ given now by 
  \begin{equation}
    \label{row reduced operator rotating}
    \mathscr{L}_1 
    \begin{pmatrix} 
      \dot \mu_1 \\ 
      \dot Q_1 \\ 
      \dot \Omega 
    \end{pmatrix}
     \colonequals  -\frac{1}{2\pi^2} \realpart{\mathcal{C} \dot \mu_1^\prime} - \dot Q_1 - \frac{2 }{\pi} \realpart{ \left( \dot \Omega \tau \right)}
  \end{equation}
and viewed as a mapping $C_{\realreal}^{\ell+\alpha}(\mathbb{T}) \times \mathbb{R}^2 \to C_{\realreal}^{\ell-1+\alpha}(\mathbb{T})$. From \eqref{multiplier formulas} one can easily verify that
\[
    \range{\realpart{( \mathcal{C} \partial_\tau)}|_{\mathring{C}_{\realreal}^{\ell+\alpha}}}   = P_{>1} {C}_\realreal^{\ell-1+\alpha}(\mathbb{T}).
\]
Hence $\mathscr{L}_1$ is an isomorphism and the implicit function theorem can be be applied, giving the existence and local uniqueness of the real-analytic curve.  The leading-order asymptotics follow exactly as in the proof of Theorem~\ref{hollow vortex local bifurcation theorem}.  Finally, the claimed symmetry properties follow from Lemma~\ref{nonlinear symmetries lemma}\ref{rotating and tripole symmetries part}
\end{proof}

Next we investigate stationary tripoles of hollow vortices bifurcating from the point vortex configuration
\[ 
	\fullparam = \fullparam_{\textup{s}}^0 \colonequals  \Big(\gamma_1 = 2,\,  \gamma_2 = -1, \,  \gamma_3=2, \, \zeta_1 = 1,\,  \zeta_2=0, \, \zeta_3=-1, \, c=0,\, \Omega=0\Big).
\]
One non-degenerate splitting is presented in Example~\ref{stationary tripole example}.  The next corollary uses symmetry to find a different curve of nearby stationary hollow vortices for which the only varying parameters is the circulation $\gamma_2$ about the central vortex.  

\begin{corollary}[Small stationary tripole]\label{small stationary tripole corollary}
Let $\fullparam_{\textup{s}}^0$ represent the stationary tripole of point vortices given above.  There exists $\rho_1 > 0$ and a curve $\km_{\mathrm{s},\loc}$ of stationary hollow vortices that admits the parameterization 
\[
	\km_{\mathrm{s},\loc} = \left\{  (u^\rho,\rho) = (\mu^\rho, \nu^\rho, Q^\rho, \gamma_2^\rho, \rho)  \in \mathcal{O} : \quad |\rho| < \rho_1 \right\} \subset \F^{-1}(0),
\]
with $\rho \mapsto u^\rho$ real analytic as a mapping $(-\rho_1, \rho_1) \mapsto \Xspace$ and the remaining components of $\fullparam$ fixed to their values in $\fullparam_{\textup{s}}^0$.  
\begin{enumerate}[label=\rm(\alph*)]
\item $\km_{\mathrm{s},\loc}$ bifurcates from the point vortex configuration $\fullparam_{\textup{s}}^0$ and to leading order is given by \eqref{leading order hollow vortex} with $\lambda^\rho = \gamma_2^\rho$ and $\lambda_0 = -1$.
\item The parameterization exhibits the symmetry \eqref{parameterization symmetry}.
\item Each solution on $\km_{\mathrm{s},\loc}$ is symmetric in that the fluid domain is even over both axes while the relative velocity is real on imaginary and imaginary on real.  In a neighborhood of $(0,\param_0,0)$ in $\mathcal{O}$,  the curve $\km_{\mathrm{s},\loc}$ comprises all solutions with this symmetry and such that $\param^\prime = \param_0^\prime$ and conformally equivalent to a circular domain with equal radii 
\end{enumerate}
\end{corollary}
\begin{proof}
In this case, it suffices to consider only the unknown $(\mu_1, \mu_2, \nu_1, \nu_2, Q_1, Q_2,\gamma_2)$, taking 
\[
	\mu_3 \colonequals  -\mu_1^*(-\placeholder), \qquad \nu_3 \colonequals  \nu_1^*(-\placeholder), \qquad Q_3 = Q_1
\]
and leaving the remaining parameters at their starting values.  The kinematic and dynamic conditions must then be satisfied on both the right and center vortex boundary.  The requirements on $\mu_1$ and $\nu_1$ are exactly the same as in the rotating pair case.  Because $\zeta_2 = 0$, however, we must have that $\mu_2$ satisfies \eqref{mu even real and imaginary} while $\nu_2$ satisfies \eqref{mu odd real and imaginary}.  Thanks to Lemma~\ref{nonlinear symmetries lemma}\ref{rotating and tripole symmetries part}, it is enough to study the zero-set of the operator $\F = (\F_1, \F_2)$ in a neighborhood of the point vortex tripole, where for $k = 1$ or $2$, $\F_k = (\mathscr{A}_k, \mathscr{B}_k) \colon  \Xspace_{\mathrm{s}} \to \Yspace_{k \mathrm{s}}$ is a real-analytic mapping between the spaces
\[
\begin{aligned}
	\Xspace_{\mathrm{s}} & \colonequals  \mathring C_\realreal^{\ell+\alpha}(\mathbb{T}) \times \mathring C_{\realreal,\imagimag}^{\ell+\alpha}(\mathbb{T})  \times \mathring C_\imagreal^{\ell+\alpha}(\mathbb{T}) \times  \mathring C_{\imagreal,\imagimag}^{\ell+\alpha}(\mathbb{T}) \times \mathbb{R}^2  \times \mathbb{R} \\
	\Yspace_{1\mathrm{s}} &\colonequals  \mathring{C}_{\imagreal}^{\ell-1+\alpha}(\mathbb{T}) \times  C_\realreal^{\ell-1+\alpha}(\mathbb{T})\\
	\Yspace_{2\mathrm{s}} & \colonequals  \mathring{C}_{\imagreal,\imagimag}^{\ell-1+\alpha}(\mathbb{T}) \times  C_{\realreal,\realimag}^{\ell-1+\alpha}(\mathbb{T}).
\end{aligned}
\]

The analysis is much the same as Theorem~\ref{hollow vortex local bifurcation theorem} and Corollary~\ref{local rotating pair corollary}.  An identical argument shows that the linearized operator $\F_{u}$ is injective upper triangular, thus we need only establish the surjectivity of the row-reduced operators $\mathscr{L}_1$ and $\mathscr{L}_2$.  Direct computation gives 
\[
	\mathscr{L}_1 
		\begin{pmatrix} 
			\dot \mu_1 \\ 
			\dot Q_1 \\ 
			\dot \gamma_2
		\end{pmatrix}
		 \colonequals  -\frac{\gamma_1^2}{2\pi^2} \realpart{\mathcal{C} \dot \mu_1^\prime} - \dot Q_1 - \frac{2 }{\pi^2} \realpart{ \left(  \tau \dot \gamma_2 \right)},
\]
which is an isomorphism $\mathring{C}_{\realreal}^{\ell+\alpha}(\mathbb{T}) \times \mathbb{R}^2 \to C_{\realreal}^{\ell-1+\alpha}(\mathbb{T})$.  On the other hand, $\mathcal{V}_{2 \gamma_2} = 0$, and hence the row-reduced operator for the second vortex boundary is 
\[
	\mathscr{L}_2
		\begin{pmatrix} 
			\dot \mu_2 \\ 
			\dot Q_2 \\ 
			\dot \gamma_2
		\end{pmatrix}
		 \colonequals  -\frac{\gamma_1^2}{2\pi^2} \realpart{\mathcal{C} \dot \mu_2^\prime} - \dot Q_2.
\]
This is independent of $\dot\gamma_2$ and we can easily verify that $(\dot\mu_2, \dot Q_n) \mapsto -(\gamma_1^2/2\pi^2) \realpart{\mathcal{C} \dot \mu_2^\prime} - \dot Q_2$  constitutes an isomorphism $\mathring{C}_{\realreal,\imagimag}(\mathbb{T}) \times \mathbb{R} \to C_{\realreal,\realimag}^{\ell-1+\alpha}(\mathbb{T})$. Combined with the previous line, this shows $(\mathscr{L}_1, \mathscr{L}_2)$ is an isomorphism, and hence the proof follows by applying the implicit function theorem.  
\end{proof}

As the last examples, let us consider a pair of co-translating point vortices:
\[
	\fullparam = \fullparam_{\textup{t}}^0 \colonequals \Big(\gamma_1 = 1,\,  \gamma_2 = -1, \, \zeta_1 = i, \, \zeta_2 = -i, \, c = \frac{1}{4\pi}, \, \Omega = 0  \Big).
\]
Again, one can use the parameter splitting from Example~\ref{translating pair example} in conjunction with Theorem~\ref{hollow vortex local bifurcation theorem} to generate one family of nearby hollow vortex solutions.  The next corollary uses a variant of this argument to recover the classical family of Pocklington vortex pairs \cite{pocklington1896} for which the sole varying parameter is the wave speed $c$.

\begin{corollary}[Small translating pairs]\label{small translating pair corollary}
Let $\fullparam_{\textup{t}}^0$ represent the pair of co-rotating point vortices given above.  There exists $\rho_1 > 0$ and a curve $\km_{\mathrm{t},\loc}$ of translating hollow vortices that admits the parameterization 
\[
	\km_{\mathrm{t},\loc} = \left\{  (u^\rho,\rho) = (\mu^\rho, \nu^\rho, Q^\rho, c^\rho, \rho)  \in \mathcal{O} : \quad |\rho| < \rho_1 \right\} \subset \F^{-1}(0),
\]
with $\rho \mapsto u^\rho$ real analytic as a mapping $(-\rho_1, \rho_1) \mapsto \Xspace$ and the remaining components of $\fullparam$ fixed to their values in $\fullparam_{\textup{t}}^0$.  
\begin{enumerate}[label=\rm(\alph*)]
\item $\km_{\mathrm{t},\loc}$ bifurcates from the point vortex configuration $\fullparam_{\textup{t}}^0$ and to leading order is given by \eqref{leading order hollow vortex} with $\lambda^\rho = c^\rho$ and $\lambda_0 = 1/(4\pi)$.
\item The parameterization exhibits the symmetry \eqref{parameterization symmetry}.
\item Each solution on $\km_{\mathrm{t},\loc}$ is symmetric in that the fluid domain is even over both axes while the relative velocity is real on real and real on imaginary.  In a neighborhood of $(0,\param_0,0)$ in $\mathcal{O}$,  the curve $\km_{\mathrm{t},\loc}$ comprises all solutions with this symmetry and such that $\param^\prime = \param_0^\prime$ and conformally equivalent to a circular domain with equal radii 
\end{enumerate}
\end{corollary}
\begin{proof}
For the translating vortices, we may work with a reduced unknown $u_1 = (\mu_1,\nu_1,Q_1,c)$ taking 
\[
	(\mu_2, \nu_2) = (\mu_1^*, \nu_1^*), \qquad Q_2 = Q_1
\]
and leaving the other components of $\fullparam$ fixed to their values in $\fullparam_{\mathrm{t}}^0$.  We need only ensure that the Bernoulli condition and dynamic condition hold on the first vortex boundary while requiring $\mu_1$ satisfies \eqref{mu even imaginary} and $\nu_1$ satisfies \eqref{mu odd imaginary}.   In view of Lemma~\ref{nonlinear symmetries lemma}\ref{translating symmetry part}, it is therefore enough to solve  $\F_1(u_1,\rho) = 0$, with the nonlinear operator now thought of as a mapping
\[
	\F_1 = (\mathscr{A}_1, \mathscr{B}_1) \colon 
	\mathring C_\imagimag^{\ell+\alpha}(\mathbb{T}) \times \mathring C_\realimag^{\ell+\alpha}(\mathbb{T}) \times \mathbb{R}^3
	\longrightarrow
	\mathring{C}_\realimag^{\ell-1+\alpha}(\mathbb{T}) \times  C_\realimag^{\ell-1+\alpha}(\mathbb{T}).
\]

The linearized operator at the point vortex configuration is once more given by \eqref{DA DB formulas}.  Because $\mathcal{V}_{1c} = 1$, we likewise see from \eqref{kernel characterization} that it is injective.  The multiplier formulas \eqref{multiplier formulas}, moreover, imply that $\mathscr{A}_{1\nu_1}^0$ is an isomorphism $\mathring{C}_\realimag^{\ell+\alpha}(\mathbb{T}) \to C_\realimag^{\ell-1+\alpha}(\mathbb{T})$, and hence to establish surjectivity it is enough to study the row-reduced operator
  \begin{equation}
    \label{row reduced operator}
    \mathscr{L}_1 
    \begin{pmatrix} 
      \dot \mu_1 \\ 
      \dot Q_1 \\ 
      \dot c 
    \end{pmatrix}
     \colonequals  -\frac{\gamma_1^2}{2\pi^2} \realpart{\mathcal{C} \dot \mu_1^\prime} - \dot Q_1 + \frac{2 \gamma_1}{\pi} \realpart{ \left( i \tau \dot c \right)}
  \end{equation}
now viewed as a mapping $\mathring{C}_{\imagimag}^{\ell+\alpha}(\mathbb{T}) \times \mathbb{R} \times \mathbb{R} \to C_{\imagimag}^{\ell+1-\alpha}(\mathbb{T})$.  But as a consequence of \eqref{multiplier formulas}, it is easy to see that
\[
    \range{\realpart{( \mathcal{C} \partial_\tau)}|_{\mathring{C}_{\imagimag}^{\ell+\alpha}}}   = P_{>1} {C}_\realimag^{\ell-1+\alpha}(\mathbb{T}),
\]
and $P_{\leq 1} {C}_\realimag^{\ell-1+\alpha}(\mathbb{T})$ is a subset of the range of $(\dot Q_1, \dot c) \mapsto -\dot Q_1 + \frac{2\gamma_1}{\pi} \realpart{(i \tau \dot c)}$.  Thus $\mathscr{L}_1$ is surjective and the proof is complete.  
\end{proof}

\section{Global continuation of hollow vortices}
\label{global section}

\subsection{Linear estimates}
In this section we prove some basic Schauder estimates for the linearized operators $D_{(\mu,\nu)} \F(u,\rho)$. Among other things, this will allow us to relate the abstract alternative \ref{K loss of fredholmness alternative} in Theorem~\ref{homoclinic global ift} to the velocity degeneracy alternative \eqref{velocity degeneracy alternative} in Theorem~\ref{intro global continuation theorem}. The main idea is to use the relative velocity
\begin{align}
  \label{U definition}
  U \colonequals  \frac{w_\zeta}{f_\zeta} + i\Omega\overline{f} - c
  = \frac{\mathcal{Z}^\rho \nu'}{1+\rho\mathcal{Z}^\rho \mu'} + i\Omega(\overline{\zeta+\rho^2 \mathcal{Z}^\rho \mu}) - c
\end{align}
as an intermediate variable so that the linearized equations become local.

\begin{lemma}\label{schauder lemma}
  Let $(u,\rho)=(\mu,\nu,Q,\lambda,\rho) \in \mathcal{O}_\delta$ with $\rho > 0$, and suppose that the function $U$ in \eqref{U definition} satisfies
  $\inf_{\partial \confD_{\rho}} |U| > 0$. Then the linearized operator $D_{(\mu,\nu)} \F(u,\rho)$ enjoys a Schauder-type estimate
  \begin{align}
    \label{schauder lemma estimate}
    \n{(\dot \mu,\dot \nu)}_{C^{\ell+\alpha}(\mathbb{T})^{2M}}
    \le C\Big(\big\|D_{(\mu,\nu)} \F(u,\rho) [(\dot \mu,\dot\nu)]\big\|_{C^{\ell-1+\alpha}(\mathbb{T})^{2M}}+
    \n{(\dot \mu,\dot \nu)}_{C^0(\mathbb{T})^{2M}}
    \Big),
  \end{align}
  where the constant $C$ depends only on lower bounds for $\rho$, $\delta$, and $\inf_{\partial \confD_{\rho}} |U|$, and upper bounds for $\rho, Q ,\lambda$, and $\n{(\mu,\nu)}_{C^{\ell+\alpha}}$.
  \begin{proof}
    Throughout the proof $\rho,Q,\lambda$ are fixed, and so we suppress dependence on them to simplify the notation. We also denote by $C$ a constant which depends only on $\rho$, $\delta$, $\n u_\Xspace$, and $\inf_{\partial \confD_{\rho}} |U|$, but whose value may change from line to line.
   
    As usual, we define $(f,w)$ by \eqref{f ansatz} and \eqref{w ansatz}, and $U$ by \eqref{U definition}. We then denote the associated traces by $f_k,w_k \in C^{\ell+\alpha}(\mathbb T,\mathbb C)$ and $U_k \in C^{\ell-1+\alpha}(\mathbb T,\mathbb C)$, i.e.~$f_k(\tau)=f(\zeta_k+\rho\tau)$ and so on, which allows us to write the nonlinear operator $\F=(\A,\B)$ as
    \begin{align*}
      \A_k(\mu,\nu) &= \rho \realpart[\tau f_k' U_k] \equalscolon  \widetilde\A_k(f_k,U_k),\\
      \B_k(\mu,\nu) &= \rho \abs{U_k}^2 - \frac 1\rho \Big(\frac{\gamma_k}{2\pi}\Big)^2 - Q_k \equalscolon  \widetilde\B_k(U_k).
    \end{align*}
    As the mappings $f_k \equalscolon  \mathscr Z_k(\mu)$ and $U_k \equalscolon  \mathscr U_k(\mu,\nu)$ are analytic, we can therefore analyze the inhomogeneous linearized equation
    \begin{align}
      \label{linearized equation}
      D_{(\mu,\nu)} \F(u,\rho) [(\dot \mu,\dot\nu)] = (\dot a, \dot b)  
    \end{align}
    using the chain rule. Indeed, after some manipulations we find that \eqref{linearized equation} is equivalent to the system
    \begin{align}
      \label{fdotbc}
      \rho \realpart[\tau U_k\dot f_k' ] &= \dot a_k - \rho \realpart[\tau f_k' \dot U_k], \\
      \label{Udotbc}
      2\rho \realpart[U_k \dot U_k] &= \dot b_k,
    \end{align}
    for $k=1,\ldots,M$, where here
    \begin{align}
      \label{linearizing fk}
      \begin{aligned}
        \dot f_k &\colonequals  D_\mu \mathscr Z_k(\mu) \dot \mu = \rho^2 \mathcal{Z}^\rho_k \dot \mu,\\
        \dot U_k &\colonequals  
        D_\mu \mathscr U_k(\mu,\nu) \dot \mu
        +
        D_\nu \mathscr U_k(\mu,\nu) \dot \nu
        = \frac{\mathcal{Z}^\rho_k \dot \nu'}{f_k'} - \frac{\mathcal{Z}^\rho_k \nu'}{(f_k')^2} \dot f_k' + i\Omega\overline{\dot f_k}.
      \end{aligned}
    \end{align}
    
    Our next goal is to interpret \eqref{fdotbc}--\eqref{Udotbc} as boundary conditions for a local problem. Using the definitions of $f$ and $U$ instance, in \eqref{f ansatz} and \eqref{U definition}, we easily observe that 
    \begin{align}
      \label{dbar equations}
      \partial_{\bar \zeta} f = 0,
      \qquad \partial_{\bar \zeta} U = i\Omega\overline{\partial_\zeta f}.
    \end{align}
    in $\confD_\rho$. Similarly using \eqref{linearizing fk} and the properties of the layer-potential operator $\mathcal{Z}^\rho_k$, we see that $\dot f_k, \dot U_k$ are traces of functions $\dot f \in C^{\ell+\alpha}(\overline{\confD_\rho},\mathbb C)$ and $\dot U \in C^{\ell-1+\alpha}(\overline{\confD_\rho},\mathbb C)$ satisfying 
    \begin{align}
      \label{linearized dbar dotf}
      \partial_{\bar \zeta} \dot f &= 0,\\
      \label{linearized dbar dotU}
      \partial_{\bar \zeta} \dot U &= i\Omega\overline{\partial_\zeta \dot f}.
    \end{align}
    in $\confD_\rho$. Unsurprisingly, \eqref{linearized dbar dotf}--\eqref{linearized dbar dotU} agree with the formal linearization of \eqref{dbar equations}. Applying Privalov's theorem to \eqref{linearizing fk} we deduce
    \begin{align}
      \label{fdot controlled by mudot}
      \n{\dot f}_{C^{1+\alpha/2}(\confD_\rho)} 
      \le C \n{\dot \mu}_{C^{1+\alpha/2}(\mathbb T)^M},
      \qquad 
      \n{\dot U}_{C^{\alpha/2}(\confD_\rho)} 
      \le C 
      \n{(\dot \mu,\dot \nu)}_{C^{1+\alpha/2}(\mathbb T)^{2M}},
    \end{align}
    while solving \eqref{linearizing fk} for $\mathcal{Z}_k^\rho \dot \mu$ and $\mathcal{Z}_k^\rho \dot \nu'$ and appealing to the basic estimate \eqref{layer potential schauder} yields
    \begin{align}
      \label{mudot controlled by fdot}
      \n{\dot \mu}_{C^{\ell+\alpha}(\mathbb T)^M}
      \le \n{\dot f}_{C^{\ell+\alpha}(\partial\confD_\rho)},
      \qquad 
      \n{\dot \nu}_{C^{\ell+\alpha}(\mathbb T)^{2M}}
      \le C\big(\n{\dot f}_{C^{\ell+\alpha}(\partial\confD_\rho)} 
      + \n{\dot U}_{C^{\ell-1+\alpha}(\partial\confD_\rho)}\big).
    \end{align}

    We now consider \eqref{Udotbc} and \eqref{linearized dbar dotU} as an elliptic problem for $\dot U$. By assumption the coefficients $U_k$ in \eqref{fdotbc} are bounded away from zero, and so standard elliptic theory (see for example \cite[Chapter 1, Section 2]{volpert2011book1}) yields a Schauder estimate
    \begin{align}
      \label{basic dotU estimate}
      \n{\dot U}_{C^{\ell-1+\alpha}(\confD_\rho)}
      &\le C \big(\n{\dot f}_{C^{\ell-1+\alpha}(\confD_\rho)}
      + \n{\dot b}_{C^{\ell-1+\alpha}(\mathbb T)^M}
      + \n{\dot U}_{C^0(\confD_\rho)}\big).
    \end{align}
    Arguing similarly for \eqref{fdotbc} and \eqref{linearized dbar dotf}, we find
    \begin{align}
      \notag
      \n{\dot f}_{C^{\ell+\alpha}(\confD_\rho)}
      &\le C \big(\n{\dot a}_{C^{\ell-1+\alpha}(\mathbb T)^M}
      + \n{\dot U}_{C^{\ell-1+\alpha}(\partial\confD_\rho)}
      + \n{\dot f}_{C^0(\confD_\rho)}\big)\\
      \label{basic dotf estimate}
      &\le C \big(\n{(\dot a, \dot b)}_{C^{\ell-1+\alpha}(\mathbb T)^{2M}}
      + \n{\dot U}_{C^0(\confD_\rho)}
      + \n{\dot f}_{C^0(\confD_\rho)}\big),
    \end{align}
    where in the second inequality we have used \eqref{basic dotU estimate} and interpolation. Combining \eqref{basic dotU estimate}--\eqref{basic dotf estimate} with \eqref{mudot controlled by fdot} and using \eqref{mudot controlled by fdot} to estimate the $C^0$ norms on the right hand side, we find
    \begin{align*}
      \n{(\dot \mu,\dot \nu)}_{C^{\ell+\alpha}(\mathbb T)^{2M}}
      &\le C \big(\n{(\dot a, \dot b)}_{C^{\ell-1+\alpha}(\mathbb T)^{2M}}
      + \n{(\dot \mu,\dot \nu)}_{C^{1+\alpha/2}(\mathbb T)^{2M}} \big).
    \end{align*}
    The desired inequality \eqref{schauder lemma estimate} now follows by interpolation.  
  \end{proof}
\end{lemma}

\subsection{Uniform regularity}

In this section we will establish uniform $C^{1+\alpha}$ control for the conformal mapping $f$. Anticipating their use in the proof of theorem Theorems~\ref{intro rotating theorem}, \ref{intro tripole theorem}, \ref{intro pocklington theorem}, and \ref{intro global continuation theorem}, we seek estimates depending only on upper bounds for $\Nconf$ and $\Nvel$ and the various parameters.  In particular, recall that the (non-normalized) Bernoulli constant $q_k$ is the magnitude of the relative velocity on the boundary component $\Gamma_k$, and hence controlled from above and below by $\Nvel$.  

As an important preliminary result, we first show that the definition of the set $\mathcal{O}$ in \eqref{definition O delta} ensures that solutions to the abstract operator equation do indeed give rise to globally conformal mappings.  

\begin{lemma}\label{lem nondeg f_zeta}
Let $(u,\rho)=(\mu,\nu,Q,\lambda,\rho)$ be given that is in the connected component of $\F^{-1}(0) \cap \mathcal{O}$ containing a point vortex solution $(u^0,0)$. Then $\partial_\zeta f \ne 0$ in $\overline{\confD_\rho}$ and 
  \begin{equation}
    \label{winding number condition}
    \frac{1}{2\pi i } \int_{\p \confD_\rho} \frac{\partial_\zeta^2 f }{\partial_\zeta f} \,d\zeta = 0.
  \end{equation}
\end{lemma}
\begin{proof}
% SW: removed the restriction that rho > 0 for convenience later
First note that if $\rho = 0$, then $f$ is simply the identity function due to \eqref{f ansatz}, and so the statements above hold trivially.  Assume now that $\rho \neq 0$.  By definition we know that $f_\zeta$ is holomorphic in $\confD_\rho$ and $f_\zeta \ne 0$ on $\p \confD_\rho$. From \eqref{f ansatz} and \eqref{definition Z operator} it follows that
\begin{equation*}
\partial_\zeta f  = O(1), \quad \partial_\zeta^2 f = O(1/\zeta^3) \quad \textrm{as } \ |\zeta| \to \infty.
\end{equation*}
Hence
\begin{equation*}
\lim_{r\to \infty} \left.\ar \partial_\zeta f \left(r e^{it}\right) \right|^{t = 2\pi}_{t = 0} = \lim_{r\to \infty} \frac{1}{2\pi i } \int_{|\zeta| = r} \frac{\partial_\zeta^2 f}{\partial_\zeta f} \,d\zeta = 0.
\end{equation*}

On the other hand, recall that $(u, \rho)$ lies in the connected component containing the point vortex configuration $(u^0, 0)$. From the asymptotic formula \eqref{informal asymptotics} we see that as $\rho \to 0$,
\[
\partial_\zeta f = 1 + O(\rho^2), \quad \partial_\zeta^2 f = O(\rho) \quad \text{on } \p B_\rho(\zeta_k).
\]
Therefore for $\rho$ sufficiently small we can estimate the winding number of the path $\partial_\zeta f(\p B_\rho(\zeta_k))$ around the origin as
\[
\frac{1}{2\pi i } \int_{\p B_\rho(\zeta_k)} \frac{\partial_\zeta^2 f }{\partial_\zeta f} \,d\zeta = O(\rho^2).
\]
Since the winding number takes integer values, the above integral must be zero.  By the continuity of the winding number it follows that for any $(u,\rho) \in \F^{-1}(0) \cap \mathcal{O}$ 
\[
\frac{1}{2\pi i } \int_{\p B_\rho(\zeta_k)} \frac{\partial_\zeta^2 f }{\partial_\zeta f} \,d\zeta = 0,
\]
which implies that \eqref{winding number condition} holds.  Finally, since $\p_\zeta f$ is bounded, it has no poles in $\confD_\rho$, and so we may conclude from the Cauchy's argument principle that $\partial_\zeta f \ne 0$ in $\confD_\rho$.
\end{proof}

Define the relative stream function to be 
\begin{equation*}
\Psi \colonequals  \imagpart{W} - m_1,
\end{equation*}
where $m_1$ is introduced in \eqref{intro kinematic condition}. Notice that from the physical point of view, changing the stream function by a constant does not change the velocity field. Here we use the above definition to normalize one of the boundary conditions for $\Psi$.

Thus the relative velocity $U = 2i \p_z \Psi$. We know that $\Psi$ satisfies
\begin{equation}
\label{Psi elliptic pde}
\left\{\begin{aligned}
\Delta \Psi & = 2 \Omega |f_\zeta|^2 \ & \textrm{ on } & \ \confD_\rho, \\
\Psi & = m_k - m_1 & \textrm{ on } & \ \p B_\rho(\zeta_k). 
\end{aligned}\right.
\end{equation}

Take $R = R(\zeta_1, \ldots, \zeta_M) >0$ sufficiently large so that 
\begin{equation}
  \label{big ball}
  \bigcup_{k=1}^M B_\rho(\zeta_k) \subset B_R(0) \equalscolon  B_R. 
\end{equation}
We would like to establish Schauder-type estimates on $\Psi$ in $\confD_\rho \cap B_R$. This would require bounds on the relative fluxes $m_k - m_1$ for $k=1, 2, \ldots, M$, as well as $\Psi|_{\p B_R}$, which amounts to asking that the relative velocity $U$ be bounded. 
\begin{lemma}\label{lem control on U}
Let $(u,\rho)=(\mu,\nu,Q,\lambda,\rho) \in \F^{-1}(0) \cap  \mathcal{O}_\delta$ with $\rho \ne 0$. Then  for any $R >0$ satisfying \eqref{big ball},
  \begin{equation}
    \label{bounds on U}
    \|U\|_{L^\infty({\confD_\rho \cap B_R})} \le C,
  \end{equation}
for a constant $C > 0$ depending on upper bounds for $|\rho|$, $|q|$, $|c|$, $|\Omega|$, and $\|f\|_{L^\infty({\confD_\rho \cap B_R})}$. 
\end{lemma}
\begin{proof}
From the definition \eqref{U definition} of $U$ we have 
\[
|U| \le \LV \frac{w_\zeta}{f_\zeta} \RV + |\Omega| |f| + |c|.
\]
Note from Lemma \ref{lem nondeg f_zeta} that $\partial_\zeta w/\partial_\zeta f$ is single-valued and holomorphic in $\confD_\rho$. Moreover the asymptotics of $\partial_\zeta f$ and $\partial_\zeta w$ indicates that 
\[
\frac{\partial_\zeta w}{\partial_\zeta f} = O(1/\zeta) \quad \text{as} \quad |\zeta| \to \infty.
\]
From maximum modulus theorem we see that
\[
\max_{\overline{\confD_\rho}} \LV \frac{\partial_\zeta w}{\partial_\zeta f} \RV = \max_k \max_{\p B_\rho(\zeta_k)} \LV \frac{\partial_\zeta w}{\partial_\zeta f} \RV.
\]
Applying \eqref{U definition} again and recalling the Bernoulli condition \eqref{bernoulli condition patch} yields
\[
\max_{\p B_\rho(\zeta_k)} \LV \frac{\partial_\zeta w}{\partial_\zeta f} \RV \le |q_k| + |\Omega| \max_{\p B_\rho(\zeta_k)} |f| + |c|.
\]
Therefore 
\[
\max_{\overline{\confD_\rho \cap B_R}} |U| \le \max_k |q_k| + 2 |\Omega| \max_{\overline{\confD_\rho \cap B_R}} |f| + 2 |c|,
\]
which completes the proof.
\end{proof}

We are now ready to establish the desired bound for $\|f\|_{C^{1+\alpha}}$.
\begin{lemma}\label{lem unif reg}
Let $(u,\rho)=(\mu,\nu,Q,\lambda,\rho) \in \F^{-1}(0) \cap \mathcal{O}_\delta$ with $\rho \ne 0$, and $q_k \ne 0$ for all $k = 1, \ldots, M$. Then for any $R >0$ satisfying \eqref{big ball} we have
  \begin{equation}
    \label{improved reg on f}
    \|f\|_{C^{1+\alpha}({\confD_\rho \cap B_R})} \le C,
  \end{equation}
for a constant $C > 0$ depending on lower bounds for $|\rho|$ and $\delta$, and upper bounds for $|\rho|, R, \Nvel$, $|c|$, $|\Omega|$, and $\|f\|_{C^1({\confD_\rho \cap B_R})}$. 
\end{lemma}
\begin{proof}
It is easy to see that any two points in $\overline{\confD_\rho \cap B_R}$ can be connected by a curve lying entirely within $\confD_\rho \cap B_R$ with length uniformly bounded, for example, by $2R + 2 \pi M \rho$. Therefore from the previous lemma one gets
\[
  \label{bounds on flux}
    \max_{1\le k\le M}|m_k - m_1| \le \max_{\overline{\confD_\rho \cap B_R}} |\Psi| \le C_1 
\]
for a constant $C_1 > 0$ depending only on a lower bound for $|\rho|$ and upper bounds for  $|\rho|, |q|, R$, and $\|f\|_{L^\infty(\overline{\confD_\rho \cap B_R})}$. Since $\Psi$ solves the elliptic PDE \eqref{Psi elliptic pde} on ${\confD_\rho \cap B_R}$, standard Schauder estimates yield
  \begin{equation}
    \label{Holder Psi}
    \|\Psi\|_{C^{1+\alpha}({\confD_\rho \cap B_R})} \le C \LC \|\partial_\zeta f\|^2_{L^\infty({\confD_\rho \cap B_R})} + \|\Psi\|_{L^\infty({\confD_\rho \cap B_R})} \RC \le C \LC \| \partial_\zeta f \|^2_{L^\infty({\confD_\rho \cap B_R})} + 1 \RC,
  \end{equation}
where the constant $C_2 > 0$ depends on lower bounds for $|\rho|$ and $\delta$, and upper bounds for $|\rho|, |q|, R$, and $\|f\|_{L^\infty({\confD_\rho \cap B_R})}$. 
In particular, this shows that $U \partial_\zeta f = 2i\partial_\zeta \Psi \in C^{\alpha}(\overline{\confD_\rho \cap B_R})$. Now, since $\rho, q_k \ne 0$, combining the Bernoulli condition \eqref{bernoulli condition patch} with \eqref{Holder Psi} gives
  \begin{equation}
    \label{bdrd reg on f-deriv}
    \| \partial_\zeta f\|_{C^\alpha(\partial B_\rho(\zeta_k))} = \frac{1}{q_k} \| \partial_\zeta \Psi\|_{C^\alpha(\partial B_\rho(\zeta_k)}  \leq C,
  \end{equation}
for a constant $C > 0$ that, as in the lemma statement, depends on lower bounds for $|\rho|$ and $\delta$, and upper bounds for $|\rho|, R, \Nvel$, and $\|f\|_{C^1({\confD_\rho \cap B_R})}$.  (Note that control of $\Nvel$ bounds $|q_k|$ and $1/|q_k|$ from above.)
From Lemma \ref{lem nondeg f_zeta} we know that $\log \partial_\zeta f$ is a well-defined single-valued holomorphic function in $\overline{\confD_\rho}$. By \eqref{bdrd reg on f-deriv} above, its real part is a holomorphic function in $\confD_\rho$ with Dirichlet data 
\[
\realpart{\log \partial_\zeta f} =  \log |\partial_\zeta f| = \log \frac{|\partial_\zeta \Psi|}{q_k} \in C^{\alpha}(\p B_\rho(\zeta_k)).
\]
Therefore, by elliptic regularity theory both $\realpart{\log \partial_\zeta f}$ and its harmonic conjugate $\imagpart{\log \partial_\zeta f}$ are $C^{\alpha}(\overline{\confD_\rho \cap B_R})$. Again, since $\partial_\zeta f \ne 0$, it follows that $\partial_\zeta f \in C^{\alpha}(\overline{\confD_\rho \cap B_R})$.  Appealing once more to Schauder theory and the above estimates, we conclude that
\[
\|\partial_\zeta f\|_{C^\alpha(\overline{\confD_\rho \cap B_R})} \le C (\|\partial_\zeta f\|_{L^\infty(\overline{\confD_\rho \cap B_R})} + \|\partial_\zeta f\|_{C^\alpha(\p \confD_\rho)}),
\]
where $C > 0$ depends on the same quantities as the lemma statement. The claimed bound \eqref{improved reg on f} now follows from \eqref{Holder Psi} and \eqref{bdrd reg on f-deriv}.
\end{proof}

Finally, we provide a variant of the above estimate that controls the relevant densities in terms of the quantity $\Nvel$, whose blowup characterizes conformal degeneracy, and $\| \partial_\zeta f \|_{L^\infty}$ which is controlled by the quantity $\Nconf$ that characterizes conformal degeneracy.

\begin{lemma}[Uniform regularity]
\label{uniform regularity lemma}
Let $(u,\rho)=(\mu,\nu,Q,\lambda,\rho) \in \F^{-1}(0) \cap \mathcal{O}_\delta$ be given with $\rho \neq 0$ and $q_k \ne 0$ for all $k = 1, \ldots, M$.  Then
\[
	\| \mu \|_{C^{\ell+\alpha}(\mathbb{T})^M} + \| \nu \|_{C^{\ell+\alpha}(\mathbb{T})^M} \leq C , 
\]
for a constant $C > 0$ depending on lower bounds for $|\rho|$ and $\delta$ and upper bounds for $|\rho|, |\lambda|, \Nvel$, and $\| \partial_\zeta f \|_{L^\infty(\partial \confD_\rho)}$.
\end{lemma}
\begin{proof}
Throughout the argument, we denote by $C > 0$ a generic constant depending on the quantities in the statement of the lemma.  Choosing $R \colonequals  2 \max\{ |\zeta_1|, \ldots, |\zeta_M| \} + 2|\rho|$ so that \eqref{big ball} holds, we may apply Lemma~\ref{lem unif reg} to infer the additional H\"older regularity of the corresponding conformal mapping $f$:
\[
	\partial_\zeta f(\zeta_k + \rho \placeholder) = \mathcal{Z}_k^\rho \mu^\prime \in C^\alpha(\mathbb{T}).
\]
Note that we cannot directly apply the estimate \eqref{improved reg on f} because it introduces a constant depending on $\| f \|_{L^\infty}$, which we do not yet control.  However, 
appealing to Lemma~\ref{layer potential bounds lemma} and Morrey's inequality leads to a uniform upper bound on the H\"older seminorm
\[
	[\mu_k]_{\alpha; \mathbb{T}} \leq C_\alpha \| \mu_k^\prime \|_{L^{\frac{1}{1-\alpha}}(\mathbb{T})} \leq C_\alpha \| \mathcal{Z}_k^\rho \mu^\prime \|_{L^{\frac{1}{1-\alpha}}(\mathbb{T})} \leq C_\alpha \| \partial_\zeta f \|_{L^\infty(\partial \confD_\rho)},
\]
where each $C_\alpha$ above is a positive constant depending only on $\alpha$. 
Now, because $P_0 \mu_k = 0$, we have that
\[
	 \mu_k(\tau) = \frac{1}{2\pi i} \int_{\mathbb{T}} \frac{\mu_k(\tau)}{\sigma} \, d\sigma =  \frac{1}{2\pi i} \int_{\mathbb{T}} \frac{ \mu_k(\tau) - \mu_k(\sigma)}{\sigma} \, d\sigma.
\]
Thus 
  \begin{equation}
    \label{first density holder bound}
    \| \mu \|_{C^\alpha(\mathbb{T})^M} \lesssim \sum_k [\mu_k]_{\alpha;\mathbb{T}} \leq C_\alpha \|\partial_\zeta f \|_{L^\infty(\partial \confD_\rho)}.
  \end{equation}

In particular, this gives a uniform $L^\infty$ bound on $\mu$, which in turn furnishes the estimate
\[
	\| f - \id \|_{L^\infty(\confD_\rho)} = \rho^2 \| \mathcal{Z}^\rho \mu \|_{L^\infty(\confD_\rho)} \leq C_1 \| \partial_\zeta f \|_{L^\infty(\partial \confD_\rho)}
\]
for a constant $C_1 > 0$ depending on a lower bound for $\delta$ and an upper bound for $|\rho|$.  As $f - \id$ is holomorphic and vanishes at infinity, the maximum modulus principle then implies that
\[
	\| f - \id \|_{C^1(\confD_\rho)}  \leq C_1 \left( \| \partial_\zeta f \|_{L^\infty(\partial \confD_\rho)} + 1 \right),
\]
where the constant $C_1 > 0$ depends on the same quantities as the previous line.  It follows that 
\[
	\| f \|_{C^1(\confD_\rho \cap B_R)} \leq C_1 \left( \| \partial_\zeta f \|_{L^\infty(\partial \confD_\rho)} + 1 \right).
\]
By the above reasoning, the $C^{1+\alpha}$ bound \eqref{improved reg on f} holds for a constant $C_2 > 0$ that depends now on lower bounds for $|\rho|$ and $\delta$ and upper bounds for $|\rho|, \zeta_1, \ldots, \zeta_M, \Nvel$, and $\| \partial_\zeta f \|_{L^\infty(\partial \confD_\rho)}$.  Finally, Lemma~\ref{layer potential bounds lemma} translates this to uniform control of $\mu^\prime$ in $C^\alpha(\mathbb{T})^M$, which combined with \eqref{first density holder bound} gives  
\[
	\| \mu \|_{C^{1+\alpha}} \leq C_2,
\]
for a constant $C_2 > 0$ with the same dependencies.  On the other hand, rearranging terms from the kinematic condition the kinematic condition \eqref{kinematic condition} leads to
\[
	\realpart{\left( \tau \mathcal{Z}_k^\rho \nu^\prime \right)} = -\realpart{\left( \tau \left( \Vrho_k^\rho + ( i \Omega \overline{( \zeta_k + \rho \tau)} -c) \rho \mathcal{Z}_k^\rho \mu^\prime + i\Omega \rho^2 \overline{\mathcal{Z}_k^\rho \mu}  + i \Omega \rho^3 \mathcal{Z}_k^\rho [\mu^\prime] \overline{\mathcal{Z}_k^\rho\mu} \right) \right)}  \equalscolon  g_k.
\]
The bounds on $\mu$ ensure that $\| g_k \|_{C^\alpha} < C$ (note that the added dependence on $\lambda$ derives from the presence of $c$, $\Omega$, and the circulations).    Since $\proj_0 \nu = 0$, we may therefore conclude from Lemma~\ref{layer potential bounds lemma} that 
\[
	\| \mu \|_{C^{1+\alpha}(\mathbb{T})^M} + \| \nu \|_{C^{1+\alpha}(\mathbb{T})^M} \leq C.
\]

Finally, to upgrade this to control of the $C^{\ell+\alpha}$ norms of $\mu$ and $\nu$, we use a standard bootstrapping scheme based on the linear Schauder-type estimates from Lemma~\ref{schauder lemma}.  Indeed, differentiating the equation $\F(u,\rho) = 0$ in $\tau$ yields
\[
	D_{(\mu,\nu)} \F(u,\rho)[ (\mu^\prime, \nu^\prime)] = (a,b)
\]
for explicit functions $a, b \in C^{\infty}(\mathbb{T})^M$ arising from the forcing terms in the kinematic and Bernoulli conditions.  Applying \eqref{schauder lemma estimate}, we then conclude that
\[
	\| \mu^\prime \|_{C^{1+\alpha}(\mathbb{T})^M} + \| \nu^\prime \|_{C^{1+\alpha}(\mathbb{T})^M}  \leq C \left( 1+ \| \mu^\prime \|_{C^{\alpha}(\mathbb{T})^M} + \| \nu^\prime \|_{C^{\alpha}(\mathbb{T})^M} \right) \leq C.
\]
The lemma now follows by iteration.
\end{proof}

\subsection{Proof of the general continuation theorem}

With the uniform estimates derived in the previous subsection now in hand, we are now able to prove the result on continuation of hollow vortices in the general setting. 

\begin{proof}[Proof of Theorem~\ref{intro global continuation theorem}] 
We established in Theorem~\ref{hollow vortex local bifurcation theorem} the existence of a real-analytic local curve $\km_\loc$ of solutions to the hollow vortex problem bifurcating from a given non-degenerate point vortex configuration.  Applying Corollary~\ref{global ift with identities corollary}, we can extend $\km_\loc$ to a curve $\km$ that admits a $C^0$ parameterization
\[
	\km = \{ (u(s), \rho(s)) : s \in \mathbb{R} \} \subset \F^{-1}(0) \cap \mathcal{O},
\]
where $u(s) = (\mu(s), \nu(s), Q(s), \lambda(s)) \in \Xspace$.  Let $\confD(s)$ denote the conformal domain at the parameter value $s$.  In light of Lemma~\ref{lem nondeg f_zeta}, at each $s \in \mathbb{R}$, the corresponding mapping $f(s)$ is conformal on $\overline{\confD(s)}$ and satisfies the winding number condition \eqref{winding number condition}.   The local real analyticity of $\km$ follows from Corollary~\ref{global ift with identities corollary}~\ref{K reparam with identities}.  As a consequence of Corollary~\ref{global ift with identities corollary}~\ref{K alternatives with identities}, moreover, we know that in the limits $s \to \infty$ and $s \to -\infty$, one of the alternatives \ref{K blowup alternative}, \ref{K loss of compactness alternative}, \ref{K loss of fredholmness alternative}, or \ref{K loop} occurs.  In particular, the quantity 
  \begin{equation}
    \label{general N quantity definition}
    N(s) \colonequals  \| \mu(s)  \|_{C^{\ell+\alpha}(\mathbb{T})^M} + \| \nu(s)  \|_{C^{\ell+\alpha}(\mathbb{T})^M} + |Q(s)|  + |\rho(s)| + |\lambda(s)| + \frac 1{\dist((u(s),\rho(s)), \, \dell \mathcal{O})}
  \end{equation}
is finite for all $s \in \mathbb{R}$.  The definition of $\mathcal{O}$ therefore ensures that $\confD(s)$ is a circular domain in that its boundary is composed of $M$ non-intersecting circles. To simplify notation, let $\km_+$ denote the portion of $\km$ corresponding to $s \geq 0$.  Without loss of generality, we can take $\rho(s) > 0$ on $( \km_+ \cap \km_\loc) \setminus \{ (u^0,0) \}$. In particular, this implies the existence of $s_1 > 0$ such that $ \rho(s) > 0$ on $(0,s_1]$. 

Seeking a contradiction, suppose now that 
  \begin{equation}
    \label{global additional assumption}
    \sup_{s \geq s_1} \left( \Nconf(s) + \Nvel(s) + |\param(s)| \right) < \infty,
  \end{equation}
where recall $\Nconf$ and $\Nvel$ were defined in \eqref{definition conformal blowup norm} and \eqref{definition velocity blowup norm}, respectively.  We will prove that this would exclude \emph{all} of the alternatives \ref{K blowup alternative}--\ref{K loop}, which is impossible.  First, observe that because $\partial_\zeta f(s)$ is non-vanishing on $\overline{\confD(s)}$ and limits to $1$ as $\zeta \to \infty$, the maximum modulus principle implies
\[
	\sup_{s \geq s_1} \sup_{\confD(s)} \left(  |\partial_\zeta f(s)|  + \frac{1}{|\partial_\zeta f(s)|} \right) \leq \sup_{s \geq s_1} \Nconf(s) < \infty. 
\]
Thus distances in the conformal domain $\confD(s)$ and physical domain $\fluidD(s) \colonequals  f(s)(\confD(s))$ are uniformly comparable.  The uniform boundedness of $\Nconf(s)$ further guarantees that the boundary component $\Gamma_k(s) \colonequals  f(s)(\partial B_{\rho(s)}(\zeta_k(s)))$ are $C^1$ Jordan curves, and a simple continuity argument then shows that $\Gamma_1(s), \ldots, \Gamma_M(s)$ are uniformly separated with no curve enclosing another.  By Lemma~\ref{lem basic injective}, each mapping $f(s)$ is therefore globally injective, hence the solutions on $\km_+$ are physical.  Observe, moreover, that $\Nconf(s)$ controls the distance between the physical vortex boundaries $\Gamma_1(s), \ldots, \Gamma_M(s)$, which is comparable to the distances between the boundary components in $\confD(s)$.  Thus there exists $\delta > 0$ such that $\km_+ \subset \F^{-1}(0) \cap \mathcal{O}_\delta$. 

By assumption, one of the circulations, say $\gamma_j$, is fixed along $\km_+$.  Then, because the boundary trace of the relative velocity field $U(s)$ is purely tangential, we have that
\begin{align*}
	q_j(s) & =  \frac{1}{|\Gamma_j(s)|} \left| \int_{\Gamma_j(s)} U(s) \, dz \right|  
		 = \frac{1}{|\Gamma_j(s)|} \left|  \int_{\Gamma_j(s)} \partial_z w(s) \, dz + i \Omega \int_{\Gamma_j(s)} \overline{z}  \, dz - \int_{\Gamma_j(s)} c \, dz  \right| \\
		 & = \frac{1}{|\Gamma_j(s)|} \left|  \gamma_j - 2 \Omega A_j(s) \right|,
\end{align*}
where the second line follows from the complex Green's formula and $A_j(s)$ is the area of the region enclosed by $\Gamma_j(s)$.  The isoperimetric inequality then gives the bound
  \begin{equation}
    \label{general rho lower bound}
    q_j(s) \geq \frac{|\gamma_j|}{|\Gamma_j(s)|} - 2 |\Omega| \frac{A_j(s)}{|\Gamma_j(s)|} \geq \frac{|\gamma_j|}{|\rho(s) |} \Nconf(s) - |\Omega| |\rho(s)| \Nconf(s).
  \end{equation}
As $q_j(s) \leq \Nvel(s)$, the above inequality provides a uniform lower bound on $|\rho(s)|$ along $\km_+$.  Indeed, because $\rho(s_1) > 0$, we have that $\rho(s)$ is uniformly positive for all $s \geq s_1$.  On the other hand, thanks to \eqref{global additional assumption}, the magnitude the vortex centers $\zeta_1(s), \ldots, \zeta_M(s)$ are uniformly bounded, which implies a uniform upper bound on $\rho(s)$.

Next, we claim that \eqref{global additional assumption} forces a uniform bound on $N(s)$ ruling out the blowup alternative \ref{K blowup alternative}.  The above paragraph shows the last three terms on the right-hand side of \eqref{general N quantity definition} are controlled uniformly.   Likewise, from \eqref{q_k Q_k relation}, we have the upper and lower bounds on $Q_k(s)$:
\[
	-\frac{\gamma_k(s)^2}{4\pi^2 \rho(s)} < Q_k(s) < \rho(s)  q_k(s)^2 \leq \rho(s) \Nvel(s)^2,
\] 
both of which are uniformly controlled according to \eqref{global additional assumption}.  Lemma~\ref{uniform regularity lemma} then gives uniform bounds on $\| \mu(s) \|_{C^{\ell+\alpha}}$ and $\| \nu(s) \|_{C^{\ell+\alpha}}$, and hence we have succeeded in proving that \eqref{global additional assumption} implies $\sup_{s \geq s_1} N(s) < \infty$.  

However, \eqref{global additional assumption}, Lemma~\ref{schauder lemma}, and Lemma~\ref{uniform regularity lemma} also exclude the loss of compactness alternative \ref{K loss of compactness alternative} and loss of Fredholmness alternative \ref{K loss of fredholmness alternative}.  Lastly, were the closed loop alternative~\ref{K loop} to occur, then $\rho(T) = 0$ for some $T > s_1$, which contradicts the positive lower bound on $\rho(s)$ we derived above.  Thus \eqref{global additional assumption} precludes all of the alternatives, meaning it cannot hold.

In total, then, we have proven that there exists $s_2 \in (s_1, \infty]$ such that 
  \begin{equation}
    \label{blowup along km+}
    \limsup_{s \to s_2} \left( \Nvel(s) + \Nconf(s) + |\lambda(s)| \right) = \infty.
  \end{equation}
Without loss of generality, take $s_2$ to be the smallest parameter value in this interval for which \eqref{blowup along km+} holds.  Now, let $\cm$ be the curve in the original variables $(f,w,Q,\lambda,\rho)$ corresponding to the portion of $\km_+$ for which $s \in [0,s_2)$.  One can then reparameterize $\cm$ as claimed in \eqref{global curve parameterization}.  By construction, each of the solutions on $\cm$ is physical, and the curve inherits the local real analyticity of $\km_+$.  Finally,  \eqref{blowup along km+} ensures that in the limit along $\cm$, either conformal degeneracy, velocity degeneracy, or blowup of $|\lambda(s)|$ occurs.  The proof is therefore complete. 
\end{proof}

\subsection{Proofs for the examples with symmetry}

We now prove Theorems~\ref{intro rotating theorem}--\ref{intro pocklington theorem}.  The main task in each case is to further refine the limiting behavior \eqref{general continuation blowup alternative} given by the general theory in Theorem~\ref{intro global continuation theorem} using the symmetry of solutions and additional a priori bounds on the wave speed, angular momentum, or circulations.  
We begin with the construction of the large hollow vortex tripoles, which is the simplest of the three.  

\begin{proof}[Proof of Theorem~\ref{intro tripole theorem}]
  Arguing as in the proof of Corollary~\ref{small stationary tripole corollary}, Lemma~\ref{nonlinear symmetries lemma}\ref{rotating and tripole symmetries part} allows us to formulate the problem in terms of an analytic operator $\F = (\F_1, \F_2) \maps \mathcal{O}_{\mathrm{s}} \to \Yspace_{\mathrm{s}}$ where $\mathcal O_{\mathrm s}$ is an open subset of $\Xspace_{\mathrm{s}}$. Following the proof of Theorem~\ref{intro global continuation theorem}, but with Theorem~\ref{hollow vortex local bifurcation theorem} replaced by Corollary~\ref{small stationary tripole corollary} and Corollary~\ref{global ift with identities corollary} replaced by Theorem~\ref{homoclinic global ift}, we can then construct a continuous global curve $\km_{\mathrm s}$ of physical solutions with $\rho(s) > 0$ and the limiting behavior
  \begin{align}
    \label{tripole basic limit}
    \limsup_{s \to \infty} \left( \Nconf(s) + \Nvel(s) + |\gamma_2(s)| \right) = \infty
  \end{align}
  as the parameter $s \to \infty$. Note that here we have used the fact that the sole varying parameter for $\km_{\mathrm{s},\loc}$ is $\lambda = \gamma_2$.  By the choice of spaces and Lemma~\ref{nonlinear symmetries lemma}, all of the solutions on $\km_{\mathrm{s}}$ exhibit the symmetry properties claimed in part~\ref{intro stationary symmetry part}.
  
  It remains to show that the third term in \eqref{tripole basic limit} can be eliminated. To this end, suppose for the sake of contradiction that $\limsup_{s\to\infty} (\Nconf(s)+\Nvel(s)) < \infty$. Arguing as in the proof of Theorem~\ref{intro global continuation theorem}, the fact that the conformal vortex centers $\zeta_1=1$, $\zeta_2=0$, and $\zeta_3=-1$ are fixed along $\cm_{\mathrm s}$ implies the basic upper bound $\rho(s) < 1/2$. But then
  \begin{align*}
    \abs{\gamma_2(s)} &= \left|\int_{\Gamma_2(s)} U(s)\, dz\right|
    \le 2\Nvel(s) \abs{\Gamma_2}
    \le 2\pi \rho(s) \Nconf(s) \Nvel(s)
  \end{align*}
  is also uniformly bounded as $s \to \infty$, contradicting \eqref{tripole basic limit}.
\end{proof}

Now, consider the co-translating hollow vortex family.  From Theorem~\ref{intro global continuation theorem}, blowup of the wave speed is one possibility as we follow the global bifurcation curve.  The next result shows that this cannot happen without an accompanying blowup up the relative velocity or degeneracy of the vortex boundary regularity.  It provides an estimate that is in fact of some independent interest as it applies not just to the Pocklington case but to any collection of translating hollow vortices. 

\begin{proposition}[Wave speed bound] \label{wave speed proposition} 
  For any translating hollow vortex configuration, the wave speed $c$ satisfies
  \begin{align*}
    |c| \le M\sup_{\p \fluidD} |U| \le M\Nvel.
  \end{align*}
  \begin{proof}
    By Theorem~\ref{koebe theorem}, we can assume that the hollow vortex configuration is described as in Section~\ref{conformal map complex potential section}, except that the conformal domain $\confD_\rho$ is replaced by a more general circular domain
    \[
      \confD_{\rho_1,\ldots, \rho_M}(\zeta_1, \ldots, \zeta_M) \colonequals  \mathbb{C} \setminus \overline{B_{\rho_1}(\zeta_{1}) \cup \cdots \cup B_{\rho_M}(\zeta_{M})}.
    \]
    Choosing $k$ so that $\rho_k = \max\{\rho_1,\ldots,\rho_M\}$, the calculus of residues gives
    \begin{align*}
      \frac{1}{2\pi i} \int_{\partial \confD_{\rho_1,\ldots,\rho_M}}\frac{U(\zeta)}{\zeta-\zeta_k} \, d\zeta 
      = \frac{1}{2\pi i} \int_{\partial \confD_{\rho_1,\ldots,\rho_M}}\frac{\partial_\zeta w(\zeta)/\partial_\zeta f(\zeta) - c}{\zeta-\zeta_k} \, d\zeta 
      = c,
    \end{align*}
    where here we have used that $\partial_\zeta w/\partial_\zeta f$ is holomorphic and $O(1/\zeta)$ as $\zeta \to \infty$. Estimating the integral we conclude that
    \begin{align*}
      |c|
      \le 
      \frac{1}{2\pi} (2\pi \rho_1 + \cdots + 2\pi \rho_M) \frac 1{\rho_k} \sup_{\partial\fluidD} |U|
      \le M\sup_{\partial\fluidD} |U|
    \end{align*}
    as desired.
  \end{proof}
\end{proposition}
  
\begin{proof}[Proof of Theorem~\ref{intro pocklington theorem}]
  Most of the arguments are very similar to the proof of Theorem~\ref{intro tripole theorem}, except that now we follow the proof of Corollary~\ref{small translating pair corollary} rather than Corollary~\ref{small stationary tripole corollary} to formulate the nonlinear problem and verify some of the hypotheses of Theorem~\ref{homoclinic global ift}. The analogue of \eqref{general continuation blowup alternative} in the Pocklington vortex setting is 
  \begin{align}
    \label{pocklington basic limit}
    \limsup_{s \to \infty} \left( \Nconf(s) + \Nvel(s) + |c(s)| \right) = \infty.
  \end{align}
  Again, the choice of spaces and Lemma~\ref{nonlinear symmetries lemma} guarantee the desired symmetry properties hold for the solutions along the global curve.  The final term in \eqref{pocklington basic limit} can then be eliminated by appealing to Proposition~\ref{wave speed proposition}.  
\end{proof}

In preparation for the proof of Theorem~\ref{intro rotating theorem}, we first derive a useful identity for rotating hollow vortex configurations whose conformal description satisfies $\sum_k \gamma_k \zeta_k = 0$.  Note that this holds for steady rotating point vortices as a consequence of \eqref{pv translation identity}, and hence it will be true for hollow vortices constructed via bifurcation theoretic arguments that fix the vortex centers and circulations.  

\begin{proposition}[Momentum identity]
\label{momentum identity proposition}
For any rotating hollow vortex configuration with conformal description $(f,w,Q,\fullparam,\rho)$ satisfying $\sum_k \gamma_k {\zeta}_k = 0$, it holds that
  \begin{equation}
    \label{more id rot}
    L(f,w) - \Omega I(f)  + \frac{ |\fluidD^c|   }{2 \pi }\sum_k \gamma_k  +  \frac{1}{4\Omega} \imagpart{ \int_{\partial \fluidD} z U^2 \, dz} - \frac{1} {8 \pi \Omega } \left( \sum_k \gamma_k \right)^2 = 0,
  \end{equation}
where $L$ is the excess angular momentum \eqref{definition angular momentum} and $I(f) \colonequals  \int_{\fluidD^c} |z|^2 \, dz$ is the moment of inertia for the vacuum region $\fluidD^c$.
\end{proposition}
\begin{proof}
For $R \gg 1$, consider the cut-off fluid region $\fluidD^R \colonequals  \fluidD \cap B_R$.  First, observe that
  \begin{equation}
    \label{point vortex angular momentum}
    \int_{\fluidD^R} z \partial_z w^0 \, dz = \frac{1}{2\pi i} \sum_k \gamma_k \int_{\fluidD^R} \frac{z}{z-\zeta_k} \, dz = \frac{ | \fluidD^R| }{2\pi i} \sum_k \gamma_k = \frac{R^2}{2 i} \sum_k \gamma_k - \frac{|\fluidD^c|}{2 \pi i } \sum_k \gamma_k.
  \end{equation}
Now, the heart of the proof is to evaluate the integral
\begin{equation}\label{rot id}
\begin{split}
\int_{\p \fluidD^R} U^2 z \,dz & = \int_{\p \fluidD^R} \LC (\partial_z w)^2 + 2i\Omega \overline{z} \partial_z w  - \Omega^2 \overline{z}^2 \RC z \,dz
\end{split}
\end{equation}
two ways.  On the one hand, using the complex Green's theorem, we see that
\begin{equation}
\label{identity integrate Green way}
  \begin{aligned}
    \int_{\p \fluidD^R} U^2 z \,dz & = 4 i \int_{\fluidD^R} U (\partial_{\bar z} U) z \, dz = -4 \Omega \int_{\fluidD^R} z w_z \, dz - 4 i \Omega^2 \int_{\fluidD^R} |z|^2 \, dz \\
    &=  -4 \Omega \int_{\fluidD^R} z \partial_z w \, dz - 2i \pi \Omega^2 R^4 + 4 i \Omega^2 I(f).
  \end{aligned}
\end{equation}
Next, we evaluate \eqref{rot id} by directly computing the boundary integrals.  The integral over $\partial \fluidD$ is bounded and therefore can be left alone.  On the other component of $\partial \fluidD^R$, we find that
\begin{align*}
	\int_{\p B_R} z (\partial_z w)^2   \,dz & = -\frac{1}{2\pi i} \LC \sum_k  \gamma_k \RC^2 + O\left( \frac{1}{R} \right) \\
	2 i \Omega \int_{\partial B_R}  |z|^2 \partial_z w \, dz & = 2i \Omega R^2 \sum_k \gamma_k = -4 \Omega \left( \int_{\fluidD^R} z \partial_z w^0 \, dz + \frac{|\fluidD^c|}{2\pi i} \sum_k \gamma_k \right) \\
	-\Omega^2 \int_{\partial B_R} \overline{z} |z|^2 \, dz & = -\Omega^2 R^2 \int_{\partial B_R} \overline{z} \, dz = -2i \pi \Omega^2  R^4,
\end{align*}
where the first line follows from the asymptotics of $f$ and $w$ as $\zeta \to \infty$ and in the second we used \eqref{point vortex angular momentum}. Finally, combining the above calculations with \eqref{identity integrate Green way}, we obtain 
\[
	-4 \Omega \int_{\fluidD^R} z \partial_z (w-w^0)  \, dz + 4 i \Omega^2 I(f) = \int_{\partial \fluidD} U^2 z \, dz  -\frac{1}{2\pi i} \LC \sum_k  \gamma_k \RC^2  - \frac{2 \Omega |\fluidD^c|   }{\pi i}\sum_k \gamma_k + O\left( \frac{1}{R} \right),
\]
which becomes \eqref{more id rot} upon taking the imaginary part and sending $R \to \infty$.  
\end{proof}

\begin{proof}[Proof of Theorem~\ref{intro rotating theorem}]
Similar to the proof of Theorem \ref{intro tripole theorem} and Theorem~\ref{intro pocklington theorem}, now we appeal to Corollary~\ref{local rotating pair corollary} to formulate the nonlinear problem and establish the existence of the local solution curve. Then a use of Theorem~\ref{homoclinic global ift} and an argument in the spirit of Theorem~\ref{intro global continuation theorem} allows us to continue the local curve to a global one with $\rho(s) > 0$ and the limiting behavior
  \begin{align}
    \label{rotating basic limit}
    \limsup_{s \to \infty} \left( \Nconf(s) + \Nvel(s) + |\Omega(s)| \right) = \infty.
  \end{align}
 Assume now that $\Nconf(s)$ and $\Nvel(s)$ are uniformly bounded, and hence $\Omega(s)$ is unbounded.  Thus the basic estimate 
  \begin{align*}
    \left|\int_{\p \fluidD(s)} U^2(s) z\, dz\right|
    \lesssim \rho(s) \Nconf(s)^2 \Nvel(s)^2.
  \end{align*} 
  and the angular momentum identity \eqref{more id rot} together give
  \[
  	 {L(s)} = \Omega(s) I(s) -  \frac{|\fluidD^c(s)|}{\pi} + O\left( \frac{1}{|\Omega(s)|} \right),
  \]
  where $L(s) \colonequals  L(f(s), w(s))$, and likewise for $I(s)$ and $\fluidD^c(s)$.  Since $I(s)$ and $|\fluidD^c(s)|$ are controlled uniformly by assumption, the excess angular momentum must exhibit the claimed blowup \eqref{angular momentum alternative}.  
\end{proof}

\section*{Acknowledgments} 

The research of RMC is supported in part by the NSF through DMS-1907584 and DMS-2205910.  The research of SW is supported in part by the NSF through DMS-1812436 and the Simons Foundation through award 960210.

\appendix

\section{From hollow vortex to point vortex identities}
\label{identities appendix}

In this appendix, we demonstrate how the classical point vortex identities \eqref{pv identities} can be recovered formally from the hollow vortices identities \eqref{id transl}, \eqref{id rot}, and \eqref{id stationary1-2} obtained in Section~\ref{hollow vortex identities section}.  Indeed, from \eqref{definition Bernoulli operator} and \eqref{abstract operator} we see that 
\[
	\Func_1 = \rho \mathscr{A}_k(u, \rho), 
	\quad 
	 \Func_2 = \frac{\mathcal{B}_k^\rho}{\rho} + \frac{\gamma_k^2}{4 \pi^2 \rho^2} - q_k^2 \qquad \textrm{on } \p B_\rho(\zeta_k),
\]
and hence 
\begin{equation}
\label{appendix F1 F2 identity}
\int_{\p B_\rho(\zeta_k)} \frac{\overline{w_\zeta} \Func_1}{(\zeta - \zeta_k) \overline{f_\zeta}} \,d\zeta  
= \int_{\mathbb{T}} \frac{\overline{w_\zeta (\zeta - \zeta_k)} \mathscr{A}_k}{\overline{f_\zeta}} \,d\tau,
\quad
\frac12 \int_{\p B_\rho(\zeta_k)} \Func_2 f_\zeta \,d\zeta
= \frac12\int_{\mathbb{T}} \mathcal{B}_k^\rho f_\zeta \,d\tau.
\end{equation}
From the ansatz \eqref{f ansatz} for $f$  and the definition of the leading-order part of the Bernoulli operator \eqref{linear Bernoulli operator}, we have
\begin{align*}
\frac12\int_{\mathbb{T}} \mathcal{B}_k^\rho f_\zeta \,d\tau = \frac{\gamma_k^2}{4\pi^2} \int_{\mathbb{T}} \realpart{\left(  \frac{2\pi i \tau}{\gamma_k} \left( \mathcal{V}_k(\fullparam) + \mathcal{C}\nu^\prime \right) - \mathcal{C} \mu^\prime\right)}\,d\tau + O(\rho).
\end{align*}
On the other hand, \eqref{f ansatz}, \eqref{w ansatz}, and the definition of $\A$ in \eqref{abstract operator} give
\begin{align*}
\int_{\mathbb{T}} \frac{\overline{w_\zeta (\zeta - \zeta_k)} \mathscr{A}_k}{\overline{f_\zeta}} \,d\tau &= \int_{\mathbb{T}} \overline{w^0_\zeta \rho\tau} \realpart{\LC \tau \LC \mathcal{Z}_k^\rho \nu^\prime + \Vrho_k^\rho \RC \RC} \,d\tau  + O(\rho) \\
	& =  \int_{\mathbb{T}} \overline{\LC \frac{\gamma_k}{2\pi i} \RC} \realpart{\LC \tau \LC \mathcal{C}\nu^\prime +\mathcal{V}_k(\fullparam) \RC \RC} \,d\tau  + O(\rho),
\end{align*}
where in the second line we have used the asymptotic form of $W_\zeta^0$ in \eqref{dWdzeta point vortex} and the definition of $\mathcal{V}_k^\rho$ \eqref{definition Lambda}.  Evaluating \eqref{appendix F1 F2 identity} using the last two calculations and formally taking the limit as $(\mu,\nu,\rho)$ approach the point vortex configuration $(0,0,0)$, we find that 
\begin{equation}\label{id pv 1}
\frac12 \int_{\p \confD_\rho} \Func_2 f_\zeta \,d\zeta - \int_{\p \confD_\rho} \frac{\overline{w_\zeta} \Func_1}{(\zeta - \zeta_k) \overline{f_\zeta}} \,d\zeta \longrightarrow  -\sum_k \gamma_k \overline{\mathcal{V}_k(\fullparam)}.
\end{equation}
Note that the minus sign above is due to the orientation of $\partial\confD_\rho$.  We proved in \eqref{id transl} and \eqref{id stationary1-2} that for translating hollow vortex configurations, the left-hand side of \eqref{id pv 1} is purely real while it vanishes identically for stationary hollow vortices.  The limit \eqref{id pv 1} then gives $\imagpart{ \sum_k \gamma_k \mathcal{V}_k} = 0$ and $\sum_k \gamma_k \mathcal{V}_k = 0$ for the translating and stationary case, respectively.  This recovers the point vortex identity \eqref{pv translation identity} in the corresponding regime.  

A similar computation can be performed on the hollow vortex identity \eqref{id rot} for rotating configurations and the second identity in \eqref{id stationary1-2} for the stationary case.  Arguing as in \eqref{appendix F1 F2 identity} leads to
\begin{align*}
\frac12 \realpart \int_{\p B_\rho(\zeta_k)} \Func_2 \overline f f_\zeta \,d\zeta 
 = \frac{\gamma_k^2}{4\pi^2} \realpart \int_{\mathbb{T}} \realpart{\left( \frac{2\pi i \tau}{\gamma_k} \left( \mathcal{V}_k(\fullparam) + \mathcal{C}\nu^\prime \right)  - \mathcal{C} \mu^\prime \right)} \overline{\zeta_k} \,d\tau + O(\rho),
\end{align*}
and
\begin{align*}
\realpart \int_{\p B_\rho(\zeta_k)} \frac{\overline{w_\zeta f} \Func_1}{(\zeta - \zeta_k) \overline{f_\zeta}} \,d\zeta 
& = \realpart \int_{\mathbb{T}} \overline{\LC \frac{\gamma_k \zeta_k}{2\pi i} \RC} \realpart{\LC \tau \LC \mathcal{C}\nu^\prime +\mathcal{V}_k(\fullparam) \RC \RC} \,d\tau  + O(\rho).
\end{align*}
In the point vortex limit, together these yield
\begin{equation}\label{id pv 2}
\frac12 \realpart \int_{\p \confD_\rho} \Func_2 \overline f f_\zeta \,d\zeta - \realpart \int_{\p \confD_\rho} \frac{\overline{w_\zeta f} \Func_1}{(\zeta - \zeta_k) \overline{f_\zeta}} \,d\zeta \longrightarrow -\realpart \LC \sum_k \gamma_k \zeta_k \mathcal{V}_k(\fullparam) \RC.
\end{equation}
Now, for rotating or stationary hollow vortex configuration, \eqref{id rot} and \eqref{id stationary1-2} state that the left-hand side above vanishes. Thus from \eqref{id pv 2} we conclude $\realpart{\sum_k \gamma_k \zeta_k \mathcal{V}_k} = 0$, which recovers the real part of the point vortex identity \eqref{pv rotation identity} in the corresponding regime. 

\bibliographystyle{siam}
\bibliography{projectdescription}

\end{document}